\documentclass[10pt]{amsart}

\usepackage[hmargin={2.1cm,2.1cm},vmargin={2.5cm,2.5cm}]{geometry}

\usepackage{amsmath, amsthm, amssymb}
\usepackage{mathtools}
\usepackage{enumerate}
\usepackage{float}
\usepackage{epsfig}
\usepackage{nicefrac}
\usepackage{cancel}

\usepackage{flafter}

\usepackage{graphicx}
\usepackage{caption}
\usepackage{subcaption}
\usepackage{stackengine}

\DeclareMathAlphabet{\mathpzc}{OT1}{pzc}{m}{it}
\usepackage{mathrsfs}

\usepackage{stmaryrd}

\usepackage{enumitem}

\theoremstyle{definition}

\newtheorem{theorem}{Theorem}[section]
\newtheorem{proposition}{Proposition}[section]

\newtheorem{remark}{Remark}[section]

\newtheorem{lemma}{Lemma}[section]

\usepackage{multirow}
\usepackage{cancel}
\usepackage{bm}
\usepackage[Symbol]{upgreek}

\newcommand{\Ld}{\mathbf{L}}
\newcommand{\Hd}{\mathbf{H}}

\newcommand{\hone}{H^1(\Omega)}
\newcommand{\hones}{H^1}
\newcommand{\hzero}{H_0^1(\Omega)}
\newcommand{\hzeros}{H_0^1}

\newcommand{\honed}{\Hd^1(\Omega)}
\newcommand{\honeds}{\Hd^1}

\newcommand{\hzerod}{\Hd_0^1(\Omega)}
\newcommand{\hzerods}{\Hd_0^1}

\newcommand{\htwod}{\mathbf{H}^2(\Omega)}

\newcommand{\hcurl}{\Hd({\text{curl}})}
\newcommand{\hdiv}{\mathbf{H}(\text{div})}

\newcommand{\hcurlh}{\Hd({\text{curl}_h})}
\newcommand{\hdivh}{\mathbf{H}(\text{div}_h)}

\newcommand{\ltwo}{L^{2}(\Omega)}
\newcommand{\ltwos}{L^{2}}
\newcommand{\ltwod}{\Ld^2(\Omega)}
\newcommand{\ltwods}{\Ld^2}

\newcommand{\linfs}{L^{^\infty}}
\newcommand{\linfds}{\Ld^{^\infty}}

\newcommand{\lzerotwo}{L_0^2(\Omega)}

\newcommand{\bv}[1]{\mathbf{#1}}
\newcommand{\gradv}[1]{\nabla\bv{#1}}
\newcommand{\diver}[1]{\textrm{div} \, \bv{#1}}

\newcommand{\curl}[1]{\textrm{curl} \, \bv{#1}}

\newcommand{\lp}{\bigl(}
\newcommand{\rp}{\bigr)}
\newcommand{\alp}{\left(}
\newcommand{\arp}{\right)}

\newcommand{\lb}{\left\lbrace}
\newcommand{\rb}{\right\rbrace}

\newcommand{\lbb}{\left\lbrace \!\! \left\lbrace}
\newcommand{\rbb}{\right\rbrace \!\! \right\rbrace}
\newcommand{\lj}{\left\llbracket}
\newcommand{\rj}{\right\rrbracket}
\newcommand{\LN}{\left\|}
\newcommand{\RN}{\right\|}

\newcommand{\dt}{\tau}
\newcommand{\inc}{\delta}

\newcommand{\tril}{\,b}

\newcommand{\trilm}{\, b_h^{m}}

\newcommand{\itemizebullet}{$\diamond$}
\newcommand{\polydegree}{\ell}

\newcommand{\nunot}{\nu_{_{\!0}}}
\newcommand{\inertiamom}{\jmath}
\newcommand{\chartime}{\mathscr{T}}
\newcommand{\magdiff}{\sigma}
\newcommand{\permit}{\varkappa_0}

\newcommand{\heff}{\bv{h}}
\newcommand{\hd}{\bv{h}_d}
\newcommand{\ha}{\bv{h}_a}

\newcommand{\Vspace}{\boldsymbol{\mathcal{V}}}
\newcommand{\Hspace}{\boldsymbol{\mathcal{H}}}
\newcommand{\Mspace}{\boldsymbol{\mathcal{M}}}

\newcommand{\utest}{\bv{v}}
\newcommand{\Utest}{\bv{V}}
\newcommand{\FEspaceU}{\mathbb{U}}
\newcommand{\FEspaceUht}{\FEspaceU_{h\dt}}

\newcommand{\FEspaceUdivfree}{\mathbb{V}}

\newcommand{\FEspaceP}{\mathbb{P}}
\newcommand{\FEspacePht}{\mathbb{Q}_{h\dt}}

\newcommand{\Ptest}{Q}

\newcommand{\FEspaceW}{\mathbb{W}}
\newcommand{\FEspaceWht}{\mathbb{W}_{h\dt}}
\newcommand{\wtest}{\bv{x}}
\newcommand{\Wtest}{\bv{X}}

\newcommand{\FEspaceM}{\mathbb{M}}
\newcommand{\FEspaceMht}{\mathbb{M}_{h\dt}}
\newcommand{\mtest}{\bv{z}}
\newcommand{\Mtest}{\bv{Z}}

\newcommand{\phitest}{\chi}
\newcommand{\FEspacePhi}{\mathbb{X}}
\newcommand{\Phitest}{\mathrm{X}}

\newcommand{\hdpot}{\varphi}
\newcommand{\hdpoth}{\Phi}
\newcommand{\hapot}{\phi}

\newcommand{\hdpoto}{\psi}

\newcommand{\bulkint}[2]{\int_{#1}{#2} \, dx }
\newcommand{\bdryint}[2]{\int_{#1} {#2} \, dS }

\newcommand{\triangulation}{\mathcal{T}_h}
\newcommand{\element}{T}
\newcommand{\quadrilateral}{\mathcal{Q}}
\newcommand{\simplex}{\mathcal{P}}

\newcommand{\dipdir}{\bv{d}}

\newcommand{\cstab}{{C}_{\text{stab}}}

\newcommand{\tf}{t_F}

\newcommand{\bdry}{\Gamma}

\newcommand{\suscep}{\varkappa_0}

\newcommand{\cf}{cf.}

\newcommand{\normal}{\boldsymbol{\mathit{n}}}

\newcommand{\SPvel}{\Pi_{s}}

\newcommand{\SPpress}{\pi_{s}}

\usepackage[usenames,dvipsnames]{xcolor}

\begin{document}


\title[Ferrohydrodynamics]{The equations of ferrohydrodynamics: \\
modeling and numerical methods}

\author[R.H.~Nochetto]{Ricardo H.~Nochetto}
\address[R.H.~Nochetto]{Department of Mathematics and Institute for Physical Science and Technology, University of Maryland, College Park, MD 20742, USA.}
\email{rhn@math.umd.edu}

\author[A.J.~Salgado]{Abner J.~Salgado}
\address[A.J.~Salgado]{Department of Mathematics, University of Tennessee, Knoxville, TN 37996, USA.}
\email{asalgad1@utk.edu}

\author[I.~Tomas]{Ignacio Tomas}
\address[I.~Tomas]{Department of Mathematics, University of Maryland, College Park, MD 20742, USA.}
\curraddr{Department of Mathematics, Texas A\&M University, College Station, TX 77843, USA.}
\email{itomas@tamu.edu}

\thanks{The work of RHN and IT is supported by NSF grants DMS-1109325 and DMS-1411808. AJS is partially supported by NSF grant DMS-1418784.}



\date{Submitted \today}

\begin{abstract}
We discuss the equations describing the motion of ferrofluids subject to an external magnetic field. We concentrate on the model proposed by R. Rosensweig, provide an appropriate definition for the effective magnetizing field, and explain the simplifications behind this definition. We show that this system is formally energy stable, and devise a numerical scheme that mimics the same stability estimate. We prove that solutions of the numerical scheme always exist and, under further simplifying assumptions, that the discrete solutions converge. We also discuss alternative formulations proposed in pre-existing work, primarily involving a regularization of the magnetization equation and supply boundary conditions which lead to an energy stable system. We present a series of numerical experiments which illustrate the potential of the scheme in the context of real applications.

\end{abstract}
\maketitle

\section{Introduction}
\label{sec:Intro}

A ferrofluid is a liquid which becomes strongly magnetized in the presence of applied magnetic fields. It is a colloid made of nanoscale monodomain ferromagnetic particles suspended in a carrier fluid (water, oil, or other organic solvent). These particles are suspended by Brownian motion and will not precipitate nor clump under normal conditions. Ferrofluids are dielectric (non conducting) and paramagnetic (they are attracted by magnetic fields, and do not retain magnetization in the absence of an applied field); see \cite{Behrens}.

Ferrofluids can be controlled by means of external magnetic fields, which gives rise to a wealth of control-based applications. They were developed in the 1960's to pump fuel in spacecrafts without mechanical action \cite{Stephen1995}. Recent interest in ferrofluids is related to technical applications such as instrumentation, vacuum technology, lubrication, vibration damping, radar absorbing materials, and acoustics \cite{Miwa2003,Raj1995,Vinoy2011}. For instance, they are used as liquid seals for the drive shafts of hard disks, for vibration control and damping in vehicles and enhanced heat transfer of electronics. Other potential applications are in micro/nanoelectromechanical systems: magnetic manipulation of microchannel flows, particle separation, nanomotors, micro electrical generators, and nanopumps \cite{Shib2011,Hart2004,Yama2005,Zahn01,Zeng2013}. One of the most promising applications are in the field of medicine, where targeted (magnetically guided) chemotherapy and radiotherapy, hyperthermia 
treatments,
 and magnetic resonance imaging contrast enhancement are very active areas of research \cite{Rin2009,Pank03,Shap2013}. An interesting potential application of ferrofluids under current research is the construction of adaptive deformable mirrors \cite{Laird04,Laird06,Brou07}.

The applications mentioned above justify the development of tools for the simulation of ferrofluids. There are some mathematical models, based on systems of partial differential equations \cite{Ros97,Shli2002}. However, these are analytically tractable only in a very limited number of cases \cite{Rinal02,Zahn95}. In this sense, numerical analysis and scientific computation have a lot to offer to smart fluids and smart materials in general.

Mathematical modeling of ferrofluids may have some points in common with magnetohydrodynamics (MHD), micromagnetism and liquid crystals, but generally speaking it uses significantly different ideas. For instance, the equations of MHD deal with nonmagnetizable but electrically conducting fluids, which is in sharp constrast to ferrofluids. The dominant body force in MHD is the Lorentz force, whereas for ferrofluids the Kelvin force is the most important one, leading to different kinds of nonlinearities. The Landau-Lifshitz-Gilbert equations of micromagnetism use rigid director fields as they model saturated \emph{magnetically hard} materials: they have a high coercive force, are difficult to magnetize and demagnetize, but are capable of retaining a significant residual magnetization. If the magnetization $\bv{m}$ satisfies $|\bv{m}| = m_s$, with $m_s$ the saturation magnetization, we can factor out $m_s$ and include it into the constitutive parameters, thus obtaining an equation for unitary director fields. On 
the other hand, ferrofluids are \emph{magnetically soft}: they are easy to magnetize, yet they retain very little or no residual magnetization in the absence of an external magnetic field and they usually exhibit a high saturation value \cite{Behrens}. Therefore, rigid director fields are not suitable to describe the magnetization of a ferrofluid.

There are currently two generally accepted ferrofluid models which we will call by the name of their developers: the Rosensweig  \cite{Ros97} and Shliomis \cite{Shli2002} models. Rigorous work on the analysis (existence of global weak solutions and local existence of strong solutions) for these models is very recent \cite{AmiShliomis2008,Ami2009,Ami2010,Ami2008}. In this work we will concentrate on the Rosensweig model which includes all the inherent difficulties of the simpler Shliomis model. Our key interest is around boundary conditions and discretization techniques leading to energy-stable continuum and discrete systems. This task becomes particularly complicated if we want to include the effects of non-trivial applied magnetizing fields, which so far has not been properly addressed in the literature. For this purpose we will need to revisit the theory of magnetostatics and typical boundary value problems associated with it. At this stage most manipulations are formal, but they are still able to shed 
light into very important issues such as space discretization techniques. Finally, we feel compelled to remark that, to the best of our knowledge, this is the first work presenting a stable numerical scheme for the Rosensweig model fully coupled with the magnetostatics equations accounting for the effect of the demagnetizing field.

Our presentation is organized as follows. In Section \ref{sec:Rosensweig} we present the Rosensweig model: after introducing notation in \S\ref{notation}, we elaborate on several issues related to boundary conditions in \S\ref{hdbcs}--\S\ref{probstat}. In \S\ref{continuumenergy} we derive a formal energy estimate which constitutes the main guideline to devise a numerical scheme. We introduce the numerical method in Section~\ref{sec:Discretization}, show that this scheme is energy stable, and that solutions always exist. In Section~\ref{simpferro1} we consider a simplified model, and devise a scheme for this model for which we show convergence besides stability. In Section~\ref{validation} we validate the scheme using prefabricated solutions, and finally in Section~\ref{sec:numerics} we show a series of numerical examples which illustrate the potential of the scheme in the context of real applications.


\section{The Rosensweig model of ferrohydrodynamics}
\label{sec:Rosensweig}

Consider a mass of homogeneous, incompressible micropolar ferrofluid with linear velocity $\bv{u}$ and angular velocity $\bv{w}$ contained in a bounded simply connected domain $\Omega\subset \mathbb{R}^3$, subject to a smooth harmonic (curl-free and div-free) applied magnetizing field $\ha$ inducing a magnetization field $\bv{m}$ and a demagnetizing (stray) field $\hd$. The evolution of such fluid is described by the following system of partial differential equations (PDEs); see for instance \cite{Ros97,Ros2002}, and Remark \ref{assderros} regarding the intrinsic limitations of this model:
\begin{subequations} 
\label{eq:ferroeq}
\begin{align}
\label{eq:ferroeq:lm}
\bv{u}_t + (\bv{u}\cdot\nabla)\bv{u} - (\nu + \nu_r) \Delta \bv{u} +\nabla p &= 2 \nu_r \curl{w} + \mu_0 (\bv{m}\cdot\nabla)\heff  \, , \\
\label{eq:ferroeq:cm}
\diver{u} &= 0 \, , \\
\label{eq:ferroeq:cam}
\inertiamom \bv{w}_t + \inertiamom (\bv{u}\cdot\nabla)\bv{w} - c_1 \Delta \bv{w} - c_2 \nabla \diver{w} + 4 {\nu}_r \bv{w} &= 2 \nu_r \curl{u} + \mu_0 \bv{m} \times \heff \, ,   \\
\label{eq:ferroeq:cmag}
\bv{m}_t + (\bv{u}\cdot\nabla)\bv{m} - \magdiff \Delta \bv{m} &= \bv{w} \times \bv{m} -
    \tfrac{1}{\chartime} (\bv{m} - \permit \heff) \, ,
\end{align}
\end{subequations}
in $\Omega$ for every $t \in (0,\tf]$, where 
\begin{align}
\label{eq:totalh}
\heff = \ha + \hd \, , 
\end{align}
is the effective magnetizing field; see Remark~\ref{remm1} below. The material constants $\nu$, $\nu_r$, $\mu_0$, $\inertiamom$, $c_a$, $c_d$, $c_0$, $\magdiff$, $\chartime$ and $\permit$ are assumed nonnegative and, moreover, we will assume that their values are such that the model satisfies the Clausius-Duhem inequality (\cf~\cite{Ros97,Ros2002,Luka,Tom13}). Expression \eqref{eq:ferroeq:lm} represents the conservation of linear momentum, \eqref{eq:ferroeq:cm} gives the conservation of mass, \eqref{eq:ferroeq:cam} corresponds to the conservation of angular momentum, and \eqref{eq:ferroeq:cmag} describes the evolution of the magnetization. The forcing term $\mu_0 (\bv{m}\cdot\nabla)\heff$ in the linear momentum equation is the so-called Kelvin force. 

The magnetic diffusion $\magdiff$ was introduced in \cite{Ami2008} as a regularization mechanism in order to prove global existence of weak solutions; but its physical grounds have been called into question \cite{Rin2013conv}, $\magdiff$ being negligibly small or zero. Therefore, we will primarily focus on the case $\magdiff = 0$. If $\magdiff >0$ the boundary conditions associated with the vector Laplacian $\Delta\bv{m}$ allow us to introduce additional modeling features. Thus, in this case, we will propose energy-stable boundary conditions.

The constant $\permit$ is dimensionless; it is the magnetic susceptibility, the product $\mu_0 (1+\suscep)$ is what is usually called the magnetic permeability of the material (\cf~\cite{Jack1998}). For oil-based ferrofluids \cite{Rinal02} we have $\permit \in [0.3,4.3]$, and for water-based ferrofluid $\permit$ is generally smaller than unity. If $\permit = 0$ the medium is not magnetizable, so there is no ferrohydrodynamic phenomena: magnetic fields cannot exert any force or torque on the fluid. The quantity $\permit \heff$ is usually called the \emph{equilibrium magnetization}: if $\magdiff \equiv 0$, $\bv{u} \equiv 0$ and $\bv{w} \equiv 0$ in the magnetization equation \eqref{eq:ferroeq:cmag}, and $\heff$ is given, then we get
\begin{align}\label{boundaryODE}
\bv{m}_t + \tfrac{1}{\chartime} \bv{m} = \tfrac{\permit}{\chartime} \heff \, ,
\end{align}
so that $\bv{m} \approx \permit \heff$ when close to equilibrium. The core dynamics of the magnetization equation in \eqref{eq:ferroeq:cmag} is dominated by the reaction terms for most flows of interest (see for instance \cite{rinaldi2002,Rinal02} for the dimensional analysis of the Rosensweig model). Essentially, this is the case because the relaxation time $\chartime$ of commercial grade ferrofluids is in the range of $10^{-5}$ to $10^{-9}$ seconds (see for instance \cite{rinaldi2002,Shli2002}), which makes $\tfrac{1}{\chartime}$ a very large constant.

System \eqref{eq:ferroeq} is supplemented with initial conditions for the linear velocity, the angular velocity, and the magnetization
\begin{align}
\label{eq:icFERRO}
\bv{u}\rvert_{t=0} = \bv{u}_0 \, , \ \
\bv{w}\rvert_{t=0} = \bv{w}_0 \, , \ \
\bv{m}\rvert_{t=0} = \bv{m}_0 \, , \ \
\end{align}
as well as (no-slip and no-spin) boundary conditions for the linear and angular velocities
\begin{align}
\label{eq:bcMNSE}
\begin{aligned}
  &\bv{u}\rvert_{\bdry \times (0,\tf)} = 0 \, , 
      &&\bv{w}\rvert_{\bdry \times (0,\tf)} = 0 \, .
\end{aligned}
\end{align}
The quantities $\heff$ and $\bv{m}$ are subordinate to Maxwell's equations, which hold in the whole space $\mathbb{R}^3$. Truncating $\heff$ to $\Omega$ and choosing suitable boundary conditions necessarily compromises the nature of the original magnetostatic problem, yet it can provide a reasonable starting point to develop and understand an energy-stable Partial Differential Equations (PDEs) system. We will reduce the magnetostatic problem to a single scalar potential in \S\ref{hdbcs}, and in \S\ref{mbcs} we discuss the information that is lost in this process, derive the (approximate) boundary value problem (BVP) that $\hd$ satisfies, the boundary conditions that can be applied to $\hd$ and $\bv{m}$, and discuss how physically realistic they are. We will also comment on the requirements for $\ha$ to be a physically reasonable magnetizing field.

\begin{remark}[assumptions underlying the derivation of the Rosensweig model] \label{assderros} A ferrofluid is a colloidal suspension of ferromagnetic particles in a carrier fluid. The limitations of model \eqref{eq:ferroeq} can be traced back to the following (quite restrictive) assumptions made on these particles at the time of its derivation (see for instance \cite{Ros2002}):
\begin{itemize}
\item[\itemizebullet] The ferromagnetic particles are spherical.
\item[\itemizebullet] The ferrofluid is a monodisperse mixture, meaning that the ferromagnetic particles are of the same mass/size.
\item[\itemizebullet] The density of ferromagnetic particles (number of particles per unit volume) in the carrier liquid is considered to be homogeneous.
\item[\itemizebullet] No agglomeration, clumping, anisotropic behavior (e.g. formation of chains), nor particle-to-particle interactions are considered.
\item[\itemizebullet] The induced fields ($\bv{m}$ and $\hd$) are unable to perturb the applied magnetic field $\ha$.
\end{itemize}
These assumptions might restrict the applicability of the Rosensweig model for some physical situations; see \cite{Ros97,Oden2009} for more details.
\end{remark}

\begin{remark}[effective magnetizing field] \label{remm1}
The effective magnetizing field is defined by \eqref{eq:totalh}. In other models the effective magnetizing field can be more complicated and include, as in micromagnetics, terms due to the exchange of energy and anisotropy \cite{Mayer2009}. Most analytic computations for the Rosensweig model are usually derived assuming that $\permit << 1$ so that $\hd$ and the equations of magnetostatics can be disregarded, thus setting $\heff := \ha$. In \S\ref{simpferro1} we will consider this scenario and provide some arguments to justify this simplification. Finally, the only available existence results for the Rosensweig model \cite{Ami2010,Ami2008} define the effective magnetizing field in a way that leads to $\heff = \hd$, which is equivalent to considering the unforced case (relaxation to equilibrium).
\end{remark}

\subsection{Notation and assumptions}
\label{notation}
We will assume that the domain $\Omega$ is a convex polyhedron with boundary $\bdry$. We define the trilinear form $\tril(\cdot,\cdot,\cdot)$
\begin{align}
\label{eq:defoftrilb}
  \tril(\bv{m},\heff,\bv{u}) = \sum_{i,j = 1}^d \int_{\Omega} \bv{m}^i \heff_{x_i}^j  \bv{u}^j \, dx \, .
\end{align}
Let $\Hspace$ and $\Vspace$ denote the classical spaces of divergence-free functions \cite{Temam}
\[
  \Hspace = \lb \utest \in \ltwod \, | \, \diver{}\utest = 0  \text{ in } \Omega \text{ and } \utest \cdot \normal = 0 \text{ on } \bdry \rb, \quad
  \Vspace = \lb \utest \in \hzerod \, | \, \diver{}\utest = 0 \, \text{in } \Omega \rb = \hzerod \cap \Hspace \, ,
\]
and let $\Mspace$ be the space
\begin{align*}
\Mspace = \lb \mtest \in \ltwod \ | \ \diver{}\mtest \in \ltwo , \
\curl{}\mtest \in \ltwod \rb = \hcurl \cap \hdiv \, .
\end{align*}
For more details about the spaces $\hcurl$ and $\hdiv$, and the characterization of their traces see, for instance, \cite{Boff2013,Girault}.

We recall the following identity for vector-valued functions in the space $\hzerod$
\begin{align}
\label{eq:Honezdcurldiv}
  |\bv{u}|_{\honeds}^2 = \|\curl{u}\|_{\ltwods}^2 + \|\diver{u}\|_{L^2}^2 \, .
\end{align}
Since the underlying function spaces will always be defined over $\Omega$, from now on we use the abbreviations $\ltwos = \ltwo$, $\hones = \hone$, $\hzeros = \hzero$ for scalar-valued functions, and $\ltwods = \ltwod$, $ \honeds = \honed$, $\hzerods = \hzerod$ for vector-valued valued functions.


\subsection{Modeling the magnetic field: The scalar potential approach}
\label{hdbcs}
One of the main difficulties in the analysis and approximation of \eqref{eq:ferroeq} lies in the fact that the magnetic field $\heff$ is governed by Maxwell's equations which are naturally defined over $\mathbb{R}^3$. Under reasonable assumptions these can be further  simplified to the equations of magnetostatics
\begin{align}
\label{eq:magnetostatics}
\curl{h} = 0 \, , \ \diver{b} = 0 \, \ \text{in }\mathbb{R}^3,
\end{align}
which imply the transmission conditions
\begin{align}
\label{eq:jumpconditions}
\begin{aligned}
\lj \heff \rj  \times \normal = 0  \,  , \ 
\lj \bv{b} \rj  \cdot \normal = 0  \ \text{on } \bdry \, , 
\end{aligned}
\end{align}
over any surface $\bdry \subset \mathbb{R}^3$; hereafter $\normal$ denotes the outward unit normal to $\bdry$ (the boundary of $\Omega$) and $\lj q \rj$ the jump of the quantity $q$ across $\bdry$. The magnetic induction $\bv{b}$ is defined as
\begin{align}
\label{eq:bfielddef}
  \bv{b} = \mu_0 \left( \heff + \boldsymbol{1}_\Omega \bv{m} \right),
\end{align}
where $\boldsymbol{1}_\Omega$ is the characteristic function of $\Omega$. Definitions \eqref{eq:totalh} and \eqref{eq:bfielddef}, together with \eqref{eq:magnetostatics} and \eqref{eq:jumpconditions}, and assuming that $\ha$ is a smooth harmonic (see Remark \ref{remHarmonic}) vector field in $\mathbb{R}^3$, yield the following constraints for $\hd$:
\begin{align}
\label{eq:h_dpde00}
  \curl{h}_d = 0  \text{ in } \mathbb{R}^3 \, , \
  \diver{h}_d = - \diver{}\bv{m} \text{ in } \Omega \, , \
  \diver{h}_d = 0 \text{ in } \mathbb{R}^3\setminus\Omega \, ,
\end{align}
supplemented with (assuming that $\lj\ha\rj = 0$ on $\bdry$):
\begin{align}
\label{eq:transscond}
\begin{aligned}
\normal \times \hd = \normal \times \hd^{\text{out}},  \ \ \ 
\hd\cdot \normal = (\hd^{\text{out}} - \bv{m})\cdot \normal \ \ \ \text{on } \bdry,
\end{aligned}
\end{align}
where we use the superscript ``out'' to denote values in $\mathbb{R}^3\setminus\Omega$. The problem defined by \eqref{eq:h_dpde00} and \eqref{eq:transscond} is the most physical approach to compute $\hd$; see for instance \cite[Chapter 3]{Mayer2009} in the context of micromagnetism. This approach, however, entails a major difficulty: we have to deal with an exterior problem. This may not be an issue from the point of view of analysis but, from the numerical point of view this would require highly specialized techniques for our (already) quite complex ferrohydrodynamics problem. There are just a few references actually solving problem \eqref{eq:h_dpde00}-\eqref{eq:transscond} (see for instance \cite{Koehler1997,Beben2014}), but most generally (see for instance \cite{ProhlBartels2008}) some form of truncation is favored. Following \cite{Ami2009,Ami2010} we will truncate the support of $\hd$, i.e.~we will assume that $\hd^{\text{out}} = 0$, which yields the BVP:
\begin{align}
\label{eq:h_dpde01}
\begin{aligned}
\curl{h}_d &= 0  \text{ in } \Omega \, , \ \diver{h}_d = - \diver{}\bv{m} \text{ in } \Omega \, , \\
\hd\cdot \normal &=  - \bv{m}\cdot \normal \, , \  \normal \times \hd = 0 \text{ on } \bdry \, .
\end{aligned}
\end{align}
The simplification that leads to \eqref{eq:h_dpde01} is not necessarily physically faithful (unless $\hd^{\text{out}} \approx 0$), yet it is widely used in practice \cite{Bossa98,Rapp2013,SolFerr2011,Neittan1996}. We can approximate problem \eqref{eq:h_dpde01} by using a scalar potential approach \cite{Bossa98}, that is we set $\hd = \nabla \hdpoto$ where $\hdpoto$ solves either
\begin{align}
\label{eq:phiNeu}
  -\Delta \hdpoto = \diver{}\bv{m} \ \ \text{in } \Omega \, , \ 
  \frac{\partial \hdpoto}{\partial \normal} =  - \bv{m} \cdot \normal \ \  \text{on } \bdry
\end{align}
or
\begin{align}
\label{eq:phiDir}
-\Delta \hdpoto = \diver{}\bv{m} \ \ \text{in } \Omega \, , \ 
\hdpoto = 0 \ \  \text{on } \bdry \, .
\end{align}
This approach, however, does not retain all the boundary conditions of \eqref{eq:h_dpde01}. The Neumann BVP \eqref{eq:phiNeu} retains $\hd\cdot \normal =  - \bv{m}\cdot \normal$, while the tangential condition $\normal \times \hd = 0$ results from $\hdpoto = 0$ of the Dirichlet BVP \eqref{eq:phiDir}.

Further simplified problems are used in practice for $\hd$. The homogeneous Neumann problem
\begin{align}
\label{eq:phiNeuII}
-\Delta \hdpoto = \diver{}\bv{m} \quad \text{in } \Omega \, , \
\frac{\partial\hdpoto}{\partial \normal} =  0 \quad \text{on } \bdry \, ,
\end{align}
is also used \cite{Ami2008,SolFerr2011,Bossa98,Rapp2013} to approximate the demagnetizing field. Problem \eqref{eq:phiNeuII}, however, can only be justified if $\bv{m}\cdot \normal$ is very small.

Starting from the Maxwell's equations \eqref{eq:magnetostatics} we have derived the Neumann problem \eqref{eq:phiNeuII} for the scalar potential. This encompasses a series of simplifying assumptions, rarely explained in the scientific literature, here made explicit. 
These simplifications compromise the nature of the original magnetostatic problem, but for the time being, the simplified problem \eqref{eq:phiNeuII} will keep the spirit of the demagnetizing field alive. Finally, it is worth mentioning that so far only \eqref{eq:phiNeu} and \eqref{eq:phiNeuII} have been used for the construction of an energy-stable system and the analysis of the Rosensweig model \cite{Ami2008, Ami2009,Ami2010}. In this work, we will use \eqref{eq:phiNeu}, giving us control of the normal condition on $\hd$ but not on the tangential component.
\begin{remark}[physical requirements on $\ha$]\label{remHarmonic}
It is not difficult to see that $\ha$ must be harmonic. If $\omega$ is a control volume and there is no magnetizable media inside $\omega$ we have that $\bv{m} \equiv 0$ and $\hd \equiv 0$ in $\omega$. By Maxwell's equations then $\curl \heff = \curl{}\ha = 0$ and $\diver{}\bv{b} = \mu_0 \diver{}\ha = 0$ in $\omega$. Since $\omega$ is arbitrary $\ha$ is harmonic (curl-free and div-free).
\end{remark}
\begin{remark}[variational problems for the magnetizing fields]\label{rem00}
Multiply \eqref{eq:phiNeu} by a sufficiently smooth test function $\phitest$. Integrating by parts, and using $\tfrac{\partial\hdpoto}{\partial \normal} = - \bv{m}\cdot \normal$ yields
\begin{align}\label{varforphi01}
\int_{\Omega} \nabla\hdpoto \cdot \nabla \phitest \, dx
= -  \int_{\Omega} \bv{m} \cdot \nabla \phitest \, dx
\ \ \ \forall \, \phitest \in \hone \, .
\end{align}
Since $\curl{h}_{a}=0$, there exists a scalar potential $\hapot$ such that $\ha = \nabla \hapot$, then
\begin{align}
\label{varforphi03}
\int_{\Omega} \nabla\hapot \cdot \nabla \phitest \, dx
=  \int_{\Omega} \ha \cdot \nabla \phitest \, dx \ \ \ \forall \, \phitest \in \hone.
\end{align}
It will be useful, primarily to simplify the presentation, to set $\heff = \nabla\hdpot$ with 
$\hdpot = \hdpoto + \hapot$, so that $\hdpot$ satisfies
\begin{align}
\label{varforphi05}
\int_{\Omega} \nabla\hdpot \cdot \nabla \phitest \, dx 
=  \int_{\Omega} ( \ha - \bv{m} )\cdot \nabla \phitest \, dx \ \ \ \forall \, \phitest \in \hone \, , 
\end{align}
which is the variational form of the BVP
\begin{align}
\label{eq:phiNeuIII}
-\Delta \hdpot = \diver{}\bv{m} \ \text{in } \Omega \, , \
\frac{\partial\hdpot}{\partial \normal} = (\ha - \bv{m})\cdot \normal \ \text{on } \bdry \, , \
\end{align}
where the term $- \diver{}\ha$ on the right-hand side of \eqref{eq:phiNeuIII} has been omitted since $\ha$ is solenoidal.
\end{remark}
%
%
\subsection{Boundary conditions for $\bv{m}$ and their coupling with $\hd$}
\label{mbcs}
For the magnetization $\bv{m}$ we consider both $\magdiff = 0$ and $\magdiff > 0$ in \eqref{eq:ferroeq}.
Since $\bv{u}\cdot \normal = 0$ on $\Gamma$, no boundary conditions for $\bv{m}$ are needed when $\sigma=0$, because the PDE for $\bv{m}$ is a transport equation. If $\sigma>0$, on the contrary, we must impose boundary conditions that are compatible with the convection-diffusion equation for $\bv{m}$. For the magnetizing field $\heff$, Maxwell's equations dictated our choice. For $\bv{m}$, however, suitable boundary conditions are rarely discussed in the literature. For this reason, our selection criterion for boundary conditions is whether they lead to an energy law.

The boundary conditions that can be applied to the magnetization $\bv{m}$ are those of the vector Laplacian. Since $- \Delta \bv{m} = \curl{}^2\bv{m} - \nabla\diver{m}$, integration by parts yields
\begin{align}\label{vectlapintparts}
\begin{split}
\bulkint{\Omega}{- \Delta \bv{m} \cdot \mtest} &=
\bulkint{\Omega}{\curl{m}\cdot \curl{}\mtest + \diver{m} \, \diver{}\mtest} \\
&- \bdryint{\bdry}{(\curl{m}\times\normal)\cdot \mtest } \
- \bdryint{\bdry}{\diver{m} (\mtest\cdot\normal)  } \, ,
\end{split}
\end{align}
so that we can consider:
\begin{itemize}
  \item[\itemizebullet] Magnetic boundary conditions (\cf~\cite{Ami2008})
  \begin{align}
  \label{eq:mmagbcs}
    \bv{m}\cdot\normal = g \, , \
    \curl{m} \times\normal = \bv{r} \
    \text{on } \bdry \, ,
  \end{align}
where $g$ and $\bv{r}$ are the boundary data. The condition $\curl{m} \times \normal = \bv{r}$ is natural, while $\bv{m}\cdot \normal = g$ is essential.
  \item[\itemizebullet] Electric boundary conditions (\cf~\cite{Arn2010})
  \begin{align}
  \label{eq:melecbcs}
    \diver{m} = q \, , \ 
    \bv{m} \times \normal = \bv{y} \
    \text{on } \bdry \, ,
  \end{align}
  where the first condition is natural, the second is essential, and the data $\bv{y}$ only has tangential component.
  \item[\itemizebullet] Robin-like boundary conditions
\begin{align}\label{eq:mrobinbcs}
\begin{split}
\curl{m} \times \normal + \gamma_1 (\bv{m} - (\bv{m}\cdot \normal) \normal - \bv{y} ) &= \bv{r} 
\ \ \text{on }\bdry \, , \\
\diver{m} + \gamma_2 (\bv{m}\cdot \normal - g) &= q  \ \ \text{on }\bdry \, .
\end{split}
\end{align}
Since $|\normal|=1$ we have that $(\bv{u} - (\bv{u}\cdot \normal) \normal )\cdot \bv{w} = (\bv{u} \times \normal)\cdot (\bv{w} \times \normal)$, then we get from \eqref{vectlapintparts} the following variational formulation of $- \Delta\bv{m} = \bv{f}$: find $\bv{m} \in \Mspace $ such that
\begin{align*}
\begin{aligned}
&\bulkint{\Omega}{\curl{m}\cdot \curl{v} + \diver{m} \, \diver{v}}
+ \gamma_1 \bdryint{\bdry}{(\bv{m}\times \normal)\cdot(\bv{v}\times \normal)} \\
&+ \gamma_2 \bdryint{\bdry}{(\bv{m}\cdot \normal)(\bv{v}\cdot \normal)}
= \bulkint{\Omega}{\bv{f} \cdot \bv{v}}
+ \bdryint{\bdry}{\bv{r} \cdot \bv{v} } \\
&+ \bdryint{\bdry}{q  (\bv{v}\cdot \normal) }
+ \gamma_1 \bdryint{\bdry}{(\bv{y}\times \normal) \cdot(\bv{v}\times \normal) }
+ \gamma_2 \bdryint{\bdry}{g  (\bv{v}\cdot \normal) }
\ \ \ \ \forall \, \bv{v} \in \Mspace.
\end{aligned}
\end{align*}
The following asymptotic cases are of interest: for $\gamma_1 = 0$, $\gamma_2 \rightarrow \infty$, \eqref{eq:mrobinbcs} tends to the magnetic boundary conditions \eqref{eq:mmagbcs}, while if $\gamma_2 = 0$, $\gamma_1 \rightarrow \infty$, \eqref{eq:mrobinbcs} tends to the electric boundary conditions \eqref{eq:melecbcs}.
\item[\itemizebullet] Natural boundary conditions
\begin{align}
\label{eq:mstable2}
\curl{m} \times \normal = 0, \quad \diver{m}= 0 \qquad \text{on }\bdry \, , 
\end{align}
which lead to an energy-stable system. We will mainly use these conditions to explain the main ideas behind the development of an energy estimate for \eqref{eq:ferroeq}.
\end{itemize}
%
%
\subsection{Simplified initial boundary value problems}
\label{probstat}
In this paper we will consider the simplified Initial Boundary Value Problem (IVBP) of ferrohydrodynamics: given a smooth harmonic vector field $\ha = \ha(x,t)$, we seek $\lb\bv{u},p,\bv{w},\bv{m}, \heff\rb$ satisfying the equations \eqref{eq:ferroeq} in $\Omega \times [0,\tf]$, where $\bv{u}$ and $\bv{w}$ satisfy the boundary conditions \eqref{eq:bcMNSE}, $\heff = \nabla\hdpot$ where $\hdpot$ solves \eqref{eq:phiNeuIII}, and the equation for $\bv{m}$ \eqref{eq:ferroeq:cmag} is supplemented with one of the following boundary conditions:
\begin{itemize}
\item[1.] $\magdiff = 0$ with no boundary conditions for $\bv{m}$.
\item[2.] $\magdiff > 0$ with the natural boundary conditions \eqref{eq:mstable2}.
\item[3.] $\magdiff > 0$ with the following variant of \eqref{eq:mrobinbcs}:
\begin{align}
\label{eq:mstable}
\begin{split}
\curl{m} \times \normal 
+ \gamma \big(\bv{m} - (\bv{m}\cdot \normal) \normal - \permit (\bv{h} - (\bv{h}\cdot \normal) \normal) \big) &= 0 \ \ \text{on }\bdry \, , \\
\diver{m} + \gamma \big(\bv{m}\cdot \normal - \permit \heff \cdot \normal \big) &= 0 \ \ \text{on }\bdry \, , 
\end{split}
\end{align}
which is obtained by setting $\gamma_1 = \gamma$, $\bv{y} = \permit (\bv{h} - (\bv{h}\cdot \normal) \normal)$,  $\bv{r} = 0$, $\gamma_2 = \gamma$, $g = \permit \heff \cdot \normal $, and $q = 0$ in \eqref{eq:mrobinbcs}, with $\gamma$ being a material constant that characterizes the magnetization dynamics on the surface of the ferrofluid.
\end{itemize}
\begin{remark}[boundary dynamics] A possible physical explanation for \eqref{eq:mstable} is that, on the boundary, $\bv{m}$ will attempt to reach equilibrium $\bv{m} = \permit \heff$ according to \eqref{boundaryODE}, but it will lag behind since $\bv{m}$ can only change at a finite rate limited by the characteristic dynamics of the magnetization. This is consistent with the no-slip and no-spin boundary conditions \eqref{eq:bcMNSE}, so the behavior of $\bv{m}$ on the boundary is solely controlled by a magnetic relaxation time as in \eqref{boundaryODE}.
\end{remark}
%
%
\section{A priori estimates and existence}
\label{existence}
Let us review the available existence results for the problems under consideration. We shall first provide some formal a priori energy estimates, which serve as basis for the existence results that will be discussed later and, more importantly, will guide us in the development of stable numerical schemes; see \S\ref{sec:Discretization}.
\subsection{A priori energy estimates}
\label{energyest}
Let us obtain an energy estimate for Case 2 of \S\ref{probstat}. Setting $\magdiff = 0$ in the final estimate we get the corresponding estimate for Case 1. The energy estimate for Case 3, is outlined in Remark~\ref{stabrobin}. We begin by showing two crucial identities that make possible the energy estimate.
\begin{lemma}[identities for the magnetization and magnetic field]
\label{lem:ids}
Let $\bv{m}$ and $\bv{h}$ denote the magnetization and effective magnetizing field, respectively. If they are sufficiently smooth we have
\begin{align}
\label{eq:vectlapint}
- (\Delta\bv{m},\heff) = - \|\diver{}\heff\|_{\ltwods}^2
- \bdryint{\bdry}{(\curl{m} \times \normal) \cdot \heff}
- \bdryint{\bdry}{\diver{m} \, (\heff\cdot \normal)} \, ,
\end{align}
and, for every smooth vector field $\utest$, such that $\diver{}\utest = 0$ in $\Omega$ and $\utest\cdot \normal = 0$ on $\bdry$
\begin{align}
\label{eq:centralident}
\tril(\utest,\bv{m},\heff) = - \tril(\bv{m},\heff,\utest) \, ,
\end{align}
where the trilinear form $\tril(\cdot,\cdot,\cdot)$ was defined in \eqref{eq:defoftrilb}. 
\end{lemma}
\begin{proof}
To obtain \eqref{eq:vectlapint}, we first take $\mtest := \heff$ in \eqref{vectlapintparts} and recall that $\curl{}\heff = 0$. Upon multiplying \eqref{eq:phiNeuIII} by $\diver{}\heff = \Delta\hdpot$ and integrating, we deduce $(\diver{m},\diver{}\heff ) = - \|\diver{}\heff\|_{\ltwods}^2$, which substituted in \eqref{vectlapintparts} yields \eqref{eq:vectlapint}.

Since $\heff$ is curl-free we have that $\heff_{x_i}^j = \heff_{x_j}^i$. Integration by parts then yields
{\allowdisplaybreaks
\begin{align}
\label{skewproof}
\begin{aligned}
\tril(\bv{m},\heff,\bv{v}) &= \sum_{i,j = 1}^d \int_{\Omega} \bv{m}^i \heff_{x_i}^j  \bv{v}^j \, dx =
\sum_{i,j = 1}^d \int_{\Omega} \bv{m}^i \heff_{x_j}^i  \bv{v}^j \, dx
= \sum_{i,j = 1}^d \int_{\Omega} \lp ( \bv{m}^i \heff^i  )_{x_j} - \bv{m}_{x_j}^i \heff^i \rp \bv{v}^j \, dx  \\
&= \sum_{i,j = 1}^d \int_{\Omega}  - ( \bv{m}^i \heff^i  ) \bv{v}_{x_j}^j
- \bv{m}_{x_j}^i \heff^i \bv{v}^j \, dx
+ \bdryint{\bdry}{( \bv{m}^i \heff^i  ) \bv{v}^j \normal^j}   \\
&= - \bulkint{\Omega}{(\bv{m}\cdot\heff) \, \diver{v} } - \tril(\bv{v},\bv{m},\heff)
+ \bdryint{\bdry}{( \bv{m} \cdot \heff ) \bv{v}\cdot \normal} \, . 
\end{aligned}
\end{align}}
The fact that $\utest$ is solenoidal and has zero normal trace on the boundary yields \eqref{eq:centralident}. \end{proof}
With these identities at hand we obtain a formal energy estimate. To shorten the exposition we denote
{\allowdisplaybreaks
\begin{align*}
\mathscr{E} &= \mathscr{E}(\bv{u},\bv{w},\bv{m},\heff;s) = \tfrac{1}{2} \alp \|\bv{u}(s)\|_{\ltwods}^2
+ \inertiamom \|\bv{w}(s)\|_{\ltwods}^2
+ \mu_0 \|\bv{m}(s)\|_{\ltwods}^2
+ \mu_0 \|\heff(s)\|_{\ltwods}^2 \arp, \\
\mathscr{D} &= \mathscr{D}(\bv{u},\bv{w},\bv{m},\heff;s) = \nu \|\gradv{u}(s)\|_{\ltwods}^2
+ c_1 \|\gradv{w}(s)\|_{\ltwods}^2
+ \magdiff \mu_0 \|\curl{m}(s)\|_{\ltwods}^2 \\
&\ \ \ \ \ \ \ \ \ \ \ \ \ \ \ \ \ \ \ \ \ \ \ \ \
+ \magdiff \mu_0 \|\diver{m}(s)\|_{\ltwods}^2
+ \magdiff \mu_0 \|\diver{}\heff(s)\|_{\ltwods}^2
+ c_2 \|\diver{w}(s)\|_{\ltwods}^2 \\
&\ \ \ \ \ \ \ \ \ \ \ \ \ \ \ \ \ \ \ \ \ \ \ \ \
+ \nu_r \|(\curl{u} - 2 \bv{w})(s)\|_{\ltwods}^2
+ \tfrac{\mu_0}{\chartime} \|\bv{m}(s)\|_{\ltwods}^2
+ \tfrac{\mu_0}{2\chartime} \big(\tfrac{1}{2}
+ 3 \permit\big) \|\heff(s)\|_{\ltwods}^2 \, , \\
\mathscr{F} &= \mathscr{F}(\ha;s) = \chartime \mu_0  \|\partial_t\ha(s)\|_{\ltwods}^2 +
\tfrac{\mu_0}{2\chartime} (1+\permit) \|\ha(s)\|_{\ltwods}^2.
\end{align*}}
\begin{proposition}[formal energy estimate]\label{continuumenergy} The solution $\lb\bv{u},p,\bv{w},\bv{m}, \heff\rb$ of problem \eqref{eq:ferroeq}, \eqref{eq:icFERRO}, \eqref{eq:bcMNSE} and \eqref{eq:phiNeuIII} satisfies
\begin{align}
\label{finalest}
\begin{split}
\mathscr{E}(\bv{u},\bv{w},\bv{m},\heff;\tf) 
&+ \int_0^{\tf} \mathscr{D}(\bv{u},\bv{w},\bv{m},\heff;s) \, ds \\
&\leq \int_0^{\tf}\mathscr{F}(\ha;s) \, ds + \mathscr{E}(\bv{u},\bv{w},\bv{m},\heff;0).
\end{split}
\end{align}
\end{proposition}
\begin{proof}
The main ideas of this estimate come from \cite{Kalo2002,Ami2008,AmiShliomis2008}, but we now use \eqref{varforphi05} for the scalar potential associated to the magnetizing field instead. We multiply \eqref{eq:ferroeq:lm} by $\bv{u}$, \eqref{eq:ferroeq:cam} by $\bv{w}$, \eqref{eq:ferroeq:cmag} by $\mu_0 \bv{m}$, and integrate by parts. Setting $\phitest = \tfrac{\permit\mu_0}{\chartime} \hdpot$ in \eqref{varforphi05} and adding the ensuing identities yields
{\allowdisplaybreaks
\begin{align}
\label{eq:firststep}
\begin{split}
\tfrac{d}{dt} \left( \mathscr{E} - \tfrac{\mu_0}{2} \|\heff\|_{\ltwods}^2 \right)
&+ \mathscr{D} - \magdiff \mu_0 \|\diver{}\heff\|_{\ltwods}^2 
- \tfrac{\mu_0}{2\chartime} \left(\tfrac{1}{2} + \permit\right) \|\heff\|_{\ltwods}^2 \\
&= \mu_0 \tril(\bv{m},\heff,\bv{u}) + \mu_0 (\bv{m} \times \heff,\bv{w})
+ \tfrac{\permit \mu_0}{\chartime} (\ha,\nabla\hdpot) \\
&+ \magdiff \mu_0 \bdryint{\bdry}{(\curl{m}\times \normal)\cdot \bv{m}}
+\magdiff \mu_0 \bdryint{\bdry}{\diver{m} \, (\bv{m}\cdot \normal)} \, .
\end{split}
\end{align}}
Note that to form the term $\|\curl{u} - 2 \bv{w}\|_{\ltwods}^2$ in $\mathscr{D}$ we have used \eqref{eq:Honezdcurldiv} on $\|\gradv{u}\|_{\ltwods}^2$.
Multiply the magnetization equation \eqref{eq:ferroeq:cmag} by $\mu_0\heff$. Identities \eqref{eq:vectlapint} and \eqref{eq:centralident} give 
{\allowdisplaybreaks\begin{align*}
\magdiff \mu_0 \|\diver{}\heff\|_{\ltwods}^2
+ \tfrac{\permit\mu_0}{\chartime} \|\heff\|_{\ltwods}^2 &=
- \mu_0 \tril(\bv{m},\heff,\bv{u})
- \mu_0(\bv{m} \times \heff, \bv{w})
+ \mu_0 (\bv{m}_t,\heff) \\
&+ \tfrac{\mu_0}{\chartime} (\bv{m},\heff)
- \magdiff \mu_0 \bdryint{\bdry}{(\curl{m} \times \normal) \cdot \heff} \\
&- \magdiff \mu_0 \bdryint{\bdry}{\diver{m} \, (\heff\cdot \normal)} \, .
\end{align*}}
Adding this expression to \eqref{eq:firststep} we get 
{\allowdisplaybreaks\begin{align}
\label{eq:secondstep}
\begin{split}
\tfrac{d}{dt} \left( \mathscr{E} - \tfrac{\mu_0}2 \|\heff\|_{\ltwods}^2 \right)
&+ \mathscr{D} - \tfrac{\mu_0}{2\chartime} \left(\tfrac{1}{2} - \permit \right) \|\heff\|_{\ltwods}^2
= \mu_0 (\bv{m}_t,\heff)
+ \tfrac{\mu_0}{\chartime} (\bv{m},\heff) \\
&+ \tfrac{\permit \mu_0}{\chartime} (\ha,\nabla\hdpot)
+ \magdiff \mu_0 \bdryint{\bdry}{(\curl{m}\times \normal)\cdot (\bv{m}-\heff)} \\
&+ \magdiff \mu_0 \bdryint{\bdry}{\diver{m} \, (\bv{m} - \heff) \cdot \normal } \, .
\end{split}
\end{align}}
Set $\phitest = \nabla \hdpot $ in \eqref{varforphi05} to obtain 
\begin{align}
\label{phiID00}
  \|\nabla\hdpot\|_{\ltwods}^2 = (\ha - \bv{m}, \nabla\hdpot) \, .
\end{align}
Differentiate \eqref{varforphi05} with respect to time and set $\phitest = \hdpot $. This implies
\begin{align}
\label{phiID01}
\tfrac{1}{2}\tfrac{d}{dt} \|\nabla\hdpot\|_{\ltwods}^2 = (\partial_t\ha - \partial_t\bv{m}, \nabla\hdpot) \, .
\end{align}
Insert these two identities in \eqref{eq:secondstep}, and use that $\bv{h} = \nabla \varphi$, to get
\begin{align}
\label{eq:thirdstep}
\begin{split}
\tfrac{d}{dt} \mathscr{E}(\bv{u},\bv{w},\bv{m},\heff;t) 
&+ \mathscr{D}(\bv{u},\bv{w},\bv{m},\heff;t)
+ \tfrac{\mu_0}{4\chartime} (3 + 2 \permit) \|\nabla\hdpot\|_{\ltwods}^2 \\
&= \mu_0 (\partial_t \ha , \nabla\hdpot) 
+ \tfrac{\mu_0}{\chartime} (1 + \permit )(\ha,\nabla\hdpot) \\
&+ \magdiff \mu_0 \bdryint{\bdry}{(\curl{m}\times \normal)\cdot (\bv{m}-\heff)} \\
&+ \magdiff \mu_0 \bdryint{\bdry}{\diver{m} \, (\bv{m} - \heff) \cdot \normal } \, .
\end{split}
\end{align}
Notice that all the boundary integrals that appear on the right hand side of \eqref{eq:thirdstep} are multiplied by $\sigma$, so that this is already the sought energy estimate for Case 1 of \S\ref{probstat}. For Case 2, the boundary conditions \eqref{eq:mstable2} imply that the boundary integrals in \eqref{eq:thirdstep} vanish, whence
\begin{align*}
\tfrac{d}{dt} \mathscr{E}(\bv{u},\bv{w},\bv{m},\heff;t) 
&+ \mathscr{D}(\bv{u},\bv{w},\bv{m},\heff;t)
+ \tfrac{\mu_0}{4\chartime} (3 + 2 \permit) \|\nabla\hdpot\|_{\ltwods}^2 \\
&= \mu_0 (\partial_t \ha , \nabla\hdpot)
+ \tfrac{\mu_0}{\chartime} (1 + \permit )(\ha,\nabla\hdpot) \, .
\end{align*}
After suitably bounding the terms on the right hand side, integration in time yields the desired estimate \eqref{finalest}. \end{proof}
Notice that \eqref{finalest} suggests that the natural spaces to search for a solution are
{\allowdisplaybreaks\begin{align}
\label{enspaces1}
\begin{aligned}
\bv{u} &\in L^{\infty}(0,\tf,\Hspace) \cap L^2(0,\tf,\Vspace) \\
\bv{w} &\in L^{\infty}(0,\tf,\ltwod) \cap L^2(0,\tf,\hzerod) \\
\bv{m}, \ \heff &\in L^{\infty}(0,\tf,\ltwod) \cap L^2(0,\tf,\Mspace) \, .
\end{aligned}
\end{align}}
\begin{remark}[energy stability using Robin boundary conditions]
\label{stabrobin} Multiplying \eqref{eq:ferroeq:lm}, \eqref{eq:ferroeq:cam} and \eqref{eq:ferroeq:cmag} by $\permit \bv{u}$, $\permit\bv{w}$ and $\permit\mu_0 \heff$, respectively, and following the arguments of Proposition~\ref{continuumenergy} we obtain
{\allowdisplaybreaks\begin{align}
\begin{aligned}
\tfrac{d}{dt} \big( \tfrac{\permit}{2} \|\bv{u}\|_{\ltwods}^2
&+ \tfrac{\permit \inertiamom}{2} \|\bv{w}\|_{\ltwods}^2
+ \tfrac{\mu_0}{2} \|\bv{m}\|_{\ltwods}^2 \big)
+ \permit \nu \|\gradv{u}\|_{\ltwods}^2
+ \permit c_1 \|\gradv{w}\|_{\ltwods}^2 \\
&+ \magdiff \mu_0 \|\curl{m}\|_{\ltwods}^2 
+ \magdiff \mu_0 (1+\permit) \|\diver{m}\|_{\ltwods}^2
+ \permit c_2 \|\diver{w}\|_{\ltwods}^2 \\
&+ \permit \nu_r \|\curl{u} - 2 \bv{w}\|_{\ltwods}^2 
+ \tfrac{\mu_0}{\chartime} \|\bv{m}\|_{\ltwods}^2
+ \tfrac{\mu_0 \permit }{\chartime} (1+\permit) \|\nabla\hdpot\|_{\ltwods}^2 \\
&= \permit \mu_0 (\bv{m}_t,\heff)
+ \tfrac{\permit\mu_0}{\chartime} (\bv{m},\heff)
+ \tfrac{\permit \mu_0}{\chartime} (\ha,\nabla\hdpot)\\
&+ \magdiff \mu_0 \bdryint{\bdry}{(\curl{m}\times \normal)\cdot (\bv{m}-\permit\heff)} \\
&+ \magdiff \mu_0 \bdryint{\bdry}{\diver{m} \, (\bv{m} - \permit \heff) \cdot \normal } \, .
\end{aligned}
\end{align}}
The Robin-type boundary conditions \eqref{eq:mstable} and identities \eqref{phiID00} and \eqref{phiID01} yield
\begin{align}
\begin{aligned}
\tfrac{d}{dt} \big( \tfrac{\permit}{2} \|\bv{u}\|_{\ltwods}^2
&+ \tfrac{\permit \inertiamom}{2} \|\bv{w}\|_{\ltwods}^2
+ \tfrac{\mu_0}{2} \|\bv{m}\|_{\ltwods}^2
+ \tfrac{\permit\mu_0}{2} \|\nabla\hdpot\|_{\ltwods}^2 \big)
+ \permit \nunot \|\gradv{u}\|_{\ltwods}^2 \\
&+ \permit c_1 \|\gradv{w}\|_{\ltwods}^2 
+ \magdiff \mu_0 \|\curl{m}\|_{\ltwods}^2
+ \magdiff \mu_0 (1+\permit) \|\diver{m}\|_{\ltwods}^2 \\
& + \permit c_2 \|\diver{w}\|_{\ltwods}^2
+ \permit \nu_r \|\curl{u} - 2 \bv{w}\|_{\ltwods}^2
+ \tfrac{\mu_0}{\chartime} \|\bv{m}\|_{\ltwods}^2 \\
&+ \tfrac{\mu_0 \permit }{\chartime} (2+\permit) \|\nabla\hdpot\|_{\ltwods}^2
+ \magdiff \mu_0 \gamma \bdryint{\bdry}{|(\bv{m}-\permit\heff) \times \normal|^2 } + \\
&+ \magdiff \mu_0 \gamma \bdryint{\bdry}{|(\bv{m} - \permit \heff) \cdot \normal|^2 } \\
&= \permit \mu_0 (\partial_t\ha,\nabla\hdpot)
+ \tfrac{2 \permit \mu_0}{\chartime} (\ha,\nabla\hdpot) \, .
\end{aligned}
\end{align}
A trivial application of the Cauchy-Schwarz inequality shows that the system is energy-stable with the boundary conditions \eqref{eq:mstable}. Note that we also have control of additional boundary terms. 
\end{remark}
\begin{remark}[Neumann boundary conditions]
\label{rem:ppp}
If we were to supplement the system with the boundary conditions $\bv{m} \cdot \normal = \heff \cdot \normal = 0$ we would obtain, for $\hdpot$, problem \eqref{eq:phiNeuIII} with homogeneous Neumann data ($\partial_{\normal}\hdpot = 0 \text{ on }\bdry$). This is used, for instance, in \cite{Ami2010}. However, with the present techinques we would not be able to obtain an energy estimate.
\end{remark}
%
%
\subsection{Existence results}
\label{relatedexistence}
To date, the results concerning existence of solutions for the equations of ferrohydrodynamics \eqref{eq:ferroeq} available in the literature are as follows:
\begin{enumerate}
\item[1.] \textbf{Local strong solution \cite{Ami2010}.} Let $\magdiff = 0$ and $\heff = \nabla \vartheta$, where $\vartheta$ solves
\begin{align}
\label{poiss1}
-\Delta \vartheta = \diver{}(\bv{m} - \ha) \ \text{in } \Omega \, , \
\frac{\partial\vartheta}{\partial \normal} = - \bv{m}\cdot \normal \ \text{on } \bdry \, ,
\end{align}
then, for $q > 3$ and $r = \min\{q,6\}$, there exists a time $T^* > 0$ for which problem \eqref{eq:ferroeq} has a unique strong solution  $\lb\bv{u},p,\bv{w},\bv{m}, \heff\rb$ such that
\begin{align*}
&\bv{u} \in L^{\infty}(0,T^*,\htwod \cap \Vspace) \, \cap \, W_{\infty}^{1}(0,T^*, \Hspace) \, \cap \, L^2(0,T^*, \bv{W}_{r}^{2}(\Omega)) \, ,\\
&p \in L^2(0,T^*,W_r^{1}(\Omega)) \, , \\
&\bv{w} \in L^{\infty}(0,T^*,\htwod \cap \hzerod) \, \cap \, W_{\infty}^{1}(0,T^*, \ltwod) \, \cap \, L^2(0,T^*, \bv{H}^{3}(\Omega)) \, , \\
&\bv{m},\heff \in L^{\infty}(0,T^*,\bv{W}_{\infty}^{1}(\Omega)) \, \cap \, W_{\infty}^{1}(0,T^*,\bv{L}^q(\Omega)) \, ,
\end{align*}
provided that the data $\lb \bv{u}_0,\bv{w}_0, \bv{m}_0, \ha \rb$ are sufficiently small and regular, i.e.
\begin{align*}
\bv{u}_0 \in \htwod \cap \Vspace \  ,  \ \
\bv{w}_0 \in \htwod \cap \hzerod \ , \ \
\bv{m}_0 \in W^{1}_q(\Omega) \ , \ \ 
\diver{}{\ha} \in W^{1}_{\infty}(0,T^*,L^{q}(\Omega)) \, . 
\end{align*}
\item[2.] \textbf{Global weak solutions \cite{Ami2008}.} Let $\magdiff > 0$ and $\heff = \nabla \vartheta$, where $\vartheta$ solves
\begin{align}
\label{poiss2}
-\Delta \vartheta = \diver{}(\bv{m} - \ha) \ \text{in } \Omega \, , \
\frac{\partial\vartheta}{\partial \normal} = 0 \ \text{on } \bdry  \, ,
\end{align}
and consider the magnetic boundary conditions $\curl{m} \times \normal = 0$ and $\bv{m} \cdot \normal = 0$ for $\bv{m}$.  Then for every set of data $\lb \bv{u}_0,\bv{w}_0, \bv{m}_0, \ha \rb$ that satisfies
\begin{align*}
\begin{gathered}
\bv{u}_0 \in \Hspace \  ,  \ \
\bv{w}_0 \in \ltwod \ , \ \
\bv{m}_0 \in \ltwod \ , \ \
\diver{}{\ha} \in H^{1}(0,\tf,\ltwo) \, .
\end{gathered}
\end{align*}
there is a global in time weak solution $\lb\bv{u},p,\bv{w},\bv{m}, \heff\rb$ of problem \eqref{eq:ferroeq} such that
\begin{align*}
&\bv{u} \in L^{\infty}(0,\tf,\Hspace) \, \cap \, L^{2}(0,\tf,\Vspace) \, \cap \, \mathcal{C}_w([0,\tf],\Hspace) \, , \\
&\bv{w} \in L^{\infty}(0,\tf,\ltwod) \, \cap \, L^{2}(0,\tf,\hzerod) \, \cap \, \mathcal{C}_w([0,\tf],\ltwod) \, , \\
&\bv{m} \in L^{\infty}(0,\tf,\ltwod) \, \cap \, L^{2}(0,\tf,\Mspace) \, \cap \, \mathcal{C}_w([0,\tf],\ltwod) \, , \\
&\heff \in L^{\infty}(0,\tf,\ltwod) \, \cap \, L^{2}(0,\tf,\honed)  ,
\end{align*}
where for a Hilbert space $\mathcal{H}$ we denote by $\mathcal{C}_w([0,\tf],\mathcal{H})$ the space of functions $f:[0,\tf] \to \mathcal{H}$ that are continuous in the weak topology of $\mathcal{H}$.
\end{enumerate}
Note that the BVPs for the magnetic potential \eqref{poiss1} and \eqref{poiss2} are different from the one proposed in \eqref{eq:phiNeuIII}. The BVPs \eqref{poiss1} and \eqref{poiss2} are not appropriate to capture the effects of an external magnetic field $\ha$. In fact, if $\ha$ is divergence-free (a physically reasonable $\ha$ in the context of dielectric media should be divergence-free, according to Remark \ref{remHarmonic}), the BVPs \eqref{poiss1} and \eqref{poiss2} would reduce to $\heff = \hd$ (as defined in \eqref{eq:phiNeu} and \eqref{eq:phiNeuII} respectively), with no effect from $\ha$, so that the behavior of the system would reduce to relaxation to equilibrium. In this sense, the BVP proposed in \eqref{eq:phiNeuIII} is a much more physically realistic approximation to the effective magnetic field than \eqref{poiss1} or \eqref{poiss2}.
%
%
\section{Discretization}
\label{sec:Discretization}
\subsection{Notation}
We introduce $K > 0$ to denote the number of time steps, define the time step as
$\dt = \tf/K >0$ and set $t^k = k\dt$ for $0\leq k \leq K$. For
$\varrho : [0,\tf] \rightarrow E$, we set 
\begin{align}\label{ptevaldef}
\varrho^k = \varrho(t^k).
\end{align}
A sequence will be denoted by  $\varrho^\dt = \lb \varrho^k\rb_{k=0}^{K}$ and we introduce the following norms:
\begin{align*}
\|\varrho^\dt \|_{\ell^\infty(E)} := \max_{0 \leq k \leq K} \|\varrho^k\|_E \, , \ \
\|\varrho^\dt \|_{\ell^{q}(E)} = \alp \sum_{k=0}^K \dt \|\varrho^k\|_E^{q} \arp^{\!\nicefrac{1}{q}} \, , \  \
q \in [1,\infty) \, .
\end{align*}
We also define the backward difference operator $\inc$: 
\begin{align}\label{backwarddifferenceop}
  \delta\varrho^k := \varrho^k - \varrho^{k-1} . 
\end{align}
The following identity will be used repeatedly
\begin{align}\label{sumid}
2 (a,a-b) = |a|^2 - |b|^2 + |a-b|^2 \, .
\end{align}
Similarly, the following ``summation by parts'' formula (also called Abel's transformation) will be used in this work:
\begin{align}\label{summation}
\sum_{k = 1}^{K} a^k \,  \inc b^k = a^K b^K - a^0 b^0 
- \sum_{k = 1}^{K-1} b^k \, \inc a^{k+1} \, . 
\end{align}
For the space discretization we introduce finite dimensional subspaces $\FEspaceU \subset \hzerod$, $\FEspaceP \subset \ltwo$, $\FEspaceW \subset \hzerod$, $\FEspaceM \subset \ltwod$ and $\FEspacePhi \subset \hone$, where we will approximate the linear velocity, pressure, angular velocity, magnetization and magnetic potential, respectively. About the pair of spaces $\lb \FEspaceU,\FEspaceP \rb$ we assume that they are LBB stable, meaning that there exists $\beta^* > 0$ independent of the discretization parameter $h$ such that
\begin{equation}
\label{discreteinfsup}
 \inf_{0\neq\Ptest \in \FEspaceP} \sup_{0\neq\Utest \in \FEspaceU} 
  \frac{( \diver{}\Utest, \Ptest )}{\|\Ptest\|_{\ltwods} \|\nabla\Utest\|_{\ltwods} } \geq \beta^* \, . 
\end{equation}
Specific construction and examples of finite element spaces satisfying this condition can be found in \cite{Girault,ErnGuermond}.

To be able to focus on the fundamental difficulties in the design of an energy stable scheme we will first describe the scheme without being specific on the particular structure of these spaces. As we will see, their choice shall come naturally from this analysis.

We introduce a discretization of the trilinear form $\tril(\cdot,\cdot,\cdot)$ in the equations for the linear and angular velocities: 
\begin{align*}
\tril_h(\cdot,\cdot, \cdot) : \FEspaceU \times (\FEspaceU + \FEspaceW ) \times (\FEspaceU + \FEspaceW ) \to \mathbb{R}. 
\end{align*}
We assume that $\tril_h(\cdot,\cdot,\cdot)$ is skew-symmetric with respect to its last two arguments, i.e.
\begin{align}
\label{eq:bhmskewns}
\tril_h(\bv{U},\bv{V},\bv{W}) = - \tril_h(\bv{U},\bv{W},\bv{V}), \quad
\forall \, \bv{U}\in \FEspaceU; \, \bv{V},\bv{W} \in (\FEspaceU + \FEspaceW) .
\end{align}
Similarly, we will also need another discretization for the trilinear form associated to the Kelvin force $\mu_0(\bv{m} \cdot \nabla) \bv{h}$ in \eqref{eq:ferroeq:lm} and the convective term of the magnetization equation  $(\bv{u}\cdot\nabla) \bv{m}$ in \eqref{eq:ferroeq:cmag}
\begin{align*}
 \trilm(\cdot, \cdot,\cdot): \FEspaceU \times \FEspaceM \times \FEspaceM \to \mathbb{R} \, ,
\end{align*}
and we will also assume that it is skew-symmetric with respect to its last two arguments
\begin{align}
\label{eq:bhmskew}
\trilm(\bv{U},\bv{V},\bv{W}) = - \trilm(\bv{U},\bv{W},\bv{V})\, , \
\forall \, \bv{U} \in \FEspaceU \, ; \ \bv{V},\bv{W} \in \FEspaceM \, .
\end{align}
Finally, we introduce a consistent discretization of the vector Laplacian
\begin{align*}
a_h(\cdot,\cdot): \FEspaceM \times \FEspaceM \to \mathbb{R} \, ,
\end{align*}
which we assume is coercive, that is
\begin{align}\label{positivityprop}
a_h(\bv{M},\bv{M}) \geq \cstab | \bv{M} |_{a}^2 \ \ \forall \, \bv{M} \in \FEspaceM \, ,
\end{align}
for a discrete semi-norm $|\cdot|_{a}$ to be specified later.

We finally assume that, if $\mathbb{A}$ is either $\FEspaceU$, $\FEspaceW$ or $\FEspaceM$, there is an interpolation operator
\begin{align}\label{interpolants} 
\text{I}_{\mathbb{A}} :\boldsymbol{\mathcal{C}}^0(\overline{\Omega}) \rightarrow \mathbb{A}  \,, 
\end{align}
with suitable approximation properties, that is there is a positive integer $\polydegree$ for which
\begin{align}\label{optestim}
\begin{split}
\|\text{I}_{\mathbb{A}}\lambda - \lambda\|_{\ltwos}
+ h \|\nabla(\text{I}_{\mathbb{A}}\lambda - \lambda)\|_{\ltwos} &\leq
c \, h^{\polydegree +1}  |\lambda|_{H^{\polydegree+1}} \, , \\
\|\text{I}_{\mathbb{A}}\lambda - \lambda\|_{\linfs}
+ h \|\nabla(\text{I}_{\mathbb{A}}\lambda - \lambda)\|_{\linfs}
 &\leq 
c \, h^{\polydegree +1} |\lambda|_{W_{\infty}^{\polydegree+1}} \, .
\end{split}
\end{align}
More notation and details about the space discretization will be provided in \S\ref{sub:spacedisc}. Here we confine ourselves to mention that they can be easily constructed using finite elements (see for instance \cite{Ciar78,ErnGuermond}), in which case the interpolation operators $\text{I}_{\mathbb{A}}$ are nothing but Lagrange interpolation.
Now we present a fully discrete scheme and show that it is unconditionally stable. This result will, in a sense, reproduce the formal energy estimate of Proposition~\ref{energyest}. In addition, it will serve as the basis for a proof of existence of discrete solutions in \S\ref{localsolv} as well as for a proof of convergence towards weak solutions in a simplified case in \S\ref{convschemesec}.
%
\subsection{Scheme}\label{firstschemesec}
In order to avoid unnecessary technicalities, assume that the initial data is smooth and initialize the scheme as follows: 
\begin{align}\label{initROSens}
\bv{U}^{0} = \text{I}_{\FEspaceU}[\bv{u}(0)] \ , \ \
\bv{W}^{0} = \text{I}_{\FEspaceW}[\bv{w}(0)] \ , \ \
\bv{M}^{0} = \text{I}_{\FEspaceM}[\bv{m}(0)] \ , 
\end{align}
after that, for every $k\in \{1,\ldots,K\}$ we compute $\{\bv{U}^k,P^k,\bv{W}^k,\bv{M}^k,\Phi^k\} \in \FEspaceU \times \FEspaceP \times \FEspaceW \times \FEspaceM \times \FEspacePhi$ that solves
\begin{subequations}
\label{firstschemeROS}
\begin{align}
\label{eq:discmom}
\begin{split}
\alp \tfrac{ \inc\bv{U}^{k}}\dt, \Utest \arp
&+ \tril_h\lp\bv{U}^{k},\bv{U}^{k},\Utest\rp + \nunot \lp \gradv{U}^{k},\gradv{}\Utest \rp
- \lp P^{k}, \diver{}\Utest \rp \\
&= 2 \nu_r \lp \curl{W}^{k},\Utest \rp + \mu_0 \trilm\lp \Utest, \bv{H}^{k}, \bv{M}^{k}\rp \, , 
\end{split} \\
\lp \Ptest, \diver{U}^k\rp &= 0 \, , \\
\begin{split}
\inertiamom  \alp \tfrac{\inc\bv{W}^{k}}\dt,\Wtest \arp
&+ \inertiamom \tril_h\lp\bv{U}^{k},\bv{W}^{k},\Wtest\rp 
+ c_1 \lp \gradv{W}^{k}, \nabla \Wtest \rp 
+ c_2 \lp \diver{W}^{k},\diver{}\Wtest \rp \\
&+ 4 {\nu}_r \lp \bv{W}^{k}, \Wtest \rp
= 2\nu_r \lp \curl{U}^{k}, \Wtest \rp
+ \mu_0 \lp \bv{M}^{k} \times \bv{H}^{k},\Wtest \rp \, , 
\end{split} \\
\begin{split}
\alp \tfrac{\inc\bv{M}^{k}}\dt,\Mtest \arp
&-\trilm\lp\bv{U}^{k},\Mtest,\bv{M}^{k}\rp
+ \magdiff \, a_h(\bv{M}^k,\Mtest)
+ \lp\bv{M}^k\times\bv{W}^k,\Mtest\rp \\
&+ \tfrac{1}{\chartime} \lp \bv{M}^k,\Mtest\rp
+ \magdiff \gamma (\bv{M}^k\times \normal, \bv{Z} \times \normal)_{\bdry} 
+ \magdiff \gamma (\bv{M}^k\cdot \normal, \bv{Z} \cdot \normal)_{\bdry} \\
&= \tfrac{1}{\chartime} \lp \permit \bv{H}^k,\Mtest\rp
+ \magdiff \gamma (\permit \bv{H}^k\times \normal, \bv{Z} \times \normal)_{\bdry} \\
&+ \magdiff \gamma (\permit \bv{H}^k \cdot \normal, \bv{Z} \cdot \normal)_{\bdry} \, , 
\end{split} \\
\label{DiscPoisson}
(\nabla\hdpoth^k,\nabla\Phitest) &= (\ha^k -\bv{M}^k,\nabla\Phitest) \, ,
\end{align}
\end{subequations}
for all $\Utest \in \FEspaceU$, $\Wtest \in \FEspaceW$, $\Mtest \in \FEspaceM$, $\Phitest \in \FEspacePhi$, where $\bv{H}^k := \nabla\hdpoth^k$. Notice that a discrete analogue of \eqref{eq:centralident} is built into the scheme because of the term $\mu_0 \trilm\lp \Utest, \bv{H}^{k}, \bv{M}^{k}\rp$ in the right hand side of \eqref{eq:discmom}. The initialization proposed in \eqref{initROSens} is the simplest choice and it is used because of that reason. From the point of view of convergence to strong solutions (a priori error estimates) it is suboptimal (\cf~\cite{Tom13,MR2249024,ErnGuermond,DiPi10}). However, the choice of initialization has no effect on the stability of the scheme; it only affects the regularity assumed on the initial data. 

We now present the stability of scheme \eqref{firstschemeROS}. To shorten the presentation we denote
\begin{subequations}
\label{eq:defofmatscrs}
\begin{align*}
\mathscr{E}_{h,\dt}^k(\bv{U}^\dt,\bv{W}^\dt,\bv{M}^\dt,\bv{H}^\dt) &= 
\tfrac{1}{2} \big( \|\bv{U}^k\|_{\ltwods}^2
+ \inertiamom \|\bv{W}^k\|_{\ltwods}^2
+ \mu_0 \|\bv{M}^k\|_{\ltwods}^2 
+ \mu_0 \|\bv{H}^k\|_{\ltwods}^2 \big) \, , \\
\mathscr{I}_{h,\dt}^k(\bv{U}^\dt,\bv{W}^\dt,\bv{M}^\dt,\bv{H}^\dt) &= \mathscr{E}_{h,\dt}^k(\inc \bv{U}^\dt,\inc \bv{W}^\dt,\inc \bv{M}^\dt, \inc \bv{H}^\dt), \\
\mathscr{D}_{h,\dt}^k(\bv{U}^\dt,\bv{W}^\dt,\bv{M}^\dt,\bv{H}^\dt) &= 
\nu \|\gradv{U}^k\|_{\ltwods}^2
+ c_1 \|\gradv{W}^k\|_{\ltwods}^2
+ \nu_r \|\diver{U}^k\|_{\ltwods}^2 \\
&+ c_2 \|\diver{W}^k\|_{\ltwods}^2 
+ \nu_r \|\curl{}\bv{U}^k - 2 \bv{W}^k\|_{\ltwods}^2 \\
&+ \tfrac{\mu_0}{\chartime} \|\bv{M}^k\|_{\ltwods}^2
+ \tfrac{\mu_0}{2\chartime} \left(\tfrac{1}{2} + 3 \permit \right) \|\bv{H}^k\|_{\ltwods}^2 \, , \\
\mathscr{F}^k(\ha) &= \tfrac{\chartime \mu_0}{\dt}  \int_{t^{k-1}}^{t^k}  \|\partial_t\ha(s)\|_{\ltwods}^2 \, ds +
\tfrac{\mu_0}{2\chartime} (1+\permit) \|\ha^k\|_{\ltwods}^2 \, .
\end{align*}
\end{subequations}
%
%
\begin{proposition}[discrete stability]
\label{disclemma}
Let $\magdiff = 0$, and $\lb\bv{U}^{\dt},P^{\dt},\bv{W}^{\dt},\bv{M}^{\dt},\hdpoth^{\dt} \rb \subset \FEspaceU \times \FEspaceP \times \FEspaceW \times \FEspaceM \times \FEspacePhi$ solve \eqref{firstschemeROS}. If $\nabla \FEspacePhi \subset \FEspaceM$, \eqref{discreteinfsup}-\eqref{eq:bhmskew} hold, then we have the following estimate
\begin{align*}
\mathscr{E}_{h,\dt}^K + \dt^{-1} \left\| \mathscr{I}_{h,\dt}^\dt \right\|_{\ell^1} +
  \left\| \mathscr{D}_{h,\dt}^\dt \right\|_{\ell^1}
  \leq \left\| \mathscr{F}^\dt \right\|_{\ell^1} + 
  \mathscr{E}_{h,\dt}^0 \, , 
\end{align*}
where $\mathscr{E}_{h,\dt}^k := \mathscr{E}_{h,\dt}^k(\bv{U}^\dt,\bv{W}^\dt,\bv{M}^\dt,\bv{H}^\dt)$, $\mathscr{I}_{h,\dt}^k := \mathscr{I}_{h,\dt}^k(\bv{U}^\dt,\bv{W}^\dt,\bv{M}^\dt,\bv{H}^\dt)$, $\mathscr{D}_{h,\dt}^k := \mathscr{D}_{h,\dt}^k(\bv{U}^\dt,\bv{W}^\dt,\bv{M}^\dt,\bv{H}^\dt)$ and $\mathscr{F}_{h,\dt}^k := \mathscr{F}^k(\ha)$.
\end{proposition}
\begin{proof} Set $\Utest = 2 \dt \bv{U}^k$, $\Wtest = 2 \dt \bv{W}^k$, $\Mtest = 2 \dt \mu_0 \bv{M}^k$, $\nabla\Phitest =  \tfrac{2 \dt \permit \mu_0}{\chartime}  \nabla\hdpoth^k$ in \eqref{firstschemeROS} and add the results. Using \eqref{eq:Honezdcurldiv} and the identity \eqref{sumid}, we get
\begin{align}
\label{firststepdisc}
\begin{split}
2 \inc \mathscr{E}_{h,\dt}^k 
&- \mu_0 \inc \| \bv{H}^k \|_{\ltwods}^2 
+ 2 \mathscr{I}_{h,\dt}^k 
- \mu_0 \| \inc \bv{H}^k \|_{\ltwods}^2
+ 2 \dt \mathscr{D}_{h,\dt}^k \\
&- \tfrac{\mu_0 \dt}{\chartime} \left( \tfrac{1}{2} + \permit \right) \| \nabla\hdpoth^k \|_{\ltwods}^2 
= 2 \mu_0 \dt \trilm\lp \bv{U}^k, \bv{H}^{k},\bv{M}^{k}\rp \\
&+ 2 \mu_0 \dt (\bv{M}^{k} \times \bv{H}^{k},\bv{W}^k)
+ \tfrac{2 \mu_0 \permit \dt}{\chartime} (\ha^k,\nabla\hdpoth^k) \, .
\end{split}
\end{align}
As in the proof of Proposition~\ref{continuumenergy}, to deal with the trilinear terms $2 \mu_0 \dt \trilm\lp \bv{U}^k, \bv{H}^{k},\bv{M}^{k}\rp$ and $2 \mu_0 \dt (\bv{M}^{k} \times \bv{H}^{k},\bv{W}^k)$ we set $\Mtest = 2 \mu_0 \dt\bv{H}^k$. Notice that $\bv{H}^k = \nabla \Phi^k$ is, by assumption, a valid test function. In doing so we obtain
\begin{align}
\label{identity}
\begin{split}
\tfrac{2 \mu_0 \dt \permit}{\chartime} \| \nabla\hdpoth^k \|_{\ltwods}^2 
&= - 2\mu_0 \dt \trilm\lp \bv{U}^{k}, \bv{H}^k,\bv{M}^{k} \rp
- 2\mu_0 \dt \lp\bv{M}^k\times\bv{H}^k,\bv{W}^k \rp \\
&+ 2\mu_0 \alp \inc\bv{M}^k,\nabla\hdpoth^k \arp
+ \tfrac{2\mu_0 \dt}{\chartime} \lp\bv{M}^k,\nabla\hdpoth^k \rp \, .
\end{split}
\end{align}
Adding \eqref{identity} to \eqref{firststepdisc} we obtain
\begin{align}
\label{secstepdisc}
\begin{split}
2 \inc \mathscr{E}_{h,\dt}^k
&- \mu_0 \inc \| \bv{H}^k \|_{\ltwods}^2
+ 2 \mathscr{I}_{h,\dt}^k
- \mu_0 \| \inc \bv{H}^k \|_{\ltwods}^2
+ 2 \dt \mathscr{D}_{h,\dt}^k  \\
&+ \tfrac{\mu_0 \dt}{\chartime} \left( \permit - \tfrac{1}{2} \right) \| \nabla\hdpoth^k \|_{\ltwods}^2 
= 2\mu_0 \alp \inc\bv{M}^k,\nabla\hdpoth^k \arp \\
&+ \tfrac{2\mu_0\dt}{\chartime} \lp\bv{M}^k,\nabla\hdpoth^k \rp
+ \tfrac{2\mu_0\permit \dt}{\chartime} (\ha^k,\nabla\hdpoth^k) \, .
\end{split}
\end{align}
We now must obtain discrete analogues of \eqref{phiID00} and \eqref{phiID01}. To do so, we set $\Phitest = \Phi^k$ to obtain
\begin{align}
\label{phiID00disc}
  \|\nabla\hdpoth^k \|_{\ltwods}^2 =  (\ha^k - \bv{M}^k, \nabla\hdpoth^k) \, .
\end{align}
Taking increments on \eqref{DiscPoisson} and setting $\Phitest = \Phi^k$ yields
\begin{align}
\label{phiID01disc}
  (\inc\nabla\hdpoth^k, \nabla\hdpoth^k) = (\inc\ha^k - \inc\bv{M}^k, \nabla\hdpoth^k) \, .
\end{align}
Adding suitable multiples of \eqref{phiID00disc} and \eqref{phiID01disc} to \eqref{secstepdisc}, and dividing everything by 2, we obtain
\begin{align}
\label{prevestimates}
\begin{split}
\mathscr{E}_{h,\dt}^k + \mathscr{I}_{h,\dt}^k + \dt \mathscr{D}_{h,\dt}^k 
+ \tfrac{\mu_0 \dt}{2 \chartime} \left( \tfrac{3}{2} + \permit \right) \| \nabla\hdpoth^k \|_{\ltwods}^2 
&= \mathscr{E}_{h,\dt}^{k-1} +
\mu_0 \alp \inc\ha^k,\nabla\hdpoth^k \arp \\
&+ \tfrac{\mu_0 \dt}{\chartime} (1+ \permit )\lp \ha^k,\nabla\hdpoth^k \rp \, .
\end{split}
\end{align}
Adding over $k$, using the trivial identity $\|\inc\ha^k\|_{\ltwods}^2 = \dt^2 \LN\tfrac{\inc\ha^k}{\dt}\RN_{\ltwods}^2$ and the estimate
\begin{align*}
\LN\tfrac{\inc\ha^k}{\dt}\RN_{\ltwods}^2 \leq \tfrac{1}{\dt} \int_{t^{k-1}}^{t^k} \|\partial_t \ha(s)\|_{\ltwods}^2 ds \, ,
\end{align*}
to control $\dt^{-1}\LN \inc \ha^k \RN_{\ltwods}$, yields the asserted estimate.
\end{proof}
%
%
\subsection{Practical space discretization}
\label{sub:spacedisc}
Having understood what is required from scheme \eqref{firstschemeROS} to achieve stability we will now specify our choices using finite elements. We assume that we have at hand a conforming and shape regular triangulation $\triangulation$ of the polygonal domain $\Omega$, made of open disjoint elements $\element$ such that $\overline\Omega = \bigcup_{\element \in \triangulation} \overline\element$. We will denote by $\mathcal{F}^i$ the collection of internal faces $F$ of $\triangulation$. As Proposition~\ref{disclemma} shows, to gain stability it is convenient to have $\nabla \FEspacePhi \subset \FEspaceM$. Since the space $\FEspacePhi$ is used to approximate the solution of an elliptic problem with Neumann boundary conditions, the simplest choice for $\FEspacePhi$ is a finite element space of continuous functions
\begin{align}
\label{choice1}
\FEspacePhi &= \lb \Phitest \in \mathcal{C}^0\bigl(\overline\Omega\bigr) \ | \ \Phitest|_{T} \in \simplex_{\polydegree}(\element) \, , \forall \element \in \triangulation \rb  \subset \hone \, .
\end{align}
To achieve $\nabla\FEspacePhi \subset \FEspaceM$ we allow $\FEspaceM$ to be a space of discontinuous functions
\begin{align}
\label{choice2}
\FEspaceM = \lb \bv{M} \in \ltwods(\Omega) \ | \ \bv{M}|_{T} \in [\simplex_{\polydegree-1}(\element)]^d \, , \forall \, \element \in \triangulation \rb \, ,
\end{align}
and, consequently, the forms $\trilm(\cdot, \cdot, \cdot)$ and $a_h(\cdot, \cdot)$ must be defined accordingly. Here $\simplex_{\polydegree}$ denotes the space of polynomials of total degree at most $\polydegree$, usually associated to simplicial elements. 

The trilinear form $\trilm$ is defined by
\begin{align}
\label{trilineardef}
\begin{split}
\trilm(\bv{U},\bv{V},\bv{W})
:= \sum_{\element \in \triangulation} \bulkint{\element}{ (\bv{U} \cdot \nabla )\bv{V} \cdot \bv{W}
    + \tfrac{1}{2} \diver{U}\bv{V}\cdot\bv{W} } -  \sum_{F \in \mathcal{F}^i}
\bdryint{F}{ ( \lj\bv{V}\rj \cdot \lbb\bv{W}\rbb)(\bv{U} \cdot \normal_F) } \, ,
\end{split}
\end{align}
the bulk integrals are the classical Temam \cite{Temam} modification of the convective term \eqref{eq:defoftrilb}, whereas the face integrals are consistency terms. This discretization of convection for discontinuous spaces traces back to \cite{DiErn2012,DiPi10,Gir2005,Les1974}. From these references it is known that, provided the first argument ($\bv{U}$ in \eqref{trilineardef}) is $\hdiv$ conforming and has a vanishing normal trace on the boundary, $\trilm(\cdot,\cdot, \cdot)$ is skew symmetric, that is \eqref{eq:bhmskew} holds.

The bilinear form $a_h(\cdot, \cdot)$ (note that $a_h(\cdot, \cdot)$ was not used in Proposition \ref{disclemma} since we considered the case $\magdiff = 0$) is obtained with interior penalty techniques
\begin{align}
\label{bilineardec}
a_h(\bv{M},\Mtest) = \langle \bv{M}, \Mtest \rangle_{\hcurlh}
+ \langle \bv{M}, \Mtest \rangle_{\hdivh} \, ,
\end{align}
where $\langle \cdot, \cdot \rangle_{\hcurlh}$ and $\langle \cdot, \cdot \rangle_{\hdivh}$ are interior penalty discretizations of $(\curl{\cdot},\curl{\cdot})$ and $(\diver{\cdot},\diver{\cdot})$, respectively
\allowdisplaybreaks\begin{align}
\label{semicurlh}
\begin{split}
\langle \bv{M}, \Mtest \rangle_{\hcurlh} &= \sum_{\element \in \triangulation}
\bulkint{\element}{\curl{M} \cdot \curl{}\Mtest } \\
&- \sum_{F \in \mathcal{F}^i} \bdryint{F}{(\lbb\curl{M}\rbb \! \times \! \normal_F) \cdot \lj\Mtest \rj + (\lbb\curl{}\Mtest\rbb \!\times \! \normal_F) \cdot \lj\bv{M} \rj }  \\
&+ \sum_{F \in \mathcal{F}^i} \tfrac{\eta}{h_F} \bdryint{F}{ (\lj \bv{M} \rj \! \times \! \normal_F)  \cdot (\lj \Mtest \rj \! \times \! \normal_F)} \, , 
\end{split} \\
\label{semidivh}
\begin{split}
\langle \bv{M}, \Mtest \rangle_{\hdivh} &= \sum_{\element \in \triangulation}
\bulkint{\element}{\diver{M} \, \diver{}\Mtest } \\
&- \sum_{F \in \mathcal{F}^i} \bdryint{F}{\lbb\diver{M}\rbb (\lj \Mtest \rj \cdot \normal_F)
+ \lbb\diver{}\Mtest \rbb (\lj \bv{M} \rj \cdot \normal_F)} \, \\
&+ \sum_{F \in \mathcal{F}^i} \tfrac{\eta}{h_F} \bdryint{F}{ (\lj \bv{M} \rj \cdot \normal_F ) (\lj \Mtest \rj \cdot \normal_F)} \, .
\end{split}
\end{align}
As usual, the parameter $\eta>0$ must be chosen large enough to yield coercivity \eqref{positivityprop} (\cf~\cite{DiPi10}). In this setting the discrete semi-norm $|\cdot|_{a}$ is defined as 
\begin{align*}
|\bv{M} |_{a}^2 = \sum_{\element \in \triangulation} \int_\element \big( |\curl{M}|^2 + |\diver{M}|^2 \big)
+ \sum_{F \in \mathcal{F}^i} \tfrac{1}{h_F} \int_F  \left( | \lj \bv{M} \rj \cdot \normal_F  |^2 + | \lj \bv{M} \rj \times \normal_F |^2 \right)  \, ,
\end{align*}
The choice of the remaining spaces is straightforward. The only restriction (for the scheme \eqref{firstschemeROS}) is that the pair $\lb\FEspaceU,\FEspaceP\rb$, used to approximate the linear velocity and pressure, must be LBB stable (see \eqref{discreteinfsup}). Therefore, for $\polydegree \geq 2$ we set
{\allowdisplaybreaks\begin{align}
\label{FEspaces}
\begin{aligned}
\FEspaceU &= \lb \bv{U} \in \bm{\mathcal{C}}^0\bigl(\overline\Omega\bigr) \ |
\ \bv{U}|_{T} \in [\simplex_{\polydegree}(\element)]^d \, , \forall \, \element \in \triangulation \rb \cap \hzerod \, , \\
\FEspaceP &= \lb \Ptest \in {\mathcal{C}}^0\bigl(\overline\Omega\bigr) \ |
\ \Ptest|_{T} \in \simplex_{{\polydegree-1}}(\element) \, , \forall \, \element \in \triangulation \rb \, , \\
\FEspaceW &= \lb \bv{W} \in \bm{\mathcal{C}}^0\bigl(\overline\Omega\bigr) \ | \ \bv{W}|_{T} \in [\simplex_{\polydegree}(\element)]^d \, , \forall \, \element \in \triangulation \rb \cap \hzerod \, . 
%
\end{aligned}
\end{align}}
It is well known that the pair $\lb\FEspaceU,\FEspaceP\rb$ in \eqref{FEspaces} (\cf~\cite{Girault,ErnGuermond}) is LBB stable for $\ell\geq 2$ under minor restrictions of the mesh $\triangulation$. Note also in \eqref{FEspaces}, that we are using a continuous finite element space $\FEspaceP$ for the pressure, which is something we might have to change (the velocity space too) if we want to consider convergence of a numerical scheme under minimal regularity assumptions. The use of discontinuous pressures will be considered later in \S\ref{simpferro1} for a slightly different model.  

In \eqref{choice1}, \eqref{choice2} and \eqref{FEspaces} we have used polynomials $\simplex_{\polydegree}$ over simplices. However, the fact that the scheme \eqref{firstschemeROS} is energy-stable is not tightly related to simplices. In \eqref{choice1}, \eqref{choice2} and \eqref{FEspaces}, it is possible to replace $\simplex_{\polydegree}$ by $\quadrilateral_{\polydegree}$ (polynomials of degree at most $\polydegree$ in each variable) and use quadrilateral/hexahedral elements. That will only require us to do some minor modifications of the scheme. To simplify our exposition we will always assume that our elements are simplicial and develop our theory under this assumption. We will provide remarks describing the required modifications, if any, when quadrilaterals are to be used. With this choice of spaces, the scheme presents itself as a generalization of those studied in \cite{Tom13,AJS2013}.
\begin{remark}[redefinition of the pressure]\label{remredefpress} Given $\varpi \in \{0,1\}$, let $(\bv{u},p) \in \hzerod \times \lzerotwo$ solve
\begin{alignat*}{2}
(\nabla\bv{u},\nabla\bv{v}) - (p,\diver{v}) &= (\bv{f},\bv{v}) + \varpi (g,\diver{v})   && \ \ \forall \bv{v} \in \hzerod \, , \\
(q,\diver{u}) &= 0  &&\ \ \forall q \in \lzerotwo \, , 
\end{alignat*}
Since the pressure can be redefined as $\widehat{p} = p + \varpi g$, the velocity $\bv{u}$ is independent of $\varpi$. This has been used to devise energy-stable schemes for phase field models \cite{LiuShen2003} and liquid crystals \cite{LinLiu2007}. In the same spirit we have eliminated the term $- \tfrac{\mu_0}{2} (\diver{V},\bv{M}^k\cdot\bv{H}^k)$ from the Kelvin force (in \eqref{eq:discmom}).
\end{remark}
%
%
\subsection{Existence of solutions for $\magdiff=0$}
\label{localsolv}
The energy estimate of Proposition~\ref{disclemma} (discrete stability) serves as an a priori estimate of solutions of \eqref{firstschemeROS}. This estimate can be used to establish, with the help of Leray-Schauder theorem (\cf~\cite[p. 280]{GT2001} ), existence of solutions. The core of the following result is a local in time energy estimate similar to that of Proposition~\ref{disclemma}. To avoid repetitions we skip some details.

\begin{theorem}[existence for $\magdiff = 0$]\label{localexist} Let $h,\dt > 0$. For every $k \in \{1,\ldots,K\}$, the scheme \eqref{firstschemeROS} has a solution. Moreover, this solution satisfies the estimate of Proposition~\ref{disclemma}.
\end{theorem}
\begin{proof} We define the linear map $\widehat{x} = \mathcal{L} x$, 
\begin{align*}
\lb\bv{U}^k,P^k,\bv{W}^k,\bv{M}^k,\nabla\hdpoth^k \rb \xmapsto{\mathcal{L}} 
\lb \widehat{\bv{U}}^k,\widehat{P}^k,\widehat{\bv{W}}^k,\widehat{\bv{M}}^k,\nabla\widehat{ \hdpoth}^k \rb \, , 
\end{align*}
where the quantities with hats solve:
\begin{subequations}
\label{eq:iterat0}
\begin{align}
\begin{split}\label{iteratU0}
\alp \tfrac{ \widehat{\bv{U}^{k}}-\bv{U}^{k-1}}\dt, \Utest \arp
&+ \tril_h\lp\bv{U}^{k},\widehat{\bv{U}}^{k},\Utest\rp 
+ \nunot \lp \gradv{}\widehat{\bv{U}}^{k},\gradv{}\Utest \rp \\
&- \lp \widehat{P}^{k}, \diver{}\Utest \rp 
= 2 \nu_r \lp \curl{}\widehat{\bv{W}}^{k},\Utest \rp + \mu_0 \trilm\lp \Utest, \bv{H}^{k},\widehat{\bv{M}}^{k}\rp \, , 
\end{split} \\
\label{iteratIncomp0}
\lp \Ptest, \diver{}\widehat{\bv{U}}^k\rp &= 0 \, , \\
\begin{split}\label{iteratW0}
\inertiamom  \alp \tfrac{ \widehat{\bv{W}}^{k} - \bv{W}^{k-1}}\dt,\Wtest \arp
&+ \inertiamom \tril_h\lp\bv{U}^{k},\widehat{\bv{W}}^{k},\Wtest\rp 
+ c_1  \lp \gradv{}\widehat{\bv{W}}^{k}, \nabla \Wtest \rp 
+ c_2 \lp \diver{}\widehat{\bv{W}}^{k},\diver{}\Wtest \rp \\
&+ 4 {\nu}_r  \lp \widehat{\bv{W}}^{k}, \Wtest \rp 
= 2\nu_r \lp \curl{U}^{k}, \Wtest \rp 
+ \mu_0 \lp \widehat{\bv{M}}^{k} \times \bv{H}^{k},\Wtest \rp \, ,
\end{split} \\
\begin{split}\label{iteratM0}
\alp \tfrac{\widehat{\bv{M}}^{k} - \bv{M}^{k-1}}\dt,\Mtest \arp
&- \trilm\lp\bv{U}^{k},\Mtest,\widehat{\bv{M}}^{k}\rp
+ \lp \widehat{\bv{M}}^k \times \bv{W}^k, \Mtest \rp 
+ \tfrac{1}{\chartime} \lp \widehat{\bv{M}}^k, \Mtest \rp \\
&= \tfrac{\permit}{\chartime} \lp \widehat{\bv{H}}^{k}, \Mtest \rp \, , 
\end{split} \\
\label{iteratPhi0}
( \nabla\widehat{\hdpoth}^k,\nabla\Phitest) &= (\ha^k - \widehat{\bv{M}}^k,\nabla\Phitest) \, ,
\end{align}
\end{subequations}
where $\widehat{\bv{H}}^{k} = \nabla\widehat{\hdpoth}^k$. To assert the existence of a solution we show that the map $\mathcal{L}$ satisfies the requirements of the Leray-Schauder theorem \cite[p. 280]{GT2001}:
\begin{itemize}[leftmargin=20pt] 
\setlength{\itemsep}{2pt} 
\setlength{\parskip}{0cm}
\item[\itemizebullet] \textbf{Well posedness.} The operator $\mathcal{L}$ is clearly well 
defined. The information follows a bottom-up path, so we start with the coupled system \eqref{iteratM0}-\eqref{iteratPhi0} on the bottom, which can be conveniently rewritten as follows:
\begin{align*}
\alp \widehat{\bv{M}}^{k},\Mtest \arp 
- \dt \trilm\lp\bv{U}^{k},\Mtest,\widehat{\bv{M}}^{k}\rp 
+ \dt \lp\widehat{\bv{M}}^k\times\bv{W}^k,\Mtest\rp 
+ \tfrac{\dt}{\chartime} \lp\widehat{\bv{M}}^k,\Mtest\rp - \tfrac{\dt\permit}{\chartime} \lp \widehat{\bv{H}}^{k}, \Mtest \rp
&= \alp \bv{M}^{k-1},\Mtest \arp \\
(\widehat{\bv{M}}^k,\nabla\Phitest) 
+ ( \nabla \widehat{\hdpoth}^k,\nabla\Phitest)
&= (\ha^k,\nabla\Phitest) \, .
\end{align*}
Set $\Mtest = \widehat{\bv{M}}^{k}$ and $\Phitest = \tfrac{\dt \permit}{\chartime} \widehat{\hdpoth}^k$, to show that this system is positive definite. This yields the functions $\widehat{\bv{M}}^k$ and $\widehat{\hdpoth}^k$, which can be used as data in \eqref{iteratW0} to obtain $\widehat{\bv{W}}^{k}$. Inserting $\widehat{\bv{W}}^{k}$ and $\widehat{\bv{M}}^{k}$ into \eqref{iteratU0}--\eqref{iteratIncomp0} gives rise to the pair $(\widehat{\bv{U}^{k}},\widehat{P^{k}})$.

\item[\itemizebullet] \textbf{Boundedness.} We must verify that solutions $\widehat{x} = \lb \widehat{\bv{U}}^k,\widehat{P}^k,\widehat{\bv{W}}^k,\widehat{\bv{M}}^k,\nabla\widehat{ \hdpoth}^k \rb$ of $\tfrac{1}{\lambda} \widehat{x} = \mathcal{L}\widehat{x}$ with $\lambda \in (0,1]$ can be bounded in terms of the local data $\lb \bv{U}^{k-1},\bv{W}^{k-1},\bv{M}^{k-1},\nabla\hdpoth^{k-1}, \ha^k \rb$ uniformly with respect to $\lambda$. In other words, we want to analyze the local boundedness of
{\allowdisplaybreaks\begin{subequations}
\label{eq:iterat}
\begin{align}
\begin{split}\label{iteratU}
\alp \tfrac{ \lambda^{-1} \, \widehat{\bv{U}^{k}}-\bv{U}^{k-1}}\dt, \Utest \arp
&+ \tril_h\lp\widehat{\bv{U}}^{k},\lambda^{-1} \, \widehat{\bv{U}}^{k},\Utest\rp  \\
&+ \nunot \lp \lambda^{-1} \, \gradv{}\widehat{\bv{U}}^{k},\gradv{}\Utest \rp 
- \lp \lambda^{-1} \, \widehat{P}^{k}, \diver{}\Utest \rp \\
&= 2 \nu_r \lp \lambda^{-1} \, \curl{}\widehat{\bv{W}}^{k},\Utest \rp + \mu_0 \trilm\lp \Utest, \widehat{\bv{H}}^{k}, \lambda^{-1} \, \widehat{\bv{M}}^{k} \rp \, ,
\end{split} \\
\label{iteratIncomp}
\lp \Ptest, \diver{}\widehat{\bv{U}}^k\rp &= 0 \, , \\
\begin{split}\label{iteratW}
\inertiamom  \alp \tfrac{\lambda^{-1} \, \widehat{\bv{W}}^{k} - \bv{W}^{k-1}}\dt,\Wtest \arp
&+ \inertiamom \tril_h\lp\widehat{\bv{U}}^{k},\lambda^{-1} \, \widehat{\bv{W}}^{k},\Wtest\rp
+ c_1 \lp \lambda^{-1} \, \gradv{}\widehat{\bv{W}}^{k}, \nabla \Wtest \rp \\
&+ c_2 \lp \lambda^{-1} \, \diver{}\widehat{\bv{W}}^{k},\diver{}\Wtest \rp
+ 4 {\nu}_r  \lp \lambda^{-1} \, \widehat{ \bv{W}}^{k}, \Wtest \rp \\
&= 2\nu_r \lp \curl{}\widehat{\bv{U}}^{k}, \Wtest \rp
+ \mu_0 \lp  \lambda^{-1} \, \widehat{\bv{M}}^{k} \times \widehat{\bv{H}}^{k},\Wtest \rp \, ,
\end{split} \\
\begin{split}\label{iteratM}
\alp \tfrac{\lambda^{-1} \, \widehat{\bv{M}}^{k} - \bv{M}^{k-1}}\dt,\Mtest \arp
&- \trilm\lp\widehat{\bv{U}}^{k},\Mtest, \lambda^{-1} \, \widehat{\bv{M}}^{k} \rp
+ \lp \lambda^{-1} \, \widehat{\bv{M}}^k \times \widehat{\bv{W}}^k, \Mtest \rp  \\
&+ \tfrac{1}{\chartime} \lp \lambda^{-1} \, \widehat{\bv{M}}^k, \Mtest \rp
= \tfrac{\permit}{\chartime} \lp \lambda^{-1} \, \widehat{\bv{H}}^{k}, \Mtest \rp
\end{split} \\
\label{iteratPhi}
 (\lambda^{-1} \, \nabla\widehat{\hdpoth}^k,\nabla\Phitest) &= (\ha^k - \lambda^{-1} \, \widehat{\bv{M}}^k,\nabla\Phitest) \, , 
\end{align}
\end{subequations}} \\
with $\lambda \in (0,1] $. Set $\Utest = 2 \dt \lambda \widehat{\bv{U}}^k$, $\Wtest = 2 \dt \lambda \widehat{\bv{W}}^k$, $\Mtest = 2 \dt \mu_0 \lambda \widehat{\bv{M}}^k$, $\Phitest =  \tfrac{2 \dt \permit \mu_0 \lambda}{\chartime}  \widehat{\hdpoth}^k$ in \eqref{iteratU}-\eqref{iteratPhi}, and use identity \eqref{sumid}. As we did to obtain \eqref{identity}, set $\Mtest = 2 \mu_0 \dt \lambda \nabla\widehat{\hdpoth}^k$ in \eqref{iteratM} (to eliminate the trilinear terms). Adding the ensuing equations, and eliminating superfluous positive terms, we obtain:
\begin{align}
\label{iterativeenergy2}
\begin{aligned}
\|\widehat{\bv{U}}^k\|_{\ltwods}^2 &- \|\lambda\bv{U}^{k-1}\|_{\ltwods}^2
+ \inertiamom \|\widehat{\bv{W}}^k\|_{\ltwods}^2 - \inertiamom \|\lambda \bv{W}^{k-1}\|_{\ltwods}^2 
+ \mu_0 \|\widehat{\bv{M}}^k\|_{\ltwods}^2\\
&- \mu_0 \|\lambda \bv{M}^{k-1}\|_{\ltwods}^2  
+ 2 \, (\nu + \nu_r ) \, \dt \, \|\gradv{}\widehat{\bv{U}}^k\|_{\ltwods}^2
+ 2 c_1 \dt \|\gradv{}\widehat{\bv{W}}^k\|_{\ltwods}^2 \\
&+ 2 c_2 \dt  \|\diver{}\widehat{\bv{W}}^k\|_{\ltwods}^2
+ 8 \dt \nu_r \|\widehat{\bv{W}}^k\|_{\ltwods}^2
+ \tfrac{2 \mu_0 \dt}{\chartime} \|\widehat{\bv{M}}^k\|_{\ltwods}^2 \\
&+\tfrac{4\mu_0 \permit \dt}{\chartime} \| \nabla\widehat{\hdpoth}^k \|_{\ltwods}^2 \leq 4 \nu_r \dt (1+ \lambda) (\curl{}\widehat{\bv{U}}^k,\widehat{\bv{W}}^k) \\
&+ 2 \mu_0 (\widehat{\bv{M}}^k - \lambda \bv{M}^{k-1}, \nabla\widehat{\hdpoth}^k) 
+ \tfrac{2 \mu_0 \dt}{\chartime} (\widehat{\bv{M}}^k,\nabla\widehat{\hdpoth}^k) \\
&+ \tfrac{2 \mu_0 \permit \dt \lambda}{\chartime} (\ha^k,\nabla\widehat{\hdpoth}^k) \, .
\end{aligned}
\end{align}
Set $\Phitest = \lambda \widehat{\hdpoth}^k$ in \eqref{iteratPhi} to obtain
$\|\nabla\widehat{\hdpoth}^k \|_{\ltwods}^2 =  (\lambda \ha^k - \widehat{\bv{M}}^k, \nabla\widehat{\hdpoth}^k).$ Consequently, \eqref{iterativeenergy2} can be rewritten as
\begin{align}
\begin{aligned}
\|\widehat{\bv{U}}^k\|_{\ltwods}^2 &- \|\lambda\bv{U}^{k-1}\|_{\ltwods}^2
+ \inertiamom \|\widehat{\bv{W}}^k\|_{\ltwods}^2 - \inertiamom \|\lambda \bv{W}^{k-1}\|_{\ltwods}^2 
+ \mu_0 \|\widehat{\bv{M}}^k\|_{\ltwods}^2 \\
&- \mu_0 \|\lambda \bv{M}^{k-1}\|_{\ltwods}^2  
+ 2 \, (\nu + \nu_r ) \, \dt \, \|\gradv{}\widehat{\bv{U}}^k\|_{\ltwods}^2 
+ 2 c_1 \dt \|\gradv{}\widehat{\bv{W}}^k\|_{\ltwods}^2 \\
&+ 2 c_2 \dt  \|\diver{}\widehat{\bv{W}}^k\|_{\ltwods}^2 
+ 8 \dt \nu_r \|\widehat{\bv{W}}^k\|_{\ltwods}^2
+ \tfrac{2 \mu_0 \dt}{\chartime} \|\widehat{\bv{M}}^k\|_{\ltwods}^2 
 \\ 
&+ 2 \mu_0 \alp \tfrac{2 \permit \dt}{\chartime} + \tfrac{\dt}{\chartime} + 1 \arp \|\nabla\widehat{\hdpoth}^k\|_{\ltwods}^2 
\leq 4 \nu_r \dt (1+ \lambda) (\curl{}\widehat{\bv{U}}^k,\widehat{\bv{W}}^k) \\
&- 2 \mu_0 \lambda( \bv{M}^{k-1}, \nabla\widehat{\hdpoth}^k)
+ 2 \mu_0 \lambda \alp \tfrac{\permit \dt}{\chartime} + \tfrac{\dt}{\chartime} + 1 \arp 
(\ha^k,\nabla\widehat{\hdpoth}^k ) \, . 
\end{aligned}
\end{align}
To conclude it suffices to suitably bound the right-hand side. This can be easily attained by recalling that $\lambda \leq 1$ and using \eqref{eq:Honezdcurldiv}.

\item[\itemizebullet] \textbf{Compactness.} Compactness of the linear operator $\mathcal{L}$ is immediate, since we are working in finite dimensions.
\end{itemize}
These three properties enable us to a apply the Leray-Schauder theorem and conclude the proof.
\end{proof}
%
%
\subsection{Lack of stability for $\magdiff>0$}
\label{failuresigmapos}
Since $\bv{H}^k = \nabla\hdpoth^k \in \hcurl$ we have that
\begin{align}
\label{hcurlprop}
\begin{gathered}
\curl{H}^k \Big|_{\element} = 0 \ \  \forall \element \in \triangulation  \ ,  \ \
\lbb\curl{}\bv{H}^k\rbb \!\! \Big|_{F} = 0 \ \  \forall F \in \mathcal{F}^i  \ ,   \ \ 
\lj \bv{H}^k \rj \times \normal_F  \Big|_{F} = 0 \ \ \forall F \in \mathcal{F}^i.
\end{gathered}
\end{align}
From these properties and definition \eqref{semicurlh} we deduce that $\langle \bv{M}^k, \nabla\hdpoth \rangle_{\hcurlh} = 0$.

If $\magdiff>0$ this result can be used to attempt to obtain an energy estimate for the scheme \eqref{firstschemeROS}. Set $\Utest = 2 \permit \dt \bv{U}^k$, $\Wtest = 2 \dt \permit \bv{W}^k$, $\Mtest = 2 \dt \mu_0 \bv{M}^k$, $\Phitest =  \tfrac{2 \dt \permit \mu_0}{\chartime} \hdpoth^k$ and $\Mtest =  \tfrac{2 \dt \permit \mu_0}{\chartime} \nabla\hdpoth^k$. Following Proposition~\ref{disclemma} we get
{\allowdisplaybreaks
\begin{align}
\label{lastfailedestimate}
\begin{split}
\inc \big(\permit \|\bv{U}^k\|_{\ltwods}^2 
&+ \permit \inertiamom \|\bv{W}^k\|_{\ltwods}^2 
+ \mu_0 \|\bv{M}^k\|_{\ltwods}^2
+ \permit \mu_0 \|\nabla\hdpoth^k\|_{\ltwods}^2
\big) \\
&+ \permit \|\inc\bv{U}^k\|_{\ltwods}^2 
+ \permit \inertiamom \|\inc\bv{W}^k\|_{\ltwods}^2
+ \mu_0 \|\inc\bv{M}^k\|_{\ltwods}^2 
+ \permit \mu_0 \|\inc\nabla\hdpoth^k\|_{\ltwods}^2 \\
&+ 2 \permit \dt \Big(\nu \|\gradv{U}^k\|_{\ltwods}^2 
+ c_1 \|\gradv{W}^k\|_{\ltwods}^2 
+ \nu_r \|\diver{U}^k\|_{\ltwods}^2 \\
&+ c_2 \|\diver{W}^k\|_{\ltwods}^2
+ \nu_r \|\curl{U}^k - 2 \bv{W}^k\|_{\ltwods}^2 
+ \tfrac{\mu_0}{\chartime\permit} \|\bv{M}^k\|_{\ltwods}^2 \\
&+ \tfrac{\mu_0}{\chartime} (2 + \permit) \| \nabla\hdpoth^k \|_{\ltwods}^2 \Big)
+ 2 \magdiff \dt \Big( \mu_0 \langle \bv{M}^k,\bv{M}^k \rangle_{\hdivh} \\
&+ \mu_0 \langle \bv{M}^k,\bv{M}^k \rangle_{\hcurlh}
+ \gamma \|(\bv{M}^k - \permit \nabla\hdpoth^k)\times \normal\|_{\bv{L}^2(\bdry)}^2 \\
&+ \gamma \|(\bv{M}^k - \permit \nabla\hdpoth^k)\cdot \normal\|_{L^2(\bdry)}^2 \Big) 
= 2\mu_0 \permit \alp \inc\ha^k,\nabla\hdpoth^k \arp \\
&+ \tfrac{2\permit\mu_0 \dt}{\chartime} \lp \ha^k,\nabla\hdpoth^k \rp
+ 2 \mu_0  \magdiff \dt \, \langle \bv{M}^k, \nabla\hdpoth^k \rangle_{\hdivh}.
\end{split}
\end{align}}
Notice that we have gained control on the needed boundary terms. However, there is no way to control the last term $2 \mu_0 \dt \magdiff \, \langle \bv{M}^k, \nabla\hdpoth^k \rangle_{\hdivh}$ unless one can reproduce the identity $(\diver{m},\diver{}\heff ) = - \|\diver{}\heff\|_{\ltwods}^2$ (used in \S\ref{energyest} for the identity \eqref{eq:vectlapint}) in a discrete setting (or at least, prove that $\langle \bv{M}^k, \nabla\hdpoth^k \rangle_{\hdivh} \leq 0$). 

\section{Simplified ferrohydrodynamics and convergent scheme}
\label{simpferro1}
We now simplify the ferrofluid model, discretize it, and study stability and convergence of the discrete scheme.

\subsection{Simplification}

We can simplify the model defined by \eqref{eq:ferroeq}, \eqref{eq:icFERRO}, \eqref{eq:bcMNSE}, \eqref{eq:phiNeuIII} and $\magdiff = 0$ by eliminating the magnetostatics problem \eqref{eq:phiNeuIII}, and setting the effective magnetizing field to be $\heff := \ha$. The purpose of this section is to explain, at least with a heuristic argument, under which circumstances this is a reasonable physical approximation. 

As we know from \S\ref{hdbcs}, $\hdpot$ is the sum of two potentials $\hdpot = \hdpoto + \hapot$ (see Remark \ref{rem00}), so that $\heff = \hd + \ha = \nabla\hdpoto + \nabla\hapot$. In this context one may ask under which circumstances we can neglect the contribution of the demagnetizing field $\hd$ and assert that $\heff \approx \ha$ is a good approximation. In other words, we want to find an estimate for the difference $\nabla\hdpoto = \nabla\hdpot-\nabla\hapot$. For this purpose we set $\phitest = \hdpoto$ in \eqref{varforphi01}, and use Cauchy-Schwarz inequality to obtain 
\begin{align}\label{varforphi04}
\|\nabla\hdpoto \|_{\ltwods} \leq \|\bv{m}\|_{\ltwods} \, .
\end{align}
In conclusion, $\heff \approx \ha$ whenever the magnetization $\bv{m}$ is small. On the other hand, as explained in \S\ref{sec:Rosensweig}, the evolution of the magnetization is such that $\bv{m} \approx \permit \heff$ when close to equilibrium. Thus, if $\permit < < 1$ the magnetization $\bv{m}$ will be small, so that we can neglect the contribution of the demagnetizing field to the total magnetic field as suggested by \eqref{varforphi04}. Water based ferrofluids subject to slowly varying magnetic fields could be modeled under these assumptions, since they usually exhibit a small magnetic susceptibility $\suscep$ \cite{Pak2006,FerrotecWebpage}. It is worth mentioning that the simplification $\heff := \ha$ is not particularly new: it has been used for analytic computations of the Rosensweig model and still retains a significant amount of valid quantitative information as shown for instance in \cite{Rinal02,Zahn95,Sunil2007}; it has also been suggested in the analysis of stationary 
configurations of free surfaces of ferrofluids \cite{Tob2006}. 
%
\subsection{Ultra weak formulation of simplified ferrohydrodynamics}
\label{simpferro}
We will consider the following weak formulation for the model defined by equations \eqref{eq:ferroeq}-\eqref{eq:bcMNSE} with $\magdiff = 0$: Find $(\bv{u},\bv{w},\bv{m})\in L^2(0,\tf,\Vspace) \times L^2(0,\tf,\hzerod) \times L^2(0,\tf,\ltwod)$ that satisfy
\begin{subequations}
\label{weakformulation}
\begin{align}
\begin{split}\label{SimpNS} %
\int_{0}^{\tf} - ( \bv{u},\utest_t ) + \tril(\bv{u},\bv{u},\utest) + \nunot ( \nabla\bv{u},\nabla\utest) &= 
( \bv{u}(0),\utest(0) )
+ \int_{0}^{\tf} \mu_0 \tril( \bv{m},\heff, \utest)  \, , 
\end{split} \\
\begin{split}\label{SimpSpin}
\int_{0}^{\tf} - \inertiamom ( \bv{w},\wtest_t ) + \tril(\bv{w},\bv{w},\wtest) + c_1 ( \nabla\bv{w},\nabla\utest) & \\
+ c_2 (\diver{w},\diver{}\wtest) + 4 \nu_r (\bv{w},\wtest) &= \inertiamom ( \bv{w}(0),\wtest(0) ) 
+ \int_{0}^{\tf} 2 \nu_r (\curl{u}, \wtest) + \mu_0 (\bv{m}\times\heff) \, , 
\end{split} \\
\label{SimpMag} 
- \int_{0}^{\tf} (\bv{m},\mtest_t) + \tril(\bv{u},\mtest,\bv{m})
- \tfrac{1}{\chartime} (\bv{m},\mtest) &=
 (\bv{m}(0),\mtest(0)) 
+ \tfrac{\permit}{\chartime} \int_{0}^{\tf} ( \heff,\mtest) \, , 
\end{align}
\end{subequations}
for all $\utest \in \lb \utest\in \boldsymbol{\mathcal{C}}_0^{\infty}([0,\tf) \times \Omega) \, | \, \diver{}\utest = 0 \text{ in } \Omega \, \rb$,
$\wtest,\mtest \in \boldsymbol{\mathcal{C}}_0^{\infty}([0,\tf) \times \Omega)$, where now the magnetic field $\heff$ is not determined by the Poisson problem \eqref{eq:phiNeuIII}, but rather $\heff := \ha$ is a given harmonic smooth vector field. 
%
%
\subsection{Additional definitions}
Some additional definitions:
\begin{itemize}
\item[\itemizebullet] We will denote $\FEspaceUdivfree$ the space of discretely divergence-free functions:
\begin{align*}
\FEspaceUdivfree = \lb \Utest \in \FEspaceU \ |
\ (\Ptest,\diver{}\Utest) = 0 \, \, \forall \, \Ptest \in \FEspaceP \rb \, ,
\end{align*} 
where the spaces $\FEspaceU$ and $\FEspaceP$ are the velocity and pressure space respectively. The choice of spaces $\FEspaceU$ and $\FEspaceP$ will be made precise in \S\ref{choicespacessec}.
\item[\itemizebullet]  Let $\Pi_{\FEspaceUdivfree}:\ltwod \longrightarrow \FEspaceUdivfree$ denote the $\ltwo$ projection onto the space $\FEspaceUdivfree$:
\begin{align}\label{ltwodef}
(\Pi_{\FEspaceUdivfree}\bv{v},\Utest) = (\bv{v},\Utest) \ \ \forall \, \Utest \in \FEspaceUdivfree \, .
\end{align}
\item[\itemizebullet] Similarly, we will denote $\Pi_{\FEspaceW}:\ltwo \longrightarrow \FEspaceW$ the $\ltwo$ projection onto the space $\FEspaceW$. 
\item[\itemizebullet] Define the Stokes projection of $(\bv{w},r)$ as the 
pair $(\SPvel\bv{w},\SPpress r) \in \FEspaceU \times \FEspaceP $ that solves
\begin{align}
\label{stokesproj}
\left\{
\begin{aligned}
( \gradv{}\SPvel\bv{w}, \nabla\Utest ) - ( \SPpress r, \diver{}\Utest ) &= 
( \gradv{w} , \nabla\Ptest ) - ( r, \diver{}\Utest ) &&\forall \Utest \in \FEspaceU \\
( \Ptest, \diver{}\SPvel\bv{w} ) &= ( \Ptest, \diver{}\bv{W} ) &&\forall \Ptest \in \FEspaceP  \, ,
\end{aligned}
\right.
\end{align}
which has the following well known approximation properties (see for instance \cite{Girault,HeyRann})
\begin{align}\label{approximab}
\|\bv{w} -\SPvel\bv{w} \|_{\ltwods}
+ h \|\bv{w} -\SPvel\bv{w} \|_{\honeds}
+ h \|r - \SPpress r \|_{\ltwos} &
\leq c \, h^{\polydegree+1} \,
\big(\|\bv{w}\|_{\bv{H}^{\polydegree+1}}
+ \|r\|_{H^{\polydegree}} \big) \, , 
\end{align}
for all $(\bv{w},r) \in \bv{H}^{\polydegree+1}(\Omega)\cap \hzerod\times H^{\polydegree}(\Omega)$, with $c$ independent of $h$, $\bv{w}$ and $r$.
\end{itemize}
%
\subsection{Scheme}
\label{secConvSchem}
To discretize the system \eqref{weakformulation}, and to avoid technicalities, we will assume that the initial data is smooth and consider an initialization as in \eqref{initROSens}.

For every $k\in \{1,\ldots,K\}$ we compute $\{\bv{U}^k,P^k,\bv{W}^k,\bv{M}^k\} \in \FEspaceU \times \FEspaceP \times \FEspaceW \times \FEspaceM$ that solves 
{\allowdisplaybreaks
\begin{subequations}
\label{converscheme}
\begin{align}
\label{Uconver}
\begin{split}
\alp \tfrac{ \inc\bv{U}^{k}}\dt, \Utest \arp
&+ \tril_h\lp\bv{U}^{k},\bv{U}^{k},\Utest\rp + \nunot \lp \gradv{U}^{k},\gradv{}\Utest \rp
- \lp P^{k}, \diver{}\Utest \rp \\
&= 2 \nu_r \lp \curl{W}^{k},\Utest \rp + \mu_0 \trilm\lp  \Utest, \bv{H}^{k}, \bv{M}^{k}\rp \, , 
\end{split} \\
\label{incompdisc}
\lp \Ptest, \diver{U}^k\rp &= 0 \, , \\
\label{Wconver}
\begin{split}
\inertiamom  \alp \tfrac{\inc\bv{W}^{k}}\dt,\Wtest \arp
&+ \inertiamom \tril_h\lp\bv{U}^{k},\bv{W}^{k},\Wtest\rp 
+ c_1 \lp \gradv{W}^{k}, \nabla \Wtest \rp 
+ c_2 \lp \diver{W}^{k},\diver{}\Wtest \rp \\
&+ 4 {\nu}_r \lp \bv{W}^{k}, \Wtest \rp =
2\nu_r \lp \curl{U}^{k}, \Wtest \rp
+ \mu_0 \lp \bv{M}^{k} \times \bv{H}^{k},\Wtest \rp \, ,
\end{split} \\
\label{Mconver}
\begin{split}
\alp \tfrac{\inc\bv{M}^{k}}\dt,\Mtest \arp
&- \trilm\lp\bv{U}^{k},\Mtest,\bv{M}^{k}\rp
+ \lp\bv{M}^k\times\bv{W}^k,\Mtest\rp 
+ \tfrac{1}{\chartime} \lp \bv{M}^k ,\Mtest\rp \\
&= \tfrac{\permit}{\chartime} \lp  \bv{H}^k,\Mtest\rp \, ,
\end{split}
\end{align}
\end{subequations}}
for all $\lb \Utest,\Ptest,\Wtest,\Mtest\rb \in \FEspaceU \times \FEspaceP \times \FEspaceW \times \FEspaceM $, where the magnetic field $\bv{H}^k$ now is given by 
\begin{align}\label{newHdef}
\bv{H}^k := \text{I}_{\FEspaceM}\ha^k \, ,
\end{align}
with $\text{I}_{\FEspaceM}$ defined in \eqref{interpolants}. The choice of spaces $\FEspaceU$, $\FEspaceP$, $\FEspaceW$ and $\FEspaceM$ does not need to be made precise now; we will provide specific constructions in \S\ref{choicespacessec}. Right now we only need to say that the following assumptions will be required:
\begin{enumerate}
\item[(A1)] The mesh $\triangulation$ is quasi-uniform.
\item[(A2)] The projector $\Pi_{\FEspaceW}$ is $\hzerod$-stable, namely
\begin{align}\label{honestabW}
\|\nabla\Pi_{\FEspaceW}\wtest\|_{\ltwods} \leq c \, \|\wtest\|_{\honeds} \ \
\forall \, \wtest \in \hzerod \, ,
\end{align}
with $c$ independent of $h$ and $\wtest$. In the context of quasi-uniform meshes the reader is referred to classical references like \cite{Ciar78,Girault}, and for non quasi-uniform meshes and different norms to \cite{Stein2002,Crou1987,MR3150226}. Quasiuniformity is a sufficient condition on $\triangulation$ for \eqref{honestabW} to hold true. More general sufficient conditions can be found in the aforementioned references. 
\item[(A3)] The projector $\Pi_{\FEspaceUdivfree}$, defined in \eqref{ltwodef}, is $\hzerod$-stable, namely
\begin{align}\label{Vhonestab}
\|\nabla\Pi_{\FEspaceUdivfree}\utest\|_{\ltwods} \leq c \, \|\utest\|_{\honeds} \ \ 
\forall \, \utest \in \Vspace \, ,
\end{align}
with $c$ independent of $h$ and $\utest$. For the proof of this stability result the reader is referred to \cite{Girault,HeyRann}. This property can be easily established provided that the pair $\lb \FEspaceU,\FEspaceP\rb$ admits a Fortin projector with optimal approximation properties in $\ltwod$. Such a Fortin operator exists for every LBB stable pair, provided that $\Omega$ is convex or of class $\mathcal{C}^{1,1}$ (see for instance \cite{salga2013} and references therein).
\item[(A4)] For all $\Mtest \in \FEspaceM$, we want each space component $\Mtest^i$ ($i:1, ... , d$) to belong to the same finite element space as the pressure, i.e. we will require $\FEspaceM = [\FEspaceP]^d$.
\item[(A5)] The pressure space $\FEspaceP$ should be discontinuous and it should contain a continuous subspace of degree 1 or higher. 
\end{enumerate}
Note that essentially, (A1) is a requirement for (A2) and (A3) to hold true. On the other hand, (A2) and (A3) (i.e. \eqref{honestabW} and \eqref{Vhonestab} respectively) will be used to obtain estimates of the time derivatives in dual norms (see Lemma \ref{derivestlemma}). 

Assumption (A5), more precisely the use of discontinuous pressures, will allow us to localize the incompressibility constraint \eqref{incompdisc} from the Stokes problem to each element, that is
\begin{align}\label{localizorth}
(\Ptest,\diver{}\bv{U}^k)_{\element} = 0 \ \ \forall \Ptest \in \FEspaceP, \, \forall \element \in \triangulation \, .
\end{align}
The constraint $\FEspaceM = [\FEspaceP]^d$ from (A4) together with \eqref{localizorth} means that:
\begin{align}\label{localizorth2}
\lp \Mtest^i, \diver{U}^k\rp_{\element} = 0  \ \ \forall \Mtest \in \FEspaceM \, , 
\forall \, i:1, .... , d \, , \ \forall \element \in \triangulation \, . 
\end{align}
Assumptions (A4)-(A5) imply that $\FEspaceM \cap \boldsymbol{\mathcal{C}}^0(\Omega) \neq \varnothing$, and that the construction of an interpolation operator $\text{I}_{\FEspaceM}$ satisfying \eqref{interpolants} and \eqref{optestim} is always possible. Note in \eqref{converscheme}, that we are using the definition \eqref{trilineardef} for the trilinear form $\trilm(\cdot,\cdot,\cdot)$, however, not all the terms are used. In particular, all the jump terms in the Kelvin force $\mu_0 \trilm\lp \Utest, \bv{H}^{k}, \bv{M}^{k}\rp$ disappear, since $\lj \bv{H}^{k} \rj\!\Big|_{F} = 0$ for all $F \in \mathcal{F}^i$, which is a consequence of definition \eqref{newHdef} and (A4)-(A5). This is a very convenient feature which will greatly simplify the a priori estimates and consistency analysis. 

The main difference between schemes \eqref{firstschemeROS} and \eqref{converscheme}, apart from the fact that the Poisson problem for $\hdpoth^k$ was eliminated, are the new requirements on the spaces $\FEspaceP$ and $\FEspaceM$. We now present the stability of scheme \eqref{converscheme}.

\begin{proposition}[existence and stability]\label{discenerglemma2} For every $k=1,\ldots,K$ there is $\lb\bv{U}^k,P^k,\bv{W}^k,\bv{M}^k\rb\in \FEspaceU \times \FEspaceP \times \FEspaceW \times \FEspaceM$ that solves \eqref{converscheme}, with $\bv{H}^k$ defined in \eqref{newHdef}. Moreover this solution satisfies the following stability estimate 
\begin{align}
\label{secstepdisc2}
\tfrac{1}{2}\mathscr{E}_{h,\dt}^K &+ \dt^{-1} \| \mathscr{I}_{h,\dt}^\dt \|_{\ell^1} + \| \mathscr{D}_{h,\dt}^\dt \|_{\ell^1} \leq \mathscr{F}_{h,\dt}^\dt + 2 \mathscr{E}_{h,\dt}^0 + \mu_0 \|\bv{H}^{0}\|_{\ltwods}^2 
+ 2 \mu_0 \|\bv{H}^{K}\|_{\ltwods}^2 \leq c \, , 
\end{align}
where
\begin{align*}
\begin{gathered}
\mathscr{E}_{h,\dt}^k := \mathscr{E}_{h,\dt}^k(\bv{U}^\dt,\bv{W}^\dt,\bv{M}^\dt,0) \, , \
\mathscr{I}_{h,\dt}^k := \mathscr{I}_{h,\dt}^k(\bv{U}^\dt,\bv{W}^\dt,\bv{M}^\dt,0) \, , \ 
\mathscr{D}_{h,\dt}^k := \mathscr{D}_{h,\dt}^k(\bv{U}^\dt,\bv{W}^\dt,\bv{M}^\dt,0) \, , \
\end{gathered}
\end{align*}
were defined in \eqref{eq:defofmatscrs}, and $\mathscr{F}_{h,\dt}^k$ is defined as
\begin{align*}
\mathscr{F}_{h,\dt}^k := \mathscr{F}^k(\ha) =
3 \mu_0 \chartime \sum_{k=1}^{K-1} \dt \LN\tfrac{\inc\bv{H}^{k+1}}{\dt} \RN_{\ltwods}^2 
+ \sum_{k=1}^{K} \tfrac{\mu_0 \dt}{\chartime} (1 + 3 \permit^2) \|\bv{H}^{k}\|_{\ltwods}^2 \, .
\end{align*}
The constant $c < \infty$ only depends on $\ha$, $\partial_t\ha$, and the initial data $\bv{u}_0$, $\bv{w}_0$, $\bv{m}_0$.
\end{proposition}
\begin{proof} Set $\Utest = 2 \dt \bv{U}^k$, $\Wtest = 2 \dt \bv{W}^k$, $\Mtest = 2 \dt \mu_0 \bv{M}^k$, in \eqref{converscheme} and add the results. Using \eqref{eq:Honezdcurldiv} and identity \eqref{sumid}, we get
\begin{align}
\label{firststepdisc2}
\begin{split}
2 \inc \mathscr{E}_{h,\dt}^k
+ 2 \mathscr{I}_{h,\dt}^k
+ 2 \dt \mathscr{D}_{h,\dt}^k &= 
2 \mu_0 \dt (\bv{M}^{k} \times \bv{H}^{k},\bv{W}^k) \\
&+ 2 \mu_0 \dt \trilm\lp \bv{U}^k, \bv{H}^{k}, \bv{M}^{k}\rp
+ \tfrac{2 \dt \mu_0 \permit}{\chartime} (\bv{M}^k,\bv{H}^k) \, .
\end{split}
\end{align}
As in the proof of Proposition~\ref{disclemma}, to deal with the trilinear terms $2 \mu_0 \dt \trilm\lp \bv{U}^k, \bv{H}^{k},\bv{M}^{k}\rp$ and $2 \mu_0 \dt (\bv{M}^{k} \times \bv{H}^{k},\bv{W}^k)$ we set $\Mtest = 2 \mu_0 \dt\bv{H}^k$ in \eqref{Mconver}. In doing so, we obtain
\begin{align}
\label{identity2}
\begin{split}
\tfrac{2 \mu_0 \dt \permit}{\chartime} \| \bv{H}^k \|_{\ltwods}^2 &=
- 2\mu_0 \dt \lp\bv{M}^k\times\bv{H}^k,\bv{W}^k \rp
- 2\mu_0 \dt \trilm\lp \bv{U}^{k}, \bv{H}^k,\bv{M}^{k} \rp \\
&+ 2\mu_0 \alp \inc\bv{M}^k,\bv{H}^k \arp
+ \tfrac{2\mu_0 \dt}{\chartime} \lp\bv{M}^k,\bv{H}^k \rp \, .
\end{split}
\end{align}
Adding \eqref{identity2} to \eqref{firststepdisc2} we obtain
\begin{align*}
2\inc \mathscr{E}_{h,\dt}^k
+ 2 \mathscr{I}_{h,\dt}^k
+ 2 \dt \mathscr{D}_{h,\dt}^k
+ \tfrac{2 \mu_0 \dt \permit}{\chartime} \| \bv{H}^k \|_{\ltwods}^2 = \tfrac{2 \dt \mu_0}{\chartime}(1+\permit) (\bv{M}^k,\bv{H}^k)
+ 2\mu_0 ( \inc \bv{M}^k, \bv{H}^k ) \, .
\end{align*}
The rest is just a matter of adding over $k$, using the summation by parts formula \eqref{summation} to $\sum_{k = 1}^{K} \alp \inc\bv{M}^{k},\bv{H}^k \arp$ and applying Cauchy-Schwarz and Young's inequalities with appropriate constants. Finally, existence is guaranteed by an analogous argument to that one used in Theorem \ref{localexist}, it only requires a local (in time) estimate which we omit to avoid repetitions. This concludes the proof.
\end{proof}

\begin{lemma}[estimates for the discrete time derivatives]\label{derivestlemma} 
The following estimates for $\dt^{-1}\inc\bv{U}^{k}$ and $\dt^{-1}\inc\bv{W}^{k}$ hold
\begin{align*}
\LN \tfrac{\inc\bv{U}^{\dt}}{\dt}\RN_{\ell^{4/3}(\Vspace^{*})}
+ \LN \tfrac{\inc\bv{W}^k}{\dt}\RN_{\ell^{4/3}(\bv{H}^{-1})} \leq c  \, , 
\end{align*}
where the constant $c < \infty$ only depends on $\ha$, $\partial_t\ha$, and the initial data $\bv{u}_0$, $\bv{w}_0$, $\bv{m}_0$.
\end{lemma}
\begin{proof} Following \cite{Feng2006,grun2013}, we use \eqref{Uconver}, \eqref{ltwodef}, \eqref{Vhonestab},   \eqref{eq:bhmskew}, \eqref{secstepdisc2}, and the regularity of the data $\ha$, to get:
{\allowdisplaybreaks\begin{align}
\label{disctimederestU}
\begin{split}
\LN \tfrac{\inc\bv{U}^{k}}{\dt}\RN_{\Vspace^*} 
&= \sup_{ \bv{v}\in \Vspace}
\frac{\alp \frac{ \inc\bv{U}^{k}}\dt, \bv{v} \arp}{\ \|\bv{v}\|_{\honeds}} 
= \sup_{ \bv{v}\in \Vspace }
\frac{\alp \frac{ \inc\bv{U}^{k}}\dt, \Pi_{\FEspaceUdivfree}[\bv{v}] \arp}{\ \|\bv{v}\|_{\honeds}} \lesssim
\sup_{ \bv{v}\in \Vspace }
\frac{\alp \frac{ \inc\bv{U}^{k}}\dt, \Pi_{\FEspaceUdivfree}[\bv{v}] \arp}{\ \|\Pi_{\FEspaceUdivfree}[\bv{v}]\|_{\honeds}}
\\
&\lesssim \|\bv{U}^{k}\|_{\bv{L}^3} \|\bv{U}^{k}\|_{\bv{L}^6}
+ \|\diver{U}^{k}\|_{L^2} \|\bv{U}^{k}\|_{\bv{L}^3}
+ \|\gradv{U}^{k}\|_{\ltwods}
+ \|\bv{W}^k\|_{\ltwods} \\
&+\|\nabla\bv{H}^{k}\|_{\linfds} \|\bv{M}^{k}\|_{\ltwods} 
+ \|\bv{H}^{k}\|_{\linfds} \|\bv{M}^{k}\|_{\ltwods} \\
&\lesssim \|\nabla\bv{U}^{k}\|_{\ltwods}^{3/2} 
+ \|\gradv{U}^{k}\|_{\ltwods}
+ \|\gradv{}\bv{W}^k\|_{\ltwods}
+ \|\bv{M}^{k}\|_{\ltwods} \\
&\lesssim \Big( \|\nabla\bv{U}^{k}\|_{\ltwods}^{2}
+ \|\gradv{U}^{k}\|_{\ltwods}^{4/3}
+ \|\nabla\bv{W}^{k}\|_{\ltwods}^{4/3}
+ \|\bv{M}^{k}\|_{\ltwods}^{4/3}  \Big)^{3/4} \, , 
\end{split} 
\end{align}}
where we have also used the estimate 
\begin{align*}
\|\bv{U}^{k}\|_{\bv{L}^3} \leq \|\bv{U}^{k}\|_{\bv{L}^2}^{1/2}
\|\bv{U}^{k}\|_{\bv{L}^6}^{1/2}
\lesssim \|\bv{U}^{k}\|_{\bv{L}^6}^{1/2}
\lesssim \|\nabla\bv{U}^{k}\|_{\ltwods}^{1/2} \, ,
\end{align*}
which relies on \eqref{secstepdisc2}, namely $\|\bv{U}^k\|_{\ltwods} \leq c$ uniformly in $k$.   Raise \eqref{disctimederestU} to power $4/3$, multiply by $\dt$, add in time, and use \eqref{secstepdisc2} to get the desired estimate for $\dt^{-1} \inc \bv{U}^k$.

Similarly, for $\tfrac{\inc\bv{W}^{k}}{\dt}$, we use \eqref{Wconver}, \eqref{honestabW}, \eqref{secstepdisc2}, and the regularity of the data $\ha$:
{\allowdisplaybreaks
\begin{align}
\label{disctimederestW}
\begin{split}
\LN \tfrac{\inc\bv{W}^{k}}{\dt}\RN_{\bv{H}^{-1}} 
&= \sup_{ \wtest \in \hzerod }
\frac{\alp \frac{ \inc\bv{W}^{k}}\dt, \wtest \arp}{\ \|\wtest\|_{\honeds}}
= \sup_{ \wtest\in \hzerod }
\frac{\alp \frac{ \inc\bv{W}^{k}}\dt, \Pi_{\FEspaceW}[\wtest] \arp}{\ \|\wtest\|_{\honeds}} \\
&\lesssim \sup_{ \wtest\in \hzerod } \frac{\alp \frac{ \inc\bv{W}^{k}}\dt, \Pi_{\FEspaceW}[\wtest] \arp}{\ \|\Pi_{\FEspaceW}[\wtest]\|_{\honeds}} \lesssim \|\bv{U}^{k}\|_{\bv{L}^4} \|\bv{W}^{k}\|_{\bv{L}^4} 
+ \|\diver{U}^{k}\|_{L^2} \|\bv{W}^{k}\|_{\bv{L}^3} \\
&+ \|\gradv{W}^{k}\|_{\ltwods}
+ \|\diver{}\bv{W}^k\|_{\ltwods} 
+ \|\bv{W}^k\|_{\ltwods}
+ \|\bv{U}^k\|_{\ltwods}\\
&+ \|\bv{M}^{k}\|_{\ltwods} \|\bv{H}^{k}\|_{\linfds}  \\
&\lesssim \|\nabla\bv{U}^{k}\|_{\ltwods}^{3/2}
+ \|\nabla\bv{W}^{k}\|_{\ltwods}^{3/2}
+ \|\gradv{}\bv{W}^k\|_{\ltwods} 
+ \|\bv{U}^{k}\|_{\ltwods}
+ \|\bv{M}^{k}\|_{\ltwods} \\
&\lesssim \Big( \|\nabla\bv{U}^{k}\|_{\ltwods}^{2}
+ \|\nabla\bv{W}^{k}\|_{\ltwods}^{2}
+ \|\gradv{}\bv{W}^k\|_{\ltwods}^{4/3} 
+ \|\bv{U}^{k}\|_{\ltwods}^{4/3}
+ \|\bv{M}^{k}\|_{\ltwods}^{4/3}  \Big)^{3/4} \, ,
\end{split}
\end{align}}
where we have also used the inequality 
\begin{align*}
\|\bv{U}^{k}\|_{\bv{L}^4} \leq \|\bv{U}^{k}\|_{\bv{L}^2}^{1/4}
\|\bv{U}^{k}\|_{\bv{L}^6}^{3/4}
\lesssim \|\bv{U}^{k}\|_{\bv{L}^6}^{3/4}
\lesssim \|\nabla\bv{U}^{k}\|_{\ltwods}^{3/2} \, , 
\end{align*}
and an analogous estimate for $\bv{W}^k$. Raise \eqref{disctimederestW} to the power $4/3$, multiply by $\dt$, add in time, and use \eqref{secstepdisc2} to get the desired estimate for $\dt^{-1} \inc \bv{W}^k$. 
\end{proof}
\subsection{Convergence}
\label{convschemesec}
We want to show that solutions generated by the scheme defined by \eqref{converscheme} and \eqref{newHdef}, with initial conditions \eqref{initROSens}, converge to solutions of the ultra weak formulation defined in \eqref{weakformulation}. However, the scheme  \eqref{converscheme} generates a sequence of functions $\lb \bv{U}^k, \bv{P}^k, \bv{W}^k, \bv{M}^k \rb_{k=0}^K$ corresponding to the nodes $\lb t^k \rb_{k=0}^K$, rather than space-time functions. In addition, the scheme \eqref{converscheme} does not have a variational structure in time. In order to reconcile these features we will rewrite scheme \eqref{converscheme} as a space-time variational formulation. For this purpose, we start by defining the space-time functions $\bv{U}_{h\dt}$, $P_{h\dt}$, $\bv{W}_{h\dt}$,  $\bv{M}_{h\dt}$,  such that
\begin{align}\label{pwconst}
\bv{U}_{h\dt} = \bv{U}^{k} , \, 
P_{h\dt} = P^{k} , \, 
\bv{W}_{h\dt} = \bv{W}^{k} , \, 
\bv{M}_{h\dt} = \bv{M}^{k} , \,
\bv{H}_{h\dt} = \bv{H}^{k} \ \ \forall \, t \in (t^{k-1},t^k] \, , 
\end{align}
which are piecewise constant in time. Even though these functions are not continuous in time, their point values are well defined, in particular they are left-continuous at the nodes $\lb t^k \rb_{k=0}^K$:
\begin{align}\label{lefcont}
\bv{U}_{h\dt} (t^k) = 
\lim_{t \nearrow t^k} \bv{U}_{h\dt} (t) \ \ \ \forall \ 0 \leq k \leq K \, .
\end{align}
Using the summation by parts formula \eqref{summation}, we can rewrite scheme \eqref{converscheme} in terms of $\lb \bv{U}_{h\dt}, P_{h\dt}, \bv{W}_{h\dt}, \bv{M}_{h\dt}\rb$ as follows:

{\allowdisplaybreaks
\begin{subequations}
\label{reformulated}
\begin{align}
\begin{split}
&(\bv{U}_{h\dt}(\tf),\Utest_{h\dt}(\tf)) 
- \int_{0}^{\tf-\dt} \alp \bv{U}_{h\dt},\tfrac{\Utest_{h\dt}(\cdot+\dt) - \Utest_{h\dt}}{\dt} \arp 
+ \tril_h\lp\bv{U}_{h\dt},\bv{U}_{h\dt},\Utest_{h\dt}\rp \\
& \ \ \ \ + \int_{0}^{\tf} \lp \nunot \gradv{U}_{h\dt},\gradv{}\Utest_{h\dt} \rp 
- \lp P_{h\dt}, \diver{}\Utest_{h\dt} \rp = (\bv{U}_{h\dt}(0),\Utest_{h\dt}(0)) \\
&\ \ \ \ + \int_{0}^{\tf} 2 \nu_r (\curl{}\bv{W}_{h\dt},\Utest_{h\dt})
+ \mu_0 \trilm\lp \Utest_{h\dt},\bv{H}_{h\dt},\bv{M}_{h\dt}\rp \, , 
\end{split} \\
\label{SNS2DG}
&\ \ \ \ \int_{0}^{\tf} \lp \Ptest_{h\dt}, \diver{U}_{h\dt}\rp = 0 \, , \\
\begin{split}\label{SspinDG}
&\inertiamom(\bv{W}_{h\dt}(\tf),\Wtest_{h\dt}(\tf)) 
-  \int_{0}^{\tf-\dt} \inertiamom \alp \bv{W}_{h\dt},\tfrac{\Wtest_{h\dt}(\cdot+\dt) - \Wtest_{h\dt}}{\dt} \arp \\
&\ \ \ \ + \int_{0}^{\tf} \inertiamom \tril_h \lp \bv{U}_{h\dt},\bv{W}_{h\dt},\Wtest_{h\dt}\rp 
+ c_1 (\gradv{}\bv{W}_{h\dt} ,\gradv{}\Wtest_{h\dt}) \\
&\ \ \ \ + c_2 \lp\diver{}\bv{W}_{h\dt},\diver{}\Wtest_{h\dt}\rp 
+ 4 \nu_r (\bv{W}_{h\dt},\Wtest_{h\dt}) = \inertiamom (\bv{W}_{h\dt}(0),\Wtest_{h\dt}(0) ) \\
&\ \ \ \ + \int_{0}^{\tf} 2 \nu_r \lp \curl{U}_{h\dt}, \Wtest_{h\dt} \rp 
+ \mu_0 (\bv{M}_{h\dt} \times \bv{H}_{h\dt},\Wtest_{h\dt}) \, , 
\end{split} \\
\begin{split}\label{SMag1DG}
&(\bv{M}_{h\dt}(\tf),\Mtest_{h\dt}(\tf)) 
- \int_{0}^{\tf-\dt} \alp \bv{M}_{h\dt},\tfrac{\Mtest_{h\dt}(\cdot+\dt) - \Mtest_{h\dt}}{\dt} \arp \\
&\ \ \ \ - \int_{0}^{\tf} \trilm\lp \bv{U}_{h\dt},\Mtest_{h\dt},\bv{M}_{h\dt}\rp 
+ (\bv{M}_{h\dt} \times \bv{W}_{h\dt},\Mtest_{h\dt}) 
+ \tfrac{1}{\chartime} \lp\bv{M}_{h\dt},\Mtest_{h\dt}\rp \\
&\ \ \ \ = (\bv{M}_{h\dt}(0),\Mtest_{h\dt}(0) )
+ \tfrac{\permit}{\chartime} \int_{0}^{\tf}  \lp \bv{H}_{h\dt}, \Mtest_{h\dt} \rp \, , 
\end{split} 
\end{align}
\end{subequations}}
for every $\lb\Utest_{h\dt}, \Ptest_{h\dt}, \Wtest_{h\dt}, \Mtest_{h\dt}, \rb \in \FEspaceUht \times \FEspacePht \times \FEspaceWht \times \FEspaceMht $, where
{\allowdisplaybreaks
\begin{align*}
\begin{split}
\FEspaceUht &= \lb \Utest_{h\dt} \in L^2(0,\tf,\FEspaceU) \ \Big| \  \Utest_{h\dt}\big|_{(t^{k-1},t^k]} \in \FEspaceU \otimes \mathbb{P}_0((t^{k-1},t^k]) \, , \ 0 \leq k \leq K  \rb \, , \\
\FEspacePht&= \lb \Ptest_{h\dt} \in L^2(0,\tf,\FEspaceP) \ \Big| \  \Ptest_{h\dt}\big|_{(t^{k-1},t^k]} \in \FEspaceP \otimes \mathbb{P}_0((t^{k-1},t^k]) \, , \ 0 \leq k \leq K  \rb \, , \\
\FEspaceWht &= \lb \Wtest_{h\dt} \in L^2(0,\tf,\FEspaceW) \ \Big| \  \Wtest_{h\dt}\big|_{(t^{k-1},t^k]} \in \FEspaceW \otimes \mathbb{P}_0((t^{k-1},t^k]) \, , \ 0 \leq k \leq K  \rb \, , \\
\FEspaceMht &= \lb \Mtest_{h\dt} \in L^2(0,\tf,\FEspaceM) \ \Big| \  \Mtest_{h\dt}\big|_{(t^{k-1},t^k]} \in \FEspaceM \otimes \mathbb{P}_0((t^{k-1},t^k])  \, , \ 0 \leq k \leq K  \rb \, , \\
\end{split}
\end{align*}}
and $\cdot + \dt$ and $\cdot - \dt$ (in \eqref{reformulated}) denote positive and negative shifts in time of size $\dt$. Expression \eqref{reformulated} is the reinterpretation of the Backward-Euler method as a zero-order Discontinuous Galerkin scheme (see for instance \cite{ErnGuermond,Walk2005,LiuWalk2007,MR2249024}). The difference between \eqref{converscheme} and \eqref{reformulated} is merely cosmetic, since they are equivalent formulations of the same scheme, but clearly \eqref{reformulated} has the right structure if we want to compare it with \eqref{weakformulation}. Note also that the choice of half-open intervals $(t^{k-1},t^k]$ in \eqref{pwconst}, leading to the left-continuity \eqref{lefcont} is consistent with upwinding fluxes --- we choose traces from the direction of flow of information, which is also consistent with causality.

\begin{lemma}[weak convergence]\label{weakconv} The family of functions $\lb\bv{U}_{h\dt},\bv{W}_{h\dt},\bv{M}_{h\dt}\rb_{h,\dt>0}$, defined in \eqref{pwconst} have the following convergence properties:
{\allowdisplaybreaks
\begin{align*}
&\bv{U}_{h\dt}  \xrightharpoonup[]{h,\dt \rightarrow 0}_* \bv{u}^* \ \text{in }L^{\infty}(0,\tf,\ltwod) \, ,\\
&\bv{U}_{h\dt} \xrightharpoonup[]{h,\dt \rightarrow 0 } \bv{u}^* \ \text{in }L^2(0,\tf,\honed) \, , \\
&\bv{W}_{h\dt}  \xrightharpoonup[]{h,\dt \rightarrow 0}_* \bv{w}^* \ \text{in }L^{\infty}(0,\tf,\ltwod) \, ,\\
&\bv{W}_{h\dt} \xrightharpoonup[]{h,\dt \rightarrow 0 } \bv{w}^* \ \text{in }L^2(0,\tf,\honed) \, , \\
&\bv{M}_{h\dt} \xrightharpoonup[]{h,\dt \rightarrow 0 }_* \bv{m}^* \ \text{in }L^{\infty}(0,\tf,\ltwod) \, , \\
&\bv{M}_{h\dt} \xrightharpoonup[]{h,\dt \rightarrow 0 } \bv{m}^* \ \text{in }L^{2}(0,\tf,\ltwod) \, , 
\end{align*}}
for some functions $\bv{u}^*$, $\bv{w}^*$ and $\bv{m}^*$. Here $\xrightharpoonup[]{}_*$ denotes weak-star convergence, and $h$ and $\dt$ tend to zero independently. 
\end{lemma}

\begin{proof} This is a direct consequence of Proposition~\ref{discenerglemma2} and definition \eqref{pwconst}. 
\end{proof}

Note that these modes of convergence are not strong enough to pass to the limit in every term of \eqref{reformulated}, so that the weak limits $\bv{u}^*$, $\bv{w}^*$ and $\bv{m}^*$ of the previous lemma might not necessarily be solutions of \eqref{weakformulation}. In order to improve these estimates we will use the classical Aubin-Lions lemma (\cf~\cite{Lions1969}):

\begin{lemma}[Aubin-Lions]\label{aubinlemma} Let $B_0$, $B$ and $B_1$ denote three Banach spaces such that
\begin{align*}
 B_0 \subset B \subset B_1 \, , 
\end{align*}
with $B_0$ and $B_1$ being reflexive, and $ B_0 \subset \subset B$. We define the space $W$ \begin{align*}
W = \lb w \, \big| \, w \in L^{p_0}(0,\tf,B_0) , \, w_t \in L^{p_1}(0,\tf,B_1)\rb \, , 
\end{align*}
with $1 < p_0, p_1 < \infty$, endowed with the following norm
\begin{align*}
\|w\|_{W} = \|w\|_{L^{p_0}(0,\tf,B_0)} +  \|w_t\|_{L^{p_1}(0,\tf,B_1)} \, .
\end{align*}
Then, the space $W$ is compactly embedded in $L^{p_0}(0,\tf,B)$.
\end{lemma}

\begin{lemma}[strong $L^2(0,\tf,\ltwo)$ convergence]\label{strongL2lem} The families of functions $\lb\bv{W}_{h\dt},\bv{U}_{h\dt}\rb_{h,\dt>0}$ defined in \eqref{pwconst} have the following additional convergence properties:
\begin{align*}
&\bv{U}_{h\dt} \xrightarrow[]{h,\dt \rightarrow 0} \bv{u}^* \ \text{in }L^2(0,\tf,\ltwod) \, , \\
&\bv{W}_{h\dt} \xrightarrow[]{h,\dt \rightarrow 0} \bv{w}^* \ \text{in }L^2(0,\tf,\ltwod) \, , 
\end{align*}
for some functions $\bv{u}^*$ and $\bv{w}^*$.
\end{lemma}
\begin{proof} Using the Aubin-Lions lemma (Lemma~\ref{aubinlemma}) we would like to conclude on the basis of the estimates provided in Proposition~\ref{discenerglemma2} and Lemma~\ref{derivestlemma}. However, this is not possible since the family of functions $\lb\bv{U}_{h\dt},\bv{W}_{h\dt}\rb_{h,\dt>0}$ is discontinuous in time --- time derivatives are not well defined. This a typical characteristic of discontinuous Galerkin methods for time integration such as the Backward Euler method. To overcome this, we define their Rothe interpolants, that is the piecewise linear and continuous auxiliary functions $\widehat{\bv{U}}_{h\dt}$ and $\widehat{\bv{W}}_{h\dt}$:
\begin{align*}
\widehat{\bv{U}}_{h\dt} = \ell_{k-1}(t) \bv{U}^{k-1} 
+ \ell_{k}(t) \bv{U}^k  \ , \ \
\widehat{\bv{W}}_{h\dt} = \ell_{k-1}(t) \bv{W}^{k-1} 
+ \ell_{k}(t) \bv{W}^k \ \ \forall \, t \in (t^{k-1},t^k] \, ,
\end{align*}
where $\ell_{k-1}(t) = (t^k - t)/\dt$ and $\ell_{k}(t) = (t - t^{k-1})/\dt$. Since
\begin{align*}
\partial_t \widehat{\bv{U}}_{h\dt}(t) = \dt^{-1} \inc \bv{U}^k, \quad \partial_t \widehat{\bv{W}}_{h\dt}(t) = \dt^{-1} \inc \bv{W}^k
\ \ \forall t \in (t^{k-1},t^k] \, ,
\end{align*}
we have that:
\begin{itemize}
\item[\itemizebullet] $\widehat{\bv{U}}_{h\dt}$ and $\widehat{\bv{W}}_{h\dt}$ converge strongly to some $\bv{u}^*$ and $\bv{w}^*$ in the $L^2(L^2)$ norm, i.e.
 \begin{align}\label{strongL2L2}
\|\widehat{\bv{U}}_{h\dt} - \bv{u}^*\|_{L^2(0,\tf,\ltwods)} + \|\widehat{\bv{W}}_{h\dt} - \bv{w}^*\|_{L^2(0,\tf,\ltwods)} \xrightarrow[]{h,\dt \rightarrow 0} 0  \, , 
\end{align} 
which is a consequence of Proposition~\ref{discenerglemma2}, the dual norm estimates for the time derivatives of Lemma~\ref{derivestlemma}, and a direct application of Lemma~\ref{aubinlemma}.

\item[\itemizebullet] The previous bullet implies that $\bv{U}_{h\dt}$ and $\bv{W}_{h\dt}$ also converge strongly to the same limits $\bv{u}^*$ and $\bv{w}^*$ in the $L^2(L^2)$ norm. For the velocity $\bv{U}_{h\dt}$ this is easy to show using the triangle inequality
\begin{align*}
\|\bv{U}_{h\dt} - \bv{u}^*\|_{L^2(0,\tf,\ltwods)} \leq 
\|\bv{U}_{h\dt} - \widehat{\bv{U}}_{h\dt}\|_{L^2(0,\tf,\ltwods)}
+ \|\widehat{\bv{U}}_{h\dt} - \bv{u}^*\|_{L^2(0,\tf,\ltwods)} \, ,
\end{align*}
where the term $\|\widehat{\bv{U}}_{h\dt} - \bv{u}^*\|_{L^2(0,\tf,\ltwods)}$ goes to zero because of \eqref{strongL2L2}, and 
$\|\bv{U}_{h\dt} - \widehat{\bv{U}}_{h\dt}\|_{L^2(0,\tf,\ltwods)}$ goes to zero because of the identity
\begin{align*}
\|\bv{U}_{h\dt} - \widehat{\bv{U}}_{h\dt}\|_{L^2(0,\tf,\ltwods)}^2 = 
\tfrac{\dt}{3} \sum_{k = 1}^{K} \|\inc\bv{U}^k\|_{\ltwods}^2
\end{align*}
and estimate \eqref{secstepdisc2}. For the angular velocity we can use the same argument to show that \\
$\|\bv{W}_{h\dt}- \bv{w}^*\|_{L^2(0,\tf,\ltwods)} \xrightarrow[]{h,\dt \rightarrow 0} 0$. 
\end{itemize}
This completes the proof.
\end{proof}
At this point we are in position to show the main convergence result.

\begin{theorem}[convergence]\label{mainconvlemma} The family of functions $\lb\bv{U}_{h\dt},\bv{W}_{h\dt},\bv{M}_{h\dt}\rb_{h,\dt>0}$, defined in \eqref{pwconst} has the following convergence properties
{\allowdisplaybreaks
\begin{align}\label{convmodesROS}
\begin{split}
&\bv{U}_{h\dt} \xrightarrow[]{h,\dt \rightarrow 0} \bv{u}^* \ \text{in }L^2(0,\tf,\ltwod) \, , \\
&\bv{U}_{h\dt} \xrightharpoonup[]{h,\dt \rightarrow 0 } \bv{u}^* \ \text{in }L^2(0,\tf,\hzerod) \, , \\
&\bv{W}_{h\dt} \xrightarrow[]{h,\dt \rightarrow 0} \bv{w}^* \ \text{in }L^2(0,\tf,\ltwod) \, , \\
&\bv{W}_{h\dt} \xrightharpoonup[]{h,\dt \rightarrow 0 } \bv{w}^* \ \text{in }L^2(0,\tf,\hzerod) \, , \\
&\bv{M}_{h\dt} \xrightharpoonup[]{h,\dt \rightarrow 0 } \bv{m}^* \ \text{in }L^2(0,\tf,\ltwod)  \, , 
\end{split}
\end{align}}
where $\{\bv{u}^*,\bv{w}^*,\bv{m}^*\} \in L^2(0,\tf,\Vspace) \times L^2(0,\tf,\hzerod) \times L^2(0,\tf,\ltwod)$ is a weak solution of  \eqref{weakformulation}. 
\end{theorem}
\begin{proof} The modes of convergence (weak or strong and their norm) in \eqref{convmodesROS} are a consequence of Lemmas \ref{weakconv} and \ref{strongL2lem}. It only remains to show that weak limits $\bv{u}^*$, $\bv{w}^*$ and $\bv{m}^*$ are solutions of the variational problem  \eqref{weakformulation}. For this purpose we will set $\lb\Utest_{h\dt},\Wtest_{h\dt},\Mtest_{h\dt}\rb$ to be the space-time interpolants/projections of the smooth test functions $\lb \utest, \wtest, \mtest\rb$ of the variational formulation  \eqref{weakformulation}:
\begin{align}\label{disctest}
\Utest_{h\dt} := \SPvel[\utest^k] \, , \ \
\Wtest_{h\dt} := \text{I}_{\FEspaceW}[\wtest^k] \, , \ \
\Mtest_{h\dt} := \text{I}_{\FEspaceM}[\mtest^k] \ \
\ \ \ \forall \, t \in (t^{k-1},t^k] \, . 
\end{align}
With this definition of discrete test functions we get in \eqref{reformulated}:
\begin{subequations}
\label{reformeval}
\begin{align}
\begin{split}\label{SNS1DG2}
&- \int_{0}^{\tf-\dt} \alp \bv{U}_{h\dt},\tfrac{\Utest_{h\dt}(\cdot+\dt) - \Utest_{h\dt}}{\dt} \arp 
+ \tril_h\lp\bv{U}_{h\dt},\bv{U}_{h\dt},\Utest_{h\dt}\rp \\
&\ \ \ \ \ \ \ \ \ + \int_{0}^{\tf} \nunot \lp \gradv{U}_{h\dt},\gradv{}\Utest_{h\dt} \rp 
= (\bv{U}_{h\dt}(0),\Utest_{h\dt}(0)) \\
&\ \ \ \ \ \ \ \ \ + \int_{0}^{\tf} 2 \nu_r (\curl{}\bv{W}_{h\dt},\Utest_{h\dt}) 
+ \mu_0 \trilm\lp \Utest_{h\dt}, \bv{H}_{h\dt},\bv{M}_{h\dt}\rp \, , 
\end{split} \\
\begin{split}\label{SspinDG2}
&- \int_{0}^{\tf-\dt} \inertiamom \alp \bv{W}_{h\dt},\tfrac{\Wtest_{h\dt}(\cdot+\dt) - \Wtest_{h\dt}}{\dt} \arp 
+ \int_{0}^{\tf} \inertiamom \tril_h \lp \bv{U}_{h\dt},\bv{W}_{h\dt},\Wtest_{h\dt}\rp \\
&\ \ \ \ \ \ \ \ \ + c_1 (\gradv{}\bv{W}_{h\dt} ,\gradv{}\Wtest_{h\dt}) 
+ c_2 \lp\diver{}\bv{W}_{h\dt},\diver{}\Wtest_{h\dt}\rp 
+ 4 \nu_r (\bv{W}_{h\dt},\Wtest_{h\dt}) \\
&\ \ \ \ \ \ \ \ \ = \inertiamom (\bv{W}_{h\dt}(0),\Wtest_{h\dt}(0) ) 
+ \int_{0}^{\tf}  2 \nu_r \lp \curl{U}_{h\dt}, \Wtest_{h\dt} \rp \\
&\ \ \ \ \ \ \ \ \ + \mu_0 (\bv{M}_{h\dt} \times \bv{H}_{h\dt},\Wtest_{h\dt}) \, , 
\end{split} \\
\begin{split}\label{SMag1DG2}
&- \int_{0}^{\tf-\dt} \alp \bv{M}_{h\dt},\tfrac{\Mtest_{h\dt}(\cdot+\dt) - \Mtest_{h\dt}}{\dt} \arp 
- \int_{0}^{\tf} \trilm\lp \bv{U}_{h\dt},\Mtest_{h\dt},\bv{M}_{h\dt}\rp \\
&\ \ \ \ \ \ \ \ \ + (\bv{M}_{h\dt} \times \bv{W}_{h\dt},\Mtest_{h\dt}) + \tfrac{1}{\chartime} \lp\bv{M}_{h\dt},\Mtest_{h\dt}\rp \\
&\ \ \ \ \ \ \ \ \ = (\bv{M}_{h\dt}(0),\Mtest_{h\dt}(0) )
+ \tfrac{\permit}{\chartime} \int_{0}^{\tf}  \lp \bv{H}_{h\dt}, \Mtest_{h\dt} \rp \, ,
\end{split} 
\end{align}
\end{subequations}
where the terms evaluated at time $t = \tf$ have disappeared because of the compact support of the test functions $\lb \utest, \wtest, \mtest\rb$. Note also that the pressure term of the Navier Stokes equation has vanished too, which is a consequence of the definition of the discrete test function $\Utest_{h\dt} := \SPvel[\utest^k] \ \ \forall \, t \in (t^{k-1},t^k]$, involving the Stokes projector $\SPvel$; see \eqref{stokesproj} and  \eqref{disctest}. Now we will pass to the limit term by term in \eqref{reformeval}: 
\begin{itemize}
\item[\itemizebullet] We start with the terms with the time derivatives, which are straightforward:
\begin{align*}
- \int_{0}^{\tf-\dt} \alp \bv{U}_{h\dt},\tfrac{\Utest_{h\dt}(\cdot+\dt) - \Utest_{h\dt}}{\dt} \arp 
\ \ \ &\xrightarrow[]{h,\dt \rightarrow 0} \ \ \ - \int_{0}^{\tf} ( \bv{u}^*,\utest_t ) \, , \\
- \int_{0}^{\tf-\dt} \alp \bv{W}_{h\dt},\tfrac{\Wtest_{h\dt}(\cdot+\dt) - \Wtest_{h\dt}}{\dt} \arp 
\ \ \ &\xrightarrow[]{h,\dt \rightarrow 0} \ \ \ - \int_{0}^{\tf} ( \bv{w}^*,\wtest_t ) \, , \\
- \int_{0}^{\tf-\dt} \alp \bv{M}_{h\dt},\tfrac{\Mtest_{h\dt}(\cdot+\dt) - \Mtest_{h\dt}}{\dt} \arp  
\ \ \ &\xrightarrow[]{h,\dt \rightarrow 0} \ \ \ - \int_{0}^{\tf} (\bv{m}^*,\mtest_t) \, , 
\end{align*}
because of the strong $L^2(L^2)$ convergence of $\bv{U}_{h\dt}$, $\bv{W}_{h\dt}$, the weak $L^2(L^2)$ convergence of $\bv{M}_{h\dt}$, and the strong convergence of the finite differences $\tfrac{\Utest_{h\dt}(\cdot+\dt) - \Utest_{h\dt}}{\dt}$, $\tfrac{\Wtest_{h\dt}(\cdot+\dt) - \Wtest_{h\dt}}{\dt}$ and $\tfrac{\Mtest_{h\dt}(\cdot+\dt) - \Mtest_{h\dt}}{\dt}$, guaranteed by the regularity of the test functions.

\item[\itemizebullet] We continue with the convective term of \eqref{SMag1DG2}. Using definition \eqref{trilineardef} we get
\begin{align}
\label{trilimlimit}
\begin{aligned}
\int_{0}^{\tf} \trilm\lp \bv{U}_{h\dt},\bv{M}_{h\dt},\Mtest_{h\dt}\rp &=
- \int_{0}^{\tf} \trilm\lp \bv{U}_{h\dt},\Mtest_{h\dt},\bv{M}_{h\dt}\rp \\
&= - \int_{0}^{\tf} \sum_{\element \in \triangulation} \int_\element \big( (\bv{U}_{h\dt} \cdot \nabla )\bv{Z}_{h\dt} \cdot \bv{M}_{h\dt} + \tfrac{1}{2} \diver{}\bv{U}_{h\dt} \, \bv{Z}_{h\dt} \cdot \bv{M}_{ht} \big) \, .
\end{aligned}
\end{align} 
Note that the consistency terms with the jumps have disappeared since $\lj\bv{Z}_{h\dt}\rj\big|_{F} = 0$ for all $F\in \mathcal{F}^i$, which is a consequence of definitions \eqref{interpolants} and \eqref{disctest} and assumptions (A4) and (A5). Passage to the limit in the first part of \eqref{trilimlimit}, that is
\begin{align*}
- \int_{0}^{\tf} \sum_{\element \in \triangulation} \int_\element (\bv{U}_{h\dt} \cdot \nabla )\bv{Z}_{h\dt} \cdot \bv{M}_{h\dt} \ \ \ \xrightarrow[]{h,\dt \rightarrow 0} \ \ \ - \int_{0}^{\tf} \tril(\bv{u}^*,\bv{z},\bv{m}^*) \, , 
\end{align*}
is carried out using the strong $L^2(L^2)$ convergence of $\bv{U}_{h\dt}$, the weak $L^2(L^2)$ convergence of $\bv{M}_{h\dt}$ and the strong convergence of $\nabla\bv{Z}_{h\dt}$ guaranteed by \eqref{optestim}. By consistency, we need the second part of \eqref{trilimlimit} to vanish when $h,\dt \rightarrow 0$:
\begin{align}
\label{temamconsistency}
\begin{split}
& - \int_{0}^{\tf} \sum_{\element \in \triangulation} \int_\element \tfrac{1}{2} \diver{}\bv{U}_{h\dt} \, \bv{Z}_{h\dt} \cdot \bv{M}_{h\dt} = - \int_{0}^{\tf} \sum_{\element \in \triangulation} \int_\element \tfrac{1}{2} \diver{}\bv{U}_{h\dt} \, (\bv{Z}_{h\dt} - \langle\bv{Z}_{h\dt} \rangle_{\element} )\cdot \bv{M}_{h\dt} \rp \\
&\ \ \ \ \ \lesssim \|\nabla\bv{U}_{h\dt}\|_{L^2(0,\tf,\ltwods)} \ h \|\nabla\Mtest\|_{L^\infty(0,\tf,\linfds)} \ \|\bv{M}_{h\dt}\|_{L^2(0,\tf,\ltwods)}
\  \xrightarrow[]{h,\dt \rightarrow 0} \  0 \, , 
\end{split}
\end{align}
where $\langle\bv{Z}_{h\dt} \rangle_{\element} = \tfrac{1}{|\element|}\int_{\element} \bv{Z}_{h\dt}$. Estimate \eqref{temamconsistency} employs the local orthogonality property \eqref{localizorth2}, which is a consequence of assumptions (A4) and (A5). In \eqref{temamconsistency} we also used the uniform bounds on $\bv{U}_{h\dt}$ and $\bv{M}_{h\dt}$, and the regularity of the test function $\mtest$. 

Passage to the limit of the remaining convective terms (those in \eqref{SNS1DG2} and \eqref{SspinDG2}) follow standard procedures and their treatment can be found in other works such as \cite{Temam,MarTem}.

\item[\itemizebullet] For the Kelvin force in \eqref{SMag1DG2}, recalling again definition \eqref{trilineardef} we get
\begin{align}
\begin{aligned}\label{PassLimitKelvin}
&\int_{0}^{\tf} \trilm\lp \Utest_{h\dt}, \bv{H}_{h\dt},\bv{M}_{h\dt}\rp = \int_{0}^{\tf} \sum_{\element \in \triangulation} \int_\element \lp (\Utest_{h\dt} \cdot \nabla )\bv{H}_{h\dt} \cdot \bv{M}_{h\dt}
+ \tfrac{1}{2} \diver{}\Utest_{h\dt} \, \bv{H}_{h\dt} \cdot \bv{M}_{h\dt} \rp \\
&\ \ \ \ \ \ \xrightarrow[]{h,\dt \rightarrow 0}
\int_{0}^{\tf} \int_{\Omega} (\utest \cdot \nabla)\bv{h} \,  \bv{m}^* = \int_{0}^{\tf} \int_{\Omega} (\bv{m}^* \cdot \nabla)\bv{h} \,  \utest \, , 
\end{aligned}
\end{align}
which follows from analogous arguments to those used in \eqref{temamconsistency}. To show that the term $\frac{1}{2} \diver{}\Utest_{h\dt} \, \bv{H}_{h\dt} \cdot \bv{M}_{h\dt}$ vanishes in the limit we resort to property \eqref{localizorth2}, this time by adding a term of the form $\langle\bv{H}_{h\dt} \rangle_{\element}\cdot\bv{M}_{h\dt}$. Finally the weak $L^2(L^2)$ convergence of $\bv{M}_{h\dt}$, the strong convergence properties of $\bv{H}_{h\dt}$ and the test function $\Utest_{h\dt}$, and the fact that $\nabla\heff = \nabla\heff^\text{T}$ since $\curl\heff = 0$, are all what we need in the passage to the limit.
\item[\itemizebullet] For the magnetic torque in \eqref{SspinDG2} we have that
\begin{align*}
\int_{0}^{\tf} \mu_0 (\bv{M}_{h\dt} \times \bv{H}_{h\dt},\Wtest_{h\dt})
\xrightarrow[]{h,\dt \rightarrow 0} \ \ \ \int_{0}^{\tf} \mu_0 (\bv{m}^* \times \bv{h},\wtest) \, , 
\end{align*}
which follows by the weak $L^2(L^2)$ convergence of $\bv{M}_{h\dt}$ and the strong convergence of $\bv{H}_{h\dt}$ and $\Wtest_{h\dt}$. 
\end{itemize}
Since the remaining terms in \eqref{reformeval} are linear they require little or no explanation in their passage to the limit. The proof is thus complete.
\end{proof}

\subsection{Finite element spaces}
\label{choicespacessec}
Now we provide a couple of finite element spaces which satisfy the key assumptions required in the proof of convergence of scheme \eqref{converscheme}.

In space dimension two, the spaces $\FEspaceW$ and $\FEspaceM$ can be the same as those defined in \eqref{choice2}--\eqref{FEspaces} with $\ell = 2$, while the pair $(\FEspaceU,\FEspaceP)$ can be chosen as
\begin{align}
\label{CrouRavPair}
\begin{split}
\FEspaceU &= \lb \bv{U} \in \bm{\mathcal{C}}^0\bigl(\overline\Omega\bigr) \, \big|
\, \bv{U}|_{T} \in [\simplex_{2}(\element) \oplus \text{Span} \, \mathcal{B}(\element)]^{2} \, , \forall \, \element \in \triangulation \rb \cap \hzerod \, , \\
\FEspaceP &= \lb \Ptest \in L^2\bigl(\Omega\bigr) \ |
\ \Ptest|_{T} \in \simplex_{{1}}(\element) \, , \forall \, \element \in \triangulation \rb \, ,  
\end{split}
\end{align}
where $\mathcal{B}(\element) = \prod_{i =1}^{3} \lambda_i$ is the cubic bubble function, and $\lb \lambda_i \rb_{i=1}^{3}$ are the barycentric coordinates of $\element$. This finite element pair is known as enriched Taylor-Hood and it is well known to be LBB stable (\cf~\cite{ErnGuermond,Girault}) in two dimensions. With this choice of finite element space we will naturally have that $\FEspaceM = [\FEspaceP]^2$, and that $\FEspaceM \cap \boldsymbol{\mathcal{C}}^0(\overline\Omega)$ will be non-empty.

In space dimension three the reader could consider using again the finite element spaces $\FEspaceW$ and $\FEspaceM$ (with $\ell = 2$) defined in \eqref{choice2}--\eqref{FEspaces} and the second-order Bernardi-Raugel element (see \cite[p. 148]{Girault}) which uses $\simplex_{1}$ discontinuous elements for the pressure space.

If a higher polynomial degree $\polydegree>2$ is desired, suitable pairs $(\FEspaceU,\FEspaceP)$ are described in \cite[Remark 8.6.2]{Boff2013} and \cite[\S8.72]{Boff2013} for two and three dimensions, respectively.

Let us, finally, point out that a careful examination of the proof of Theorem~\ref{mainconvlemma} reveals that, in conjunction, assumptions (A4) and (A5) are only used to show the consistency of the discrete trilinear forms through the orthogonality property \eqref{localizorth2}; see \eqref{temamconsistency} and \eqref{PassLimitKelvin}. Therefore, they can be replaced by:
\begin{enumerate}
  \item[(A4')] The pair $(\FEspaceU,\FEspaceP)$ yields pointwise divergence free velocities.
  \item[(A5')] The magnetization space $\FEspaceM$ is discontinuous and contains a continuous subspace of degree 1 or higher.
\end{enumerate}
With additional, and somewhat restrictive, mesh conditions spaces that verify (A4') are presented in \cite{FalkNeiland} and \cite{Neilan3d} for two and three dimensions, respectively.

\section{Numerical validation}
\label{validation}
Let us explore computationally the convergence properties of the scheme \eqref{firstschemeROS} using manufactured solutions. The implementation has been carried out with the help of the \texttt{deal.II} library; see \cite{BHK2007,DealIIReference}. Since we must guarantee the inclusion $\nabla\FEspacePhi \subset \FEspaceM$ and this library uses quadrilateral/hexahedrons, the finite element spaces defined in \eqref{choice1}-\eqref{choice2}, which are based on simplices, cannot be used. In addition, we cannot use the same polynomial degrees used in \eqref{choice1}-\eqref{choice2}, since the inclusion $\nabla\quadrilateral_{\polydegree}(\element) \subset [\quadrilateral_{\polydegree-1}(\element)]^d$ is not true. For this reason we set 
\begin{align}
\label{quadchoice}
\begin{aligned}
\FEspacePhi &= \lb \Phitest \in \mathcal{C}^0\bigl(\overline\Omega\bigr) \ | \ \Phitest|_{T} \in \quadrilateral_\polydegree(\element) \, , \forall \element \in \triangulation \rb  \subset \hone \, ,  \\
\FEspaceM &= \lb \bv{M} \in \ltwods(\Omega) \ | \ \bv{M}|_{T} \in [\quadrilateral_\polydegree(\element)]^d \, , \forall \, \element \in \triangulation \rb \, ,
\end{aligned}
\end{align}
and since $\nabla\quadrilateral_{\polydegree}(\element) \subset [\quadrilateral_{\polydegree}(\element)]^d$ holds true \eqref{quadchoice} is a valid pair. Even without a proper a priori error analysis, it is easy to realize that this choice is suboptimal in terms of approximation. Since $\polydegree = 2$, the best we can expect
\begin{align*}
\|\bv{U}^\dt - \bv{u}^\dt\|_{\ell^{\infty}(\ltwods)}
&+\|\bv{W}^\dt - \bv{w}^\dt\|_{\ell^{\infty}(\ltwods)} 
+\|\bv{M}^\dt - \bv{m}^\dt\|_{\ell^{\infty}(\ltwods)} 
+\|\nabla\hdpoth^\dt - \nabla\hdpot^\dt\|_{\ell^{\infty}(\ltwods)} \lesssim \dt + h^{2} \, , 
\end{align*}
which is suboptimal by a power of $h$ for the linear velocity $\bv{U}$, angular velocities $\bv{W}$ and magnetization $\bv{M}$. The reason for this loss of accuracy is the term $\|\nabla\hdpoth^\dt - \nabla\hdpot^\dt\|_{\ell^{\infty}(\ltwods)}$.

The arising linear systems have been solved with the direct solver \texttt{UMFPACK}$^\copyright$ for validation purposes. The nonlinear system is solved using a fixed point iteration. Let $\Omega = (0,1)^2 \subset \mathbb{R}^2$ and
\begin{align*}
\bv{u}(x,y,t) &= \sin(t) \big(  \sin(\pi x) \sin(\pi(y + 0.5)), \cos(\pi x) \cos(\pi (y + 0.5))  \big)^\intercal \, , \\
p(x,y,t) &= \sin(2 \pi (x - y) + t) \, , \\
\bv{w}(x,y,t) &= \sin(2 \pi x + t) \, \sin(2 \pi y + t) \, , \\
\bv{m}(x,y,t) &= \big( \sin(2 \pi x + t) \cos(2 \pi y + t) , \cos(2 \pi x + t)  \sin(2 \pi y + t) \big)^\intercal \, , \\
\hdpot(x,y,t) &= \sin(\pi x + t) \sin(\pi y + t).
\end{align*}
The right-hand sides are computed accordingly. To verify the convergence rates we set $\dt = h^2$ and consider a sequence of meshes with $h = 2^{-i}$ for $2 \leq i \leq 7$. Figure~\ref{FigErrors} shows the experimental errors, thereby confirming the predicted convergence rates.

\begin{figure}
\begin{center}
  \includegraphics[scale=0.6]{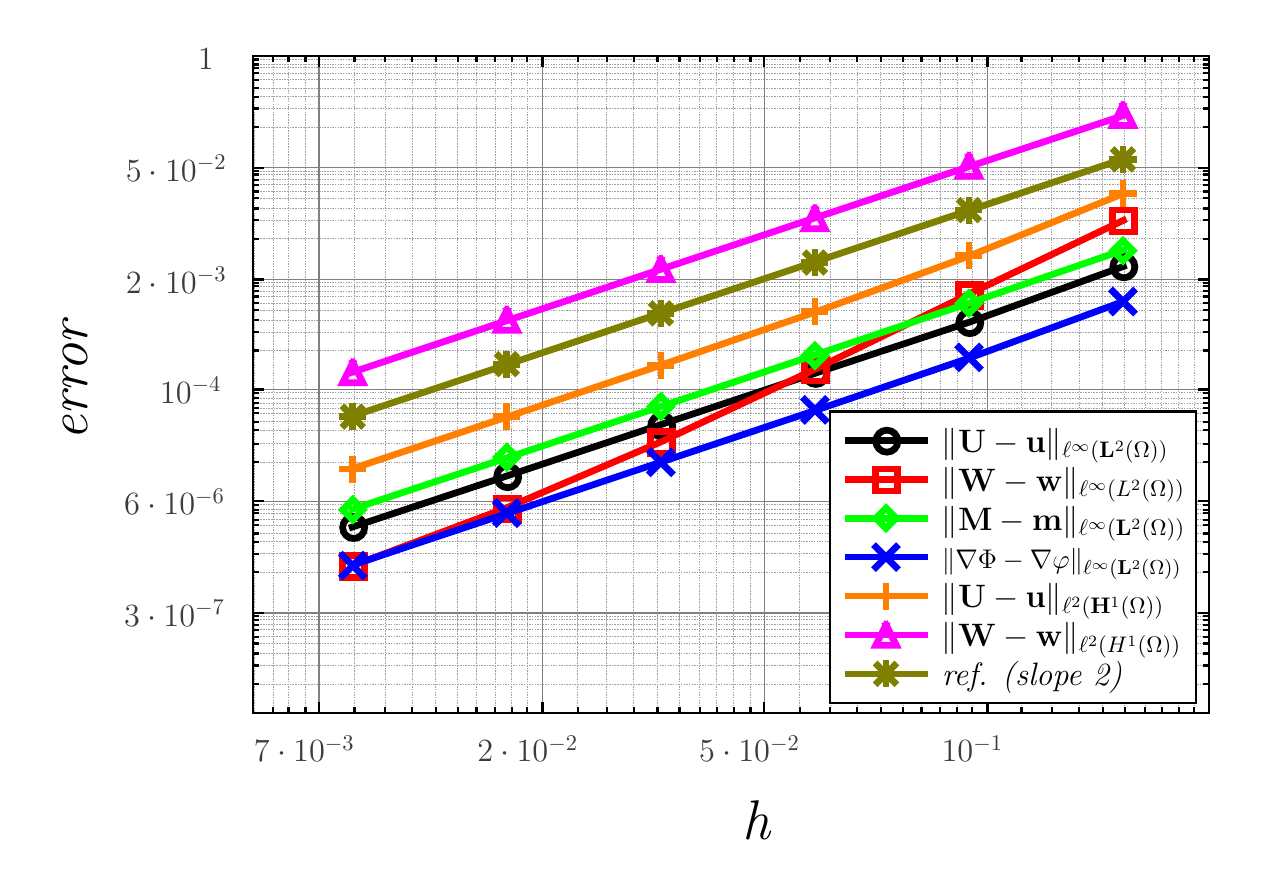}
\end{center}
\caption[Experimental order of convergence for the Rosensweig's model (log-log scale)]{\textbf{Experimental order of convergence for the Rosensweig's model (log-log scale).} Errors for the linear velocity $\bv{u}$, angular velocity $\bv{w}$, magnetization $\bv{m}$ and magnetic potential $\hdpot$. Note that the convergence rates in the $\ell^{\infty}(\ltwods)$ norms are suboptimal as expected, because of the choice of finite element spaces \eqref{quadchoice}. Note that the angular velocity (square markers) initially converges much faster than it should, and shows the asymptotic rate in the last two points.}
\label{FigErrors}
\end{figure}

\section{Numerical experiments with point dipoles}
\label{sec:numerics}
Let us now explore model \eqref{eq:ferroeq} and scheme \eqref{firstschemeROS} with a series of more realistic examples. In all these experiments we will consider the gradient of the potential of a point dipole $\nabla\hapot_s$ as a prototype for an applied magnetizing field, where
\begin{align}\label{dipole}
\hapot_s(\bv{x}) = \frac{\dipdir\cdot (\bv{x}_s - \bv{x})}{|\bv{x}_s- \bv{x}|^3} \, , 
\end{align}
$|\dipdir| = 1 $ indicates the direction of the dipole, and $\bv{x}_s = (x_s,y_s,z_s) \in \mathbb{R}^3$ is its location. It is not difficult to verify that $\curl{}\nabla\hapot_s=0$ and $\diver{}\nabla\hapot_s = \Delta\hapot_s = 0$ for every $\bv{x} \neq \bv{x}_s$, so that $\nabla\hapot_s$ defines a harmonic vector field. This is a physical requirement in the context of non-conducting media: the magnetic field $\ha$ must satisfy the equations of magnetostatics (see Remark \ref{remHarmonic}). \\
\begin{figure}
\begin{center}
  \setlength\fboxsep{0pt}
  \setlength\fboxrule{1pt}
  \fbox{\includegraphics[scale=0.2]{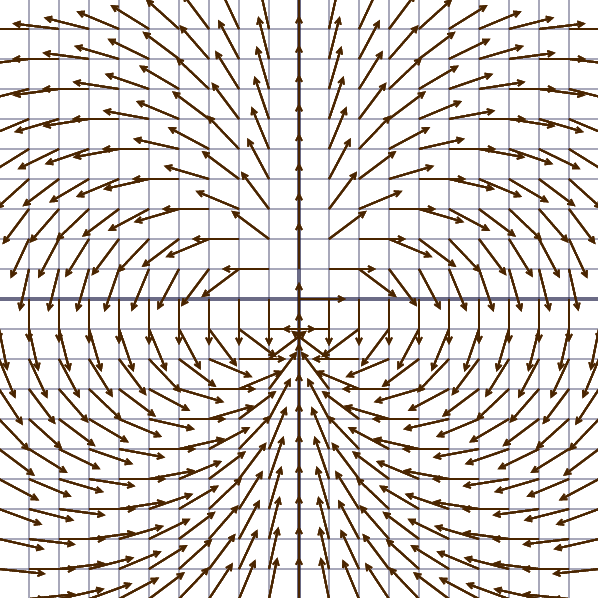}}
\end{center}
\caption[Plot the magnetic field due to a point dipole in 2d]{Normalized plot of the two dimensional harmonic field $\nabla\hapot_s$, with $\hapot_s$ as in \eqref{dipole2D}. Here $\bv{x}_s = (0,0)$ and $\dipdir = (0,1)$.\label{figura2}}
\end{figure}
Formula \eqref{dipole}, however, is intrinsically three dimensional \cite{Jack1998}. For this reason, we consider an alternative definition which leads to a two dimensional harmonic vector field:
\begin{align}\label{dipole2D}
\hapot_s(\bv{x}) = \frac{\dipdir\cdot (\bv{x}_s - \bv{x})}{|\bv{x}_s- \bv{x}|^2} \, ,
\end{align}
where now $\bv{x} , \bv{x}_s, \bv{d} \in \mathbb{R}^2$. Figure~\ref{figura2} shows $\nabla\hapot_s/|\nabla\hapot_s|$ for $\dipdir = (0,1)$ and $\bv{x}_s = (0,0)$. In our numerical experiments we will use linear combinations of dipoles as applied magnetizing field $\ha$
\begin{align}
\label{haformula}
\ha = \sum_{s} \alpha_s \nabla\hapot_s.
\end{align}

\subsection{Experiment 1: Spinning magnet}
\label{sub:spinmagnet}

We consider $\Omega = (0,1)^2 \subset \mathbb{R}^2$ filled with ferrofluid, $t\in [0,4]$, and the applied magnetic field of dipole $\ha = \alpha_s \nabla\hapot_s$, represented in Figure~\ref{figura3} as a rectangular permanent magnet. The application of the magnetic field will obey the following sequence:
\begin{itemize}
 \item[\itemizebullet] The initial position and orientation of the dipole are $\bv{x}_s = (0.5, -0.4)$ and $\bv{d} = (0,1)$ respectively.
 
 \item[\itemizebullet] For $t \in [0,1)$ we linearly increase the intensity $\alpha_s$ from $\alpha_s = 0$ to $\alpha_s = 10$ without changing $\dipdir$.
 
 \item[\itemizebullet] We let the fluid rest for $t \in [1,2)$.
 
 \item[\itemizebullet] For $t\in [2,4]$ we make the dipole go through a circular path around the center of the box $(0.5,0.5)$, with $\bv{d}$ pointing to the center of the box at every time as depicted in Figure~\ref{figura3}.
\end{itemize}

All the constitutive parameters of the model ($\nu$, $\nu_r$, $\mu_0$, $\inertiamom$, $c_a$, $c_d$, $c_0$ and $\permit$) were set to one with the exception of $\chartime$ which was taken to be $10^{-4}\text{ s}$ in order to achieve an almost instantaneous alignment of the magnetization with the magnetic field. The discretization uses a uniform mesh with $100$ elements in each space direction and $400$ time steps. The main goal of this experiment, and the displayed graphical results, is to help the reader form some intuition about the coupling between the linear velocity $\bv{u}$, angular velocity $\bv{w}$, and the magnetization $\bv{m}$. For this reason reference/scales have been omitted. Figures~\ref{figura4}--\ref{figura6} show some graphical results.

\begin{figure}
\begin{center}
  \includegraphics[scale=0.25]{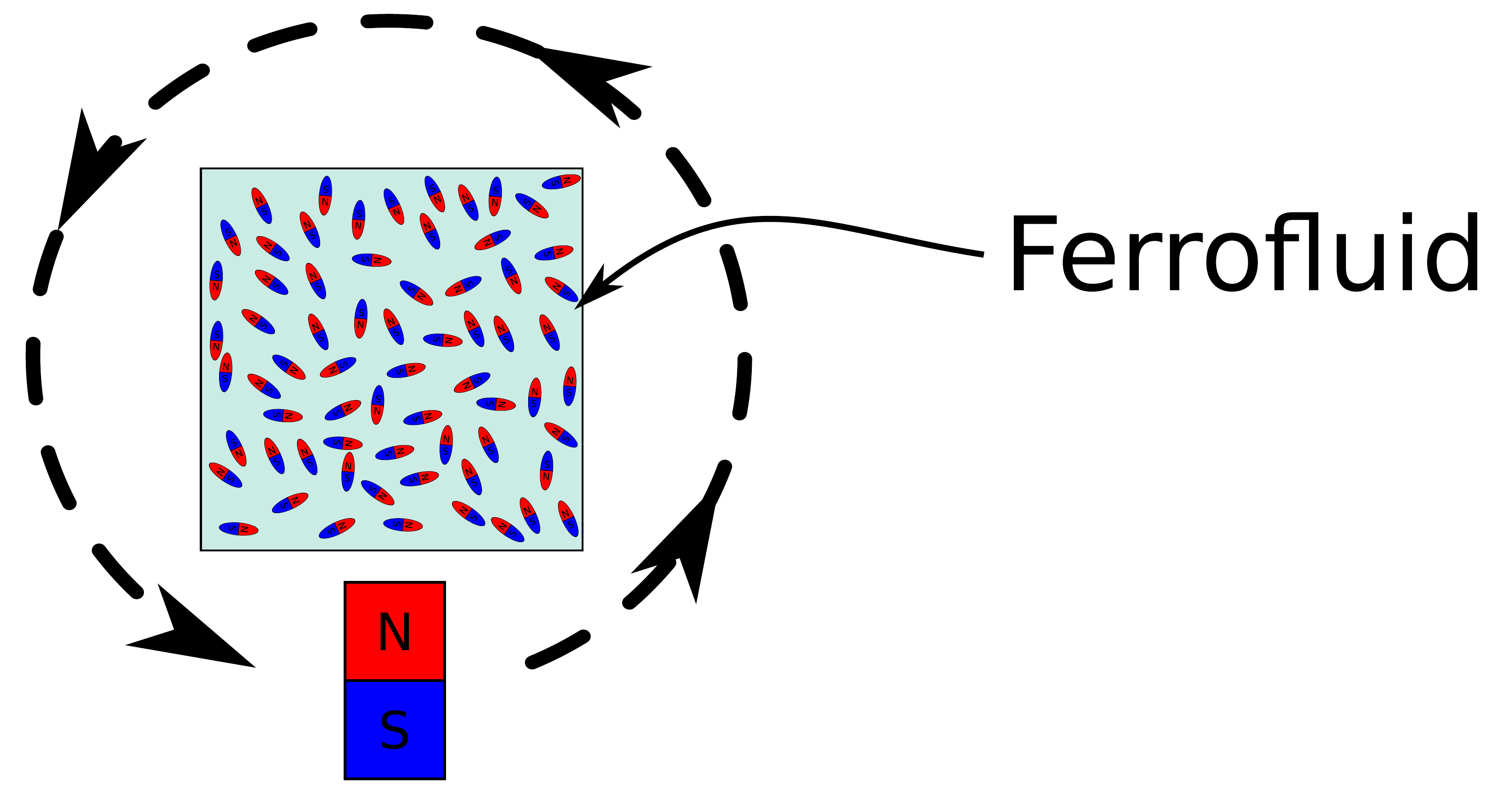}
\end{center}
\caption[Spinning magnet experiment: setup.]{\textbf{Spinning magnet experiment: setup.} Scheme of the experiment of \S\ref{sub:spinmagnet} (spinning magnet). We make the dipole go through a circular path around the center of the box $(0.5,0.5)$, with $\bv{d}$ pointing to the center of the box at every time. \label{figura3}}
\end{figure}


\begin{figure}
\begin{center}
\stackunder[5pt]{\includegraphics[width=42mm]{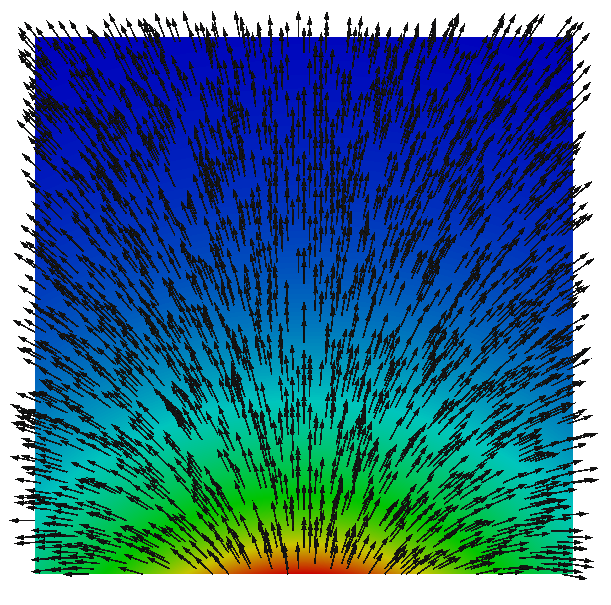}}{$t=2$s}%
\hspace{1pt}%
\stackunder[5pt]{\includegraphics[width=42mm]{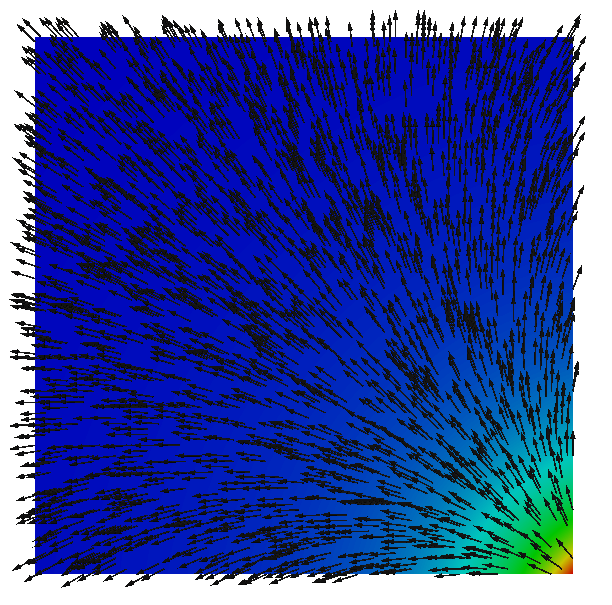}}{$t=2.25$s} 
\stackunder[5pt]{\includegraphics[width=42mm]{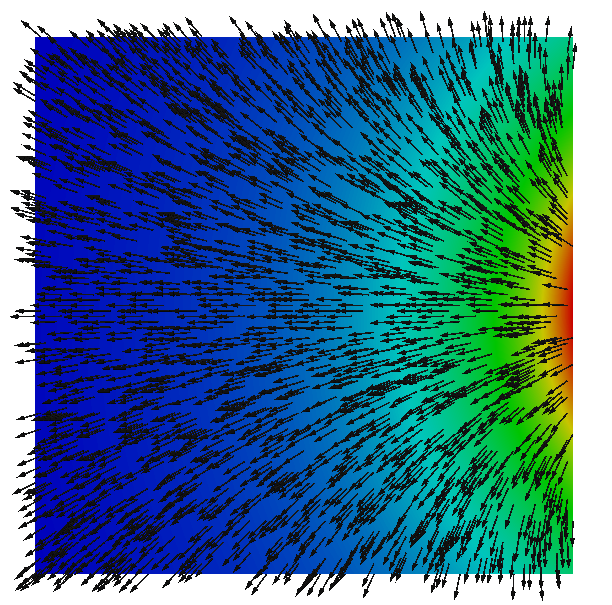}}{$t=2.5$s}%
\hspace{1pt}%
\stackunder[5pt]{\includegraphics[width=42mm]{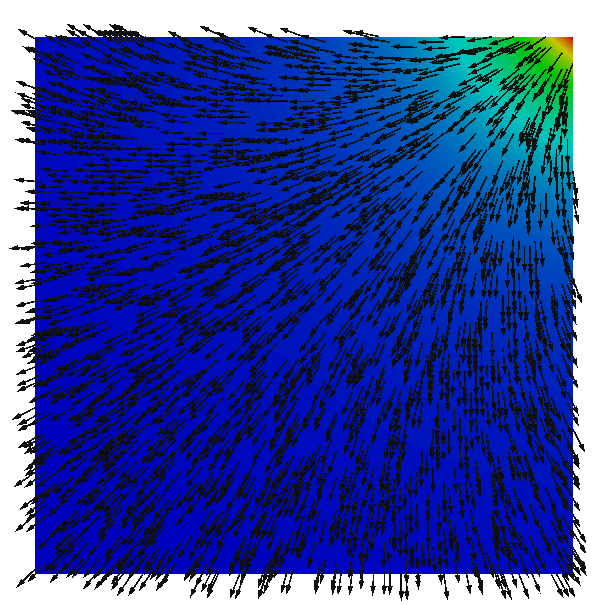}}{$t=2.75$s}%
\end{center}
\caption[Spinning magnet experiment: magnetization field]{\textbf{Spinning magnet experiment: magnetization field.}  Magnetization at times $t=2$s, $t=2.25$s, $t=2.5$s and $t=2.75$s for the experiment described in \S\ref{sub:spinmagnet}. The magnetization vectors are normalized for visualization, scale has been omitted for brevity. At time $t=2$ the dipole is at the bottom of the box, at time $t=2.25$ the dipole is pointing in the $(-1,1)$ direction, at time $t=2.5$ the dipole is pointing to $(-1,0)$, and at $t=2.75$ the dipole is pointing to $(-1,-1)$.\label{figura4}}
\end{figure}


\begin{figure}
\begin{center}
%
\stackunder[5pt]{\includegraphics[width=42mm]{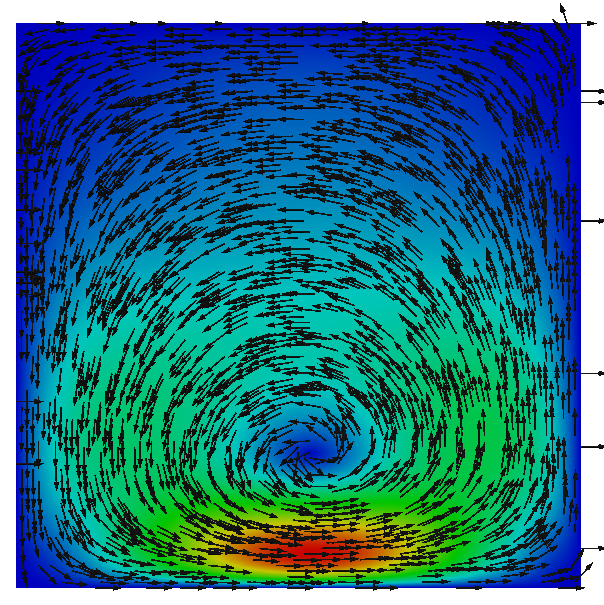}}{$t=2$s}%
\hspace{1pt}%
\stackunder[5pt]{\includegraphics[width=42mm]{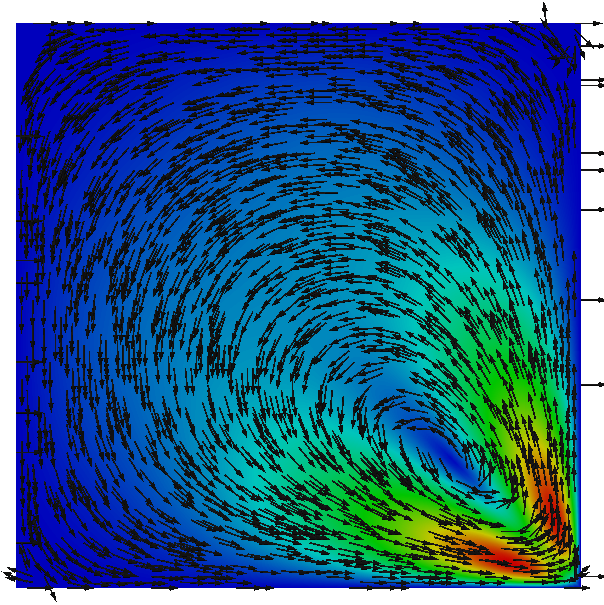}}{$t=2.25$s} 
\stackunder[5pt]{\includegraphics[width=42mm]{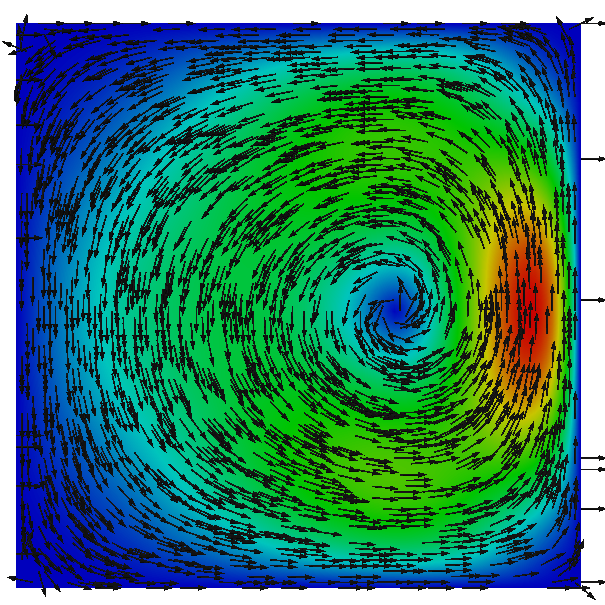}}{$t=2.5$s}%
\hspace{1pt}%
\stackunder[5pt]{\includegraphics[width=42mm]{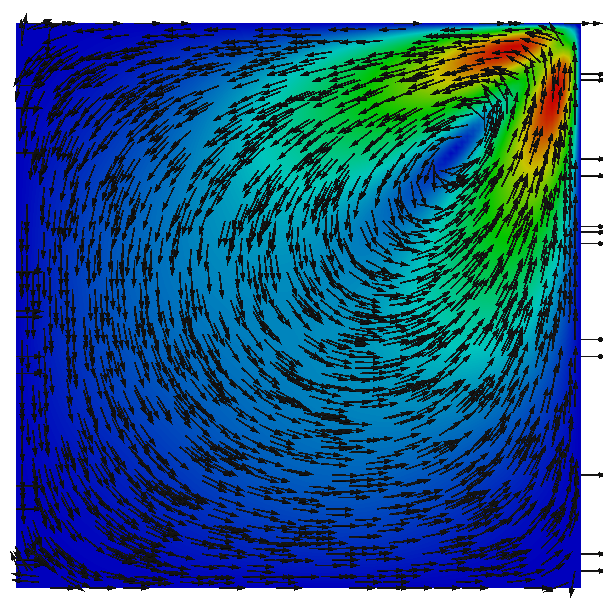}}{$t=2.75$s}%
\end{center}
\caption[Spinning magnet experiment: velocity field]{\textbf{Spinning magnet experiment: velocity field.} Same experiment as in Figure~\ref{figura4}, but here we plot the velocity field at times $t=2$s, $t=2.25$s, $t=2.5$s and $t=2.75$s. The velocity vectors are normalized to facilitate their visualization. It is easy to see that we have a positive circulation of the velocity field.\label{figura5}}
\end{figure}


\begin{figure}
\begin{center}
\stackunder[5pt]{\includegraphics[width=42mm]{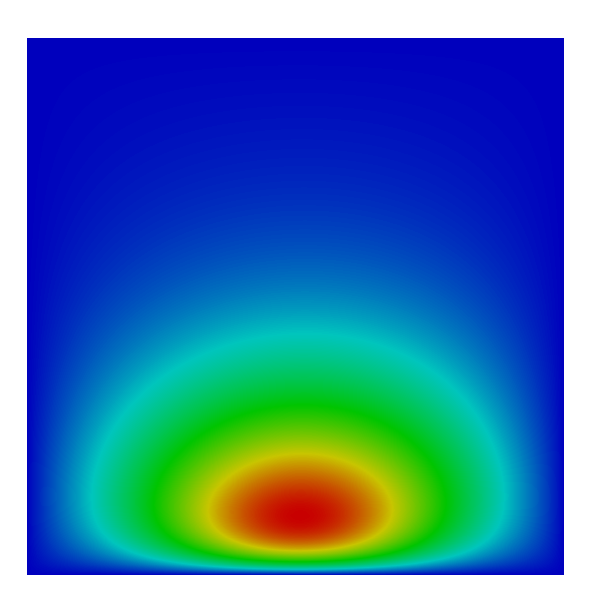}}{$t=2$s}%
\hspace{1pt}%
\stackunder[5pt]{\includegraphics[width=42mm]{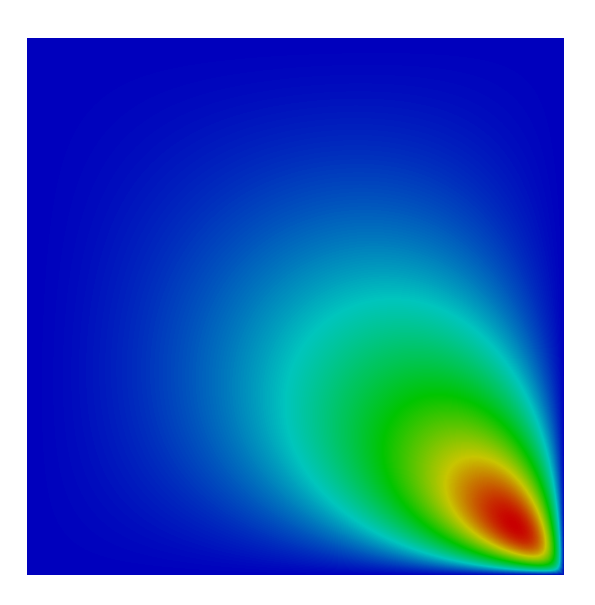}}{$t=2.25$s} 
\stackunder[5pt]{\includegraphics[width=42mm]{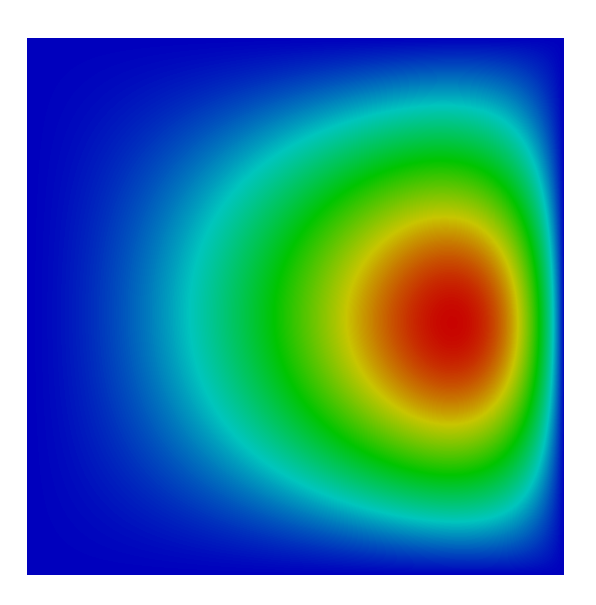}}{$t=2.5$s}%
\hspace{1pt}%
\stackunder[5pt]{\includegraphics[width=42mm]{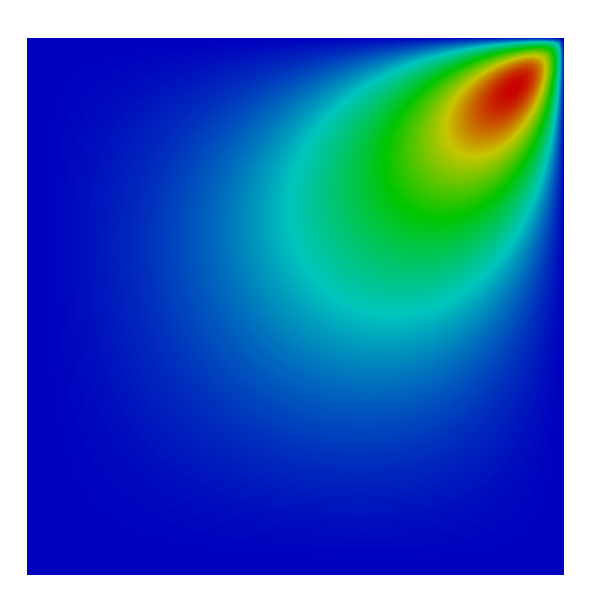}}{$t=2.75$s}%
\end{center}
\caption[Spinning magnet experiment: angular velocity field]{\textbf{Spinning magnet experiment: angular velocity field.} Angular velocity (spin) at times $t=2$s, $t=2.25$s, $t=2.5$s and $t=2.75$s for the experiment of \S\ref{sub:spinmagnet}. The angular velocity $\bv{w}$ takes positive values, i.e. $\bv{w} = w \hat k$ is a vector pointing out of the plane. This is consistent with the velocity field which shows a positive circulation ($\curl{}$) in Figure~\ref{figura5} and also with the fact that the magnetic dipole is spinning around the box in counter-clockwise direction. \label{figura6}}
\end{figure}

\subsection{Experiment 2: Ferrofluid pumping}
\label{sub:pump}
This experiment is related to what was the initial motivation for the development of ferrofluids \cite{Stephen1995}, which is pumping by means of magnetic fields without the action of any moving or mechanical device.  There are two well known methodologies used to induce pumping in ferrofluids: 
\begin{itemize}
 \item[\itemizebullet] Using a spatially-uniform but sinusoidal-in-time magnetic field \cite{Zahn95}.
 \item[\itemizebullet] Using a magnetic field that is varying in space and time \cite{Mao05}.
\end{itemize}
Here we will follow the second approach: we will apply a magnetic field which keeps its polarity (sign) constant, but its intensity changes in space, basically a periodic sequence of peaks (pulses) that travel in the direction we
want to induce linear momentum. This methodology was chosen because if we assume that we are always close to equilibrium, that is $\bv{m} \approx \permit \heff$, we get the following crude approximation for the Kelvin force:
\begin{align*}
\mu_0 (\bv{m}\cdot\nabla )\heff 
\approx \mu_0 \permit (\heff\cdot\nabla)\heff 
= \tfrac{\mu_0 \permit}{2} \nabla |\heff |^2 \, . 
\end{align*}
Consequently, if $\nabla\heff \approx  0$, then there is no force acting on the ferrofluid, and a methodology based on a spatially uniform magnetic field has very little chances of success. If we want to induce linear momentum in the ferrofluid we need to create steep gradients in the magnetic field. Technical details about the physical implementation of similar ideas, all of them using a magnetic field which resembles traveling pulses or a traveling magnetic wave, can be found in \cite{Mao05,Mao11,Crowley1989pat}.

The idea of a periodic sequence of pulses that travels in the direction we want to induce linear momentum was numerically recreated in a channel of 6 units of length, and one unit of height, using a total of $64$ dipoles: $32$ on the lower part of the duct and $32$ on the upper part, distributed uniformly through 2 units of length in its middle section. The configuration of magnetic dipoles, and the intensities associated to them, in the middle section of the channel is sketched in  Figure~\ref{figura7}. The corresponding applied magnetic field is computed as $\ha = \sum_{s = 1}^{64} \alpha_s \nabla\hapot_s$, where $\alpha_s = |\text{sin}(\omega t - \kappa x_s)|^{2q}$, $\omega =2\pi f$ with $f = 10\text{Hz.}$, $q = 5$, $\kappa = 2 \pi / \lambda$ with $\lambda = 1.0$, creating the effect of pulses traveling from left to right. In Figure~\ref{travelingpulses} we plot the function $\alpha_s = |\text{sin}(\omega t - \kappa x_s)|^{2q}$ in terms of $x_s$, so that the reader can appreciate the shape of the 
pulses (intensity of the magnetic field). Some numerical results are depicted in Figures~\ref{figura9}--\ref{figura10}.

\begin{figure}
\setlength\fboxrule{1pt}
\begin{center}
\fbox{\includegraphics[scale=0.16]{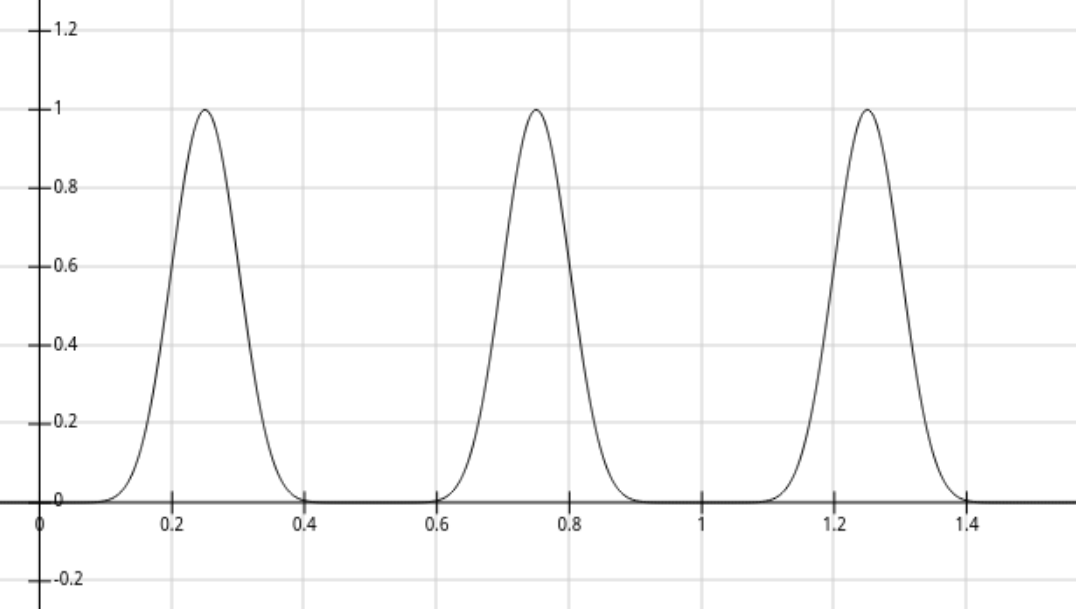}}
\end{center}
\caption[Plot of the function $\alpha_s = |\text{sin}(\omega t - \kappa x_s)|^{2q}$]{\textbf{Plot of the function $\alpha_s = |\text{sin}(\omega t - \kappa x_s)|^{2q}$}. Here $x_s$ is the horizontal axis, $t = 0$, $\kappa = 2\pi$, and $q = 5$. \label{travelingpulses}}  
\end{figure}

\begin{figure}
\setlength\fboxrule{0pt}
\begin{center}
\fbox{\includegraphics[scale=0.30]{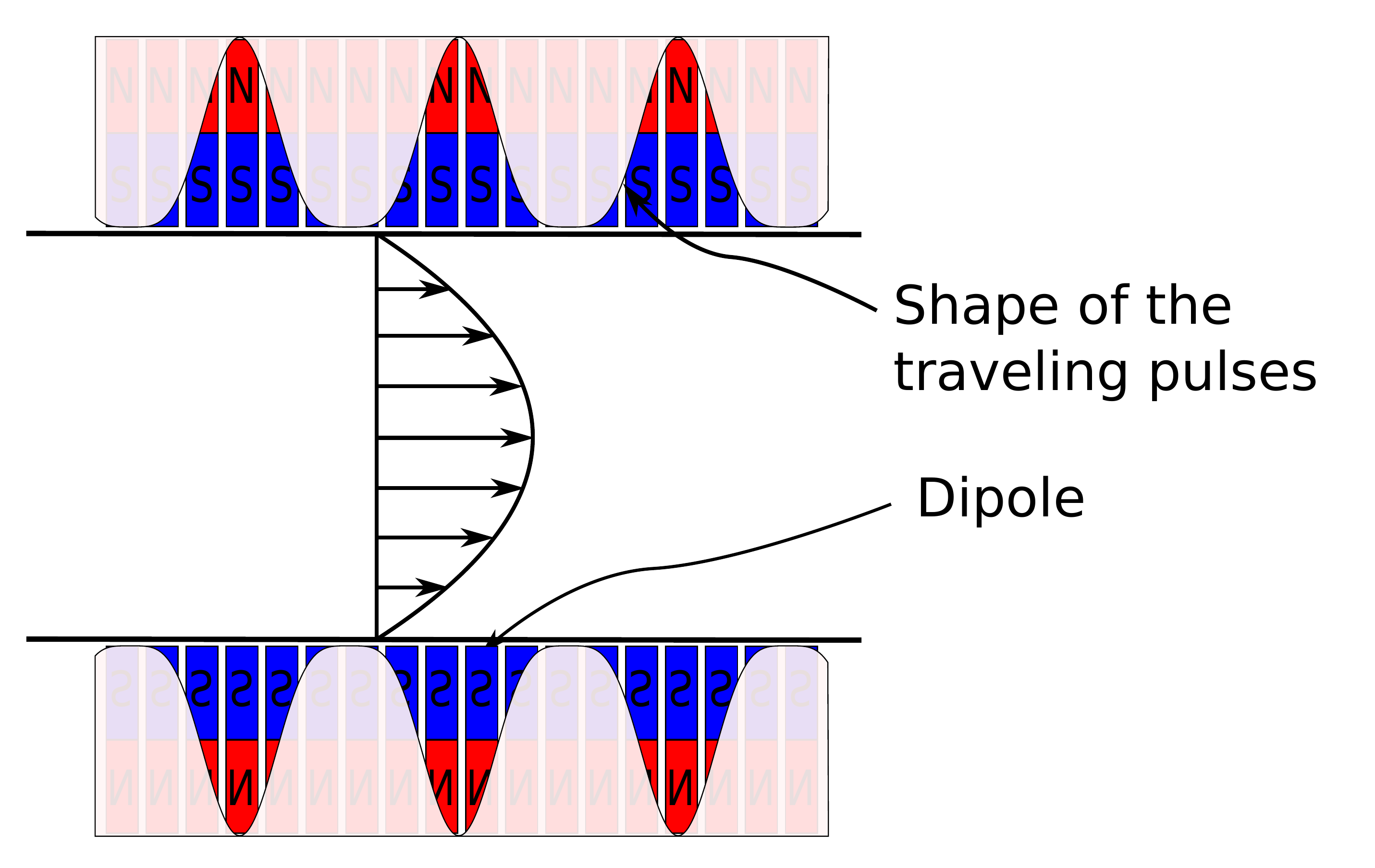}} 
\end{center}
\caption[Ferrofluid pumping experiment: setup (sketch) of the experiment]{\textbf{Ferrofluid pumping experiment: setup (sketch) of the experiment}. Sketch of the middle section of the channel showing the setup of the ferrofluid pumping experiment of \S\ref{sub:pump}. The vertical bars represent magnetic dipoles (see formula \eqref{dipole2D}) located in the center of the channel. All the dipoles have the same polarity, which does not change in the $x$ direction, and their intensity is represented by the unshaded region. These dipoles do not move, but their intensity changes in time, reproducing the effect of traveling pulses. \label{figura7}}  
\end{figure}


The magnitude of the magnetization is only relevant in the middle as it can be appreciated in Figure~\ref{figura8}. A noteworthy outcome of these experiments is displayed in Figure~\ref{figura10}: where the spin does not seems to help (it induces a flow in the opposite direction), which is a very intriguing result. More precisely, the spin in the upper part of Figure~\ref{figura10} is negative and the spin in the lower part of the channel is positive, which will induce flow from right to left as it was shown numerically in \cite{Tom13}. This is an unexpected outcome which requires further investigation.

\begin{figure}
\begin{center}
\includegraphics[scale=0.4]{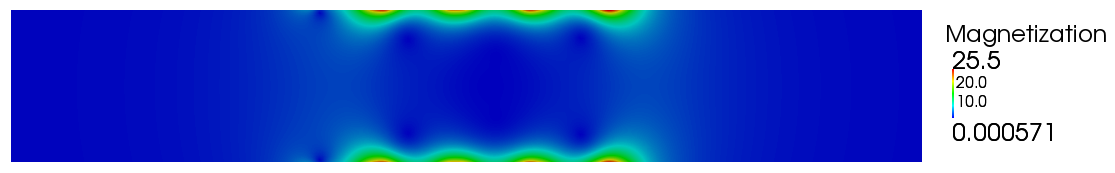}
\end{center}
\caption[Ferrofluid pumping experiment: detail of the magnetization in the middle section of the channel.]{
\textbf{Ferrofluid pumping experiment: detail of the magnetization in the middle section of the channel.} The intensity of the magnetization is relevant close to the upper and lower walls, and is negligible in the center of the channel. This means that most of the force is exerted close to the upper and lower walls. \label{figura8}}
\end{figure}

\begin{figure}
\begin{center}
\includegraphics[scale=0.26]{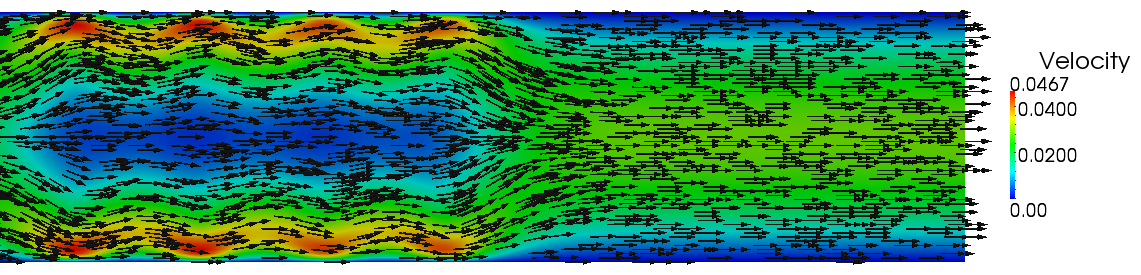}
\end{center}
\caption[Ferrofluid pumping experiment: detail of the velocity in the middle and final section of the channel]{\textbf{Ferrofluid pumping experiment: detail of the velocity in the middle and final section of the channel.} It can be appreciated that on the region affected by the external magnetic fields (left) the velocity field is not very aligned, but it becomes much more uniform as we move to the right of the picture (exit section of the channel). \label{figura9}} 
\end{figure}
 
\begin{figure}
\begin{center}
\includegraphics[scale=0.4]{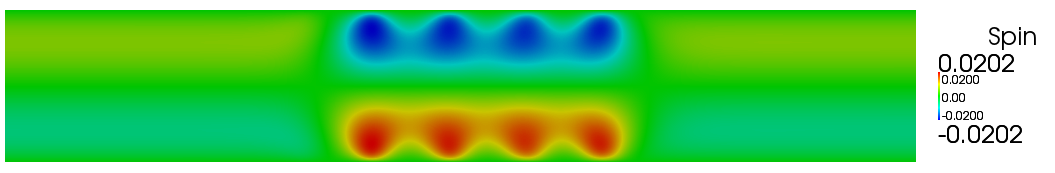}
\end{center}
\caption[Ferrofluid pumping experiment: angular velocity]{\textbf{Ferrofluid pumping experiment: angular velocity.} Note that it is positive (counterclockwise) in the lower part of the channel and negative (clockwise) in the upper part. Such profile of angular velocities does not help in the pumping process from left to right. This is a quite intriguing effect. \label{figura10}}
\end{figure}

\subsection{Experiment 3: Ferromagnetic stirring of a passive scalar}
\label{sub:stirr}
For low Reynolds numbers ($\textsl{Re} \backsimeq 1$) one of the bottlenecks of chemical reactions is mixing, in particular when the effects of diffusion are not strong enough. Slow mixing can lead to very long waiting times or a very poor completion of chemical reactions. Flows inside microfluidic devices (usually called lab-on-chip devices) have quite low Reynolds numbers and there is a growing interest in accelerating the mixing process with the addition of active and passive mixers. Passive mixer designs can range from simple grooved channels and Y-shaped channels to much more sophisticated ideas \cite{Wong2003,Lee2009,Lin2011}. Among active mixers we can find mixing by means of ferromagnetic particles \cite{Fu2010,Tsai2009}. Here we will illustrate the idea of ferrofluid mixing by adding the following convection diffusion equation:
\begin{align}\label{passiscalareq}
c_t + \bv{u} \cdot \nabla c - \alpha \Delta c = 0  
\end{align}
where $\bv{u}$ is the velocity from the equations of ferrohydrodynamics  \eqref{eq:ferroeq}, and $\alpha = 0.001$ for all our experiments. We consider such a small diffusion so that mixing depends mostly on advection. Equation \eqref{passiscalareq} is uncoupled from the system \eqref{eq:ferroeq}, meaning that no quantity in \eqref{eq:ferroeq} depends on $c$ (passive scalar).

We explore the possibility of designing an active mixer by applying a time-dependent magnetic field to a ferrofluid contained in $\Omega=(0,1)^2$, and we track the evolution of the concentration $c$ satisfying equation \eqref{passiscalareq}. As in the previous experiments we make no attempt to use realistic scalings or try to relate the numerical results with any physical situation (as microflows). The main goal is to provide a proof of concept of some ideas and discard those which have no room for additional improvement. Designing a functional stirrer is a deceptively simple task. For instance, using the setup of Figure~\ref{figura3} which involves moving physically the magnet does not work, at least in our experience, even if we relax the viscosity and use values much smaller than unity, or if we make the magnet spin around the box at a much higher angular velocity.

As a first attempt, we let the applied magnetizing field be that of two dipoles with alternating polarity, more precisely
\begin{align}\label{alternating}
\ha = \sum_{s = 1}^{2} \alpha_s \nabla\hapot_s \, , \ \ 
\text{where} \ \ 
\alpha_1 = \alpha_0 \, \text{sin}(\omega t) \, , \ \ 
\alpha_2 = \alpha_0 \, \text{sin}(\omega t + \pi/2) \, ,  
\end{align}
i.e., the two dipoles have a phase mismatch of $\pi/2$; see Figure~\ref{figura19} for the sketch of this setup. Here $\omega = 2 \pi f$, is the angular velocity of the periodic excitation, $f$ is the frequency, and $\alpha_0$ is the amplitude. The magnetization and velocity fields induced by this setup are displayed in Figures~\ref {figura16} and \ref{figura18}. Some results regarding the ability of this setup to actually induce mixing of the passive scalar $c$ can be appreciated in Figure~\ref{figura17}, where we have used $f = 20 \text{Hz.}$, $\alpha_0 = 5.0$, and $\nu = \nu_r = 0.5$. This setup does indeed shake the fluid as it can be appreciated in Figures \ref{figura18} and \ref{figura17}, but the magnitude of the velocity is very small, less than $10^{-2}$ for most of our experiments, and is not very sensitive to the inputs, meaning that increasing the value of $\omega$ and $\alpha_0$, or using a very small viscosity, does not significantly increase the velocity and as a consequence the mixing 
will be very poor (at least in this context, which is that of an homogeneous fluid). For instance, the results of setting $f = 40 \text{Hz.}$ and $\nu = \nu_r =0.1$, are very similar to those of Figure~\ref{figura17}.

\begin{figure}
\begin{center}
\includegraphics[scale=0.35]{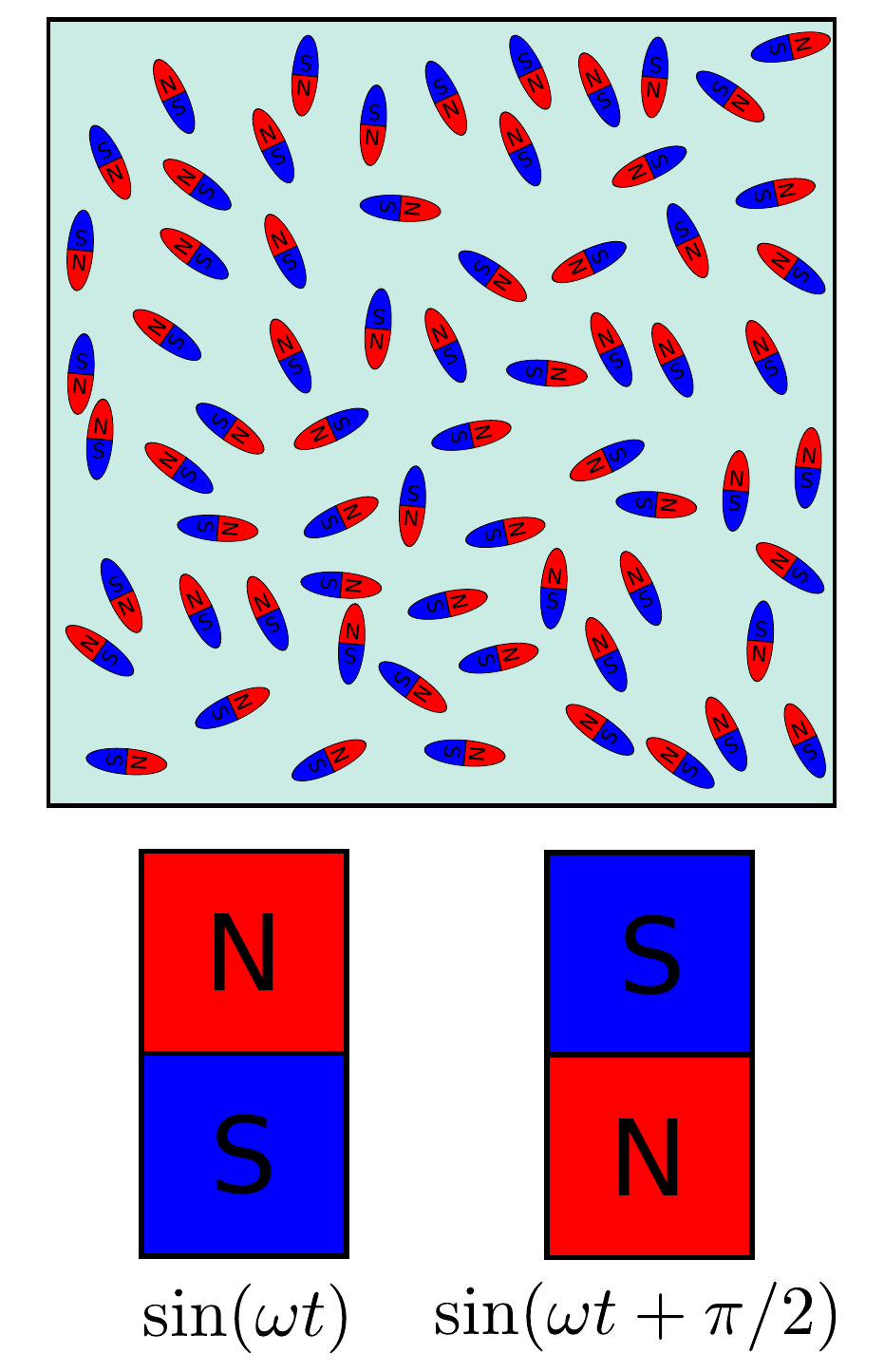}
\end{center}
\caption[Ferromagnetic stirring (first approach): setup]{\textbf{Ferromagnetic stirring (first approach): setup.} The magnetic field satisfies \eqref{alternating}.\label{figura19}}
\end{figure}


\begin{figure}
\begin{center}
\begin{tabular}{ccc}
\includegraphics[width=42mm]{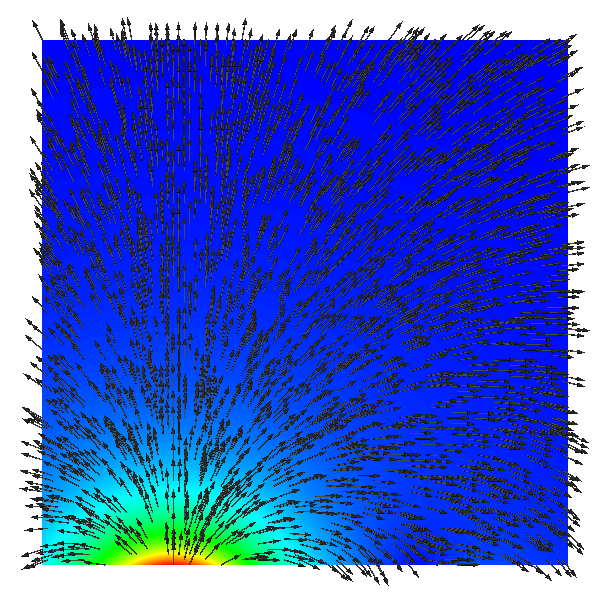}&
\includegraphics[width=42mm]{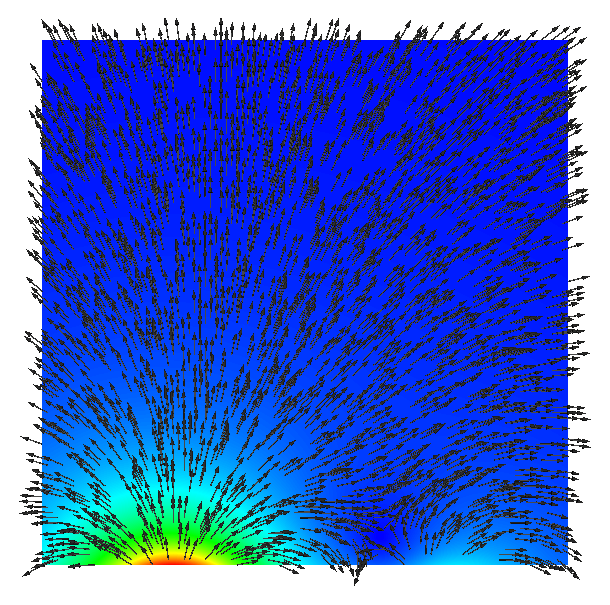}&
\includegraphics[width=42mm]{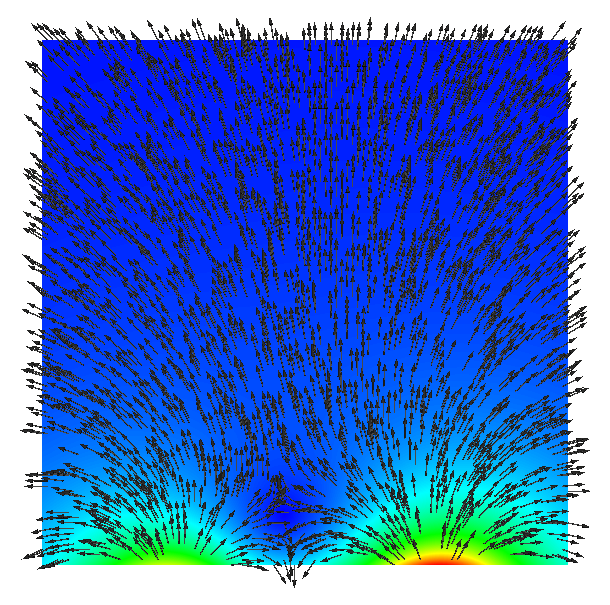}\\
\includegraphics[width=42mm]{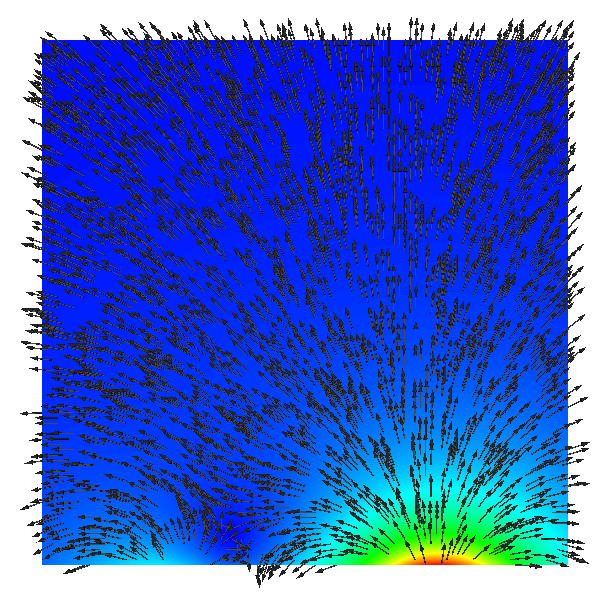}&
\includegraphics[width=42mm]{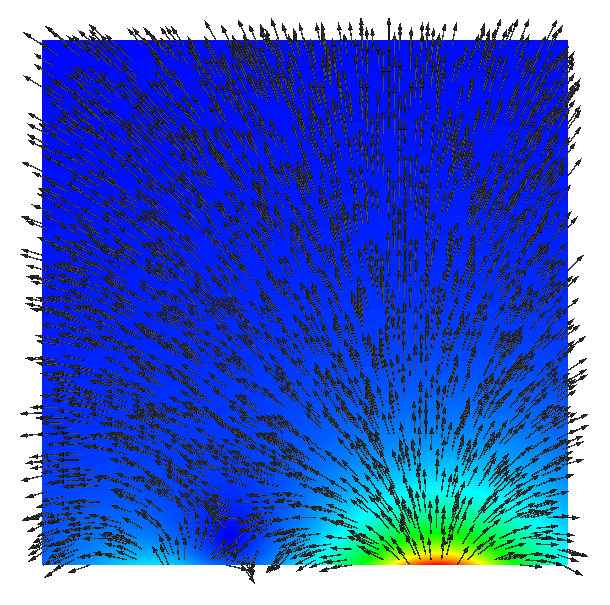}&
\includegraphics[width=42mm]{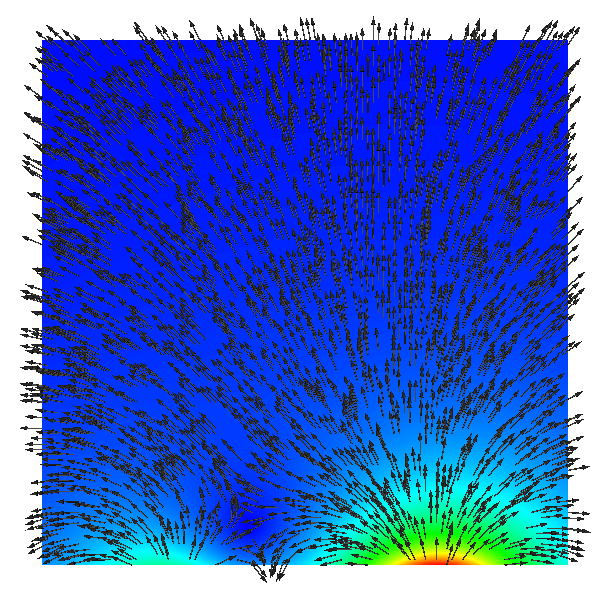}\\
\includegraphics[width=42mm]{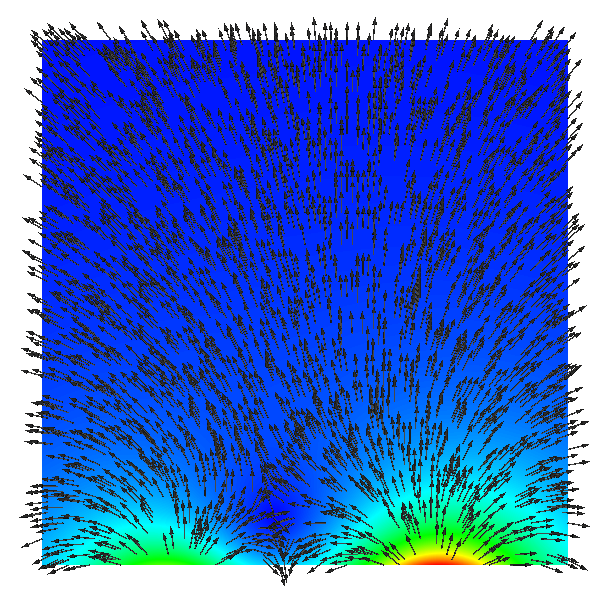}&
\includegraphics[width=42mm]{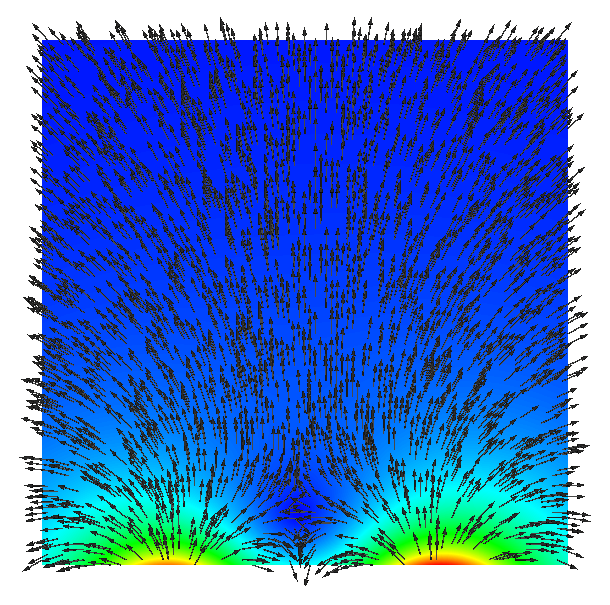}&
\includegraphics[width=42mm]{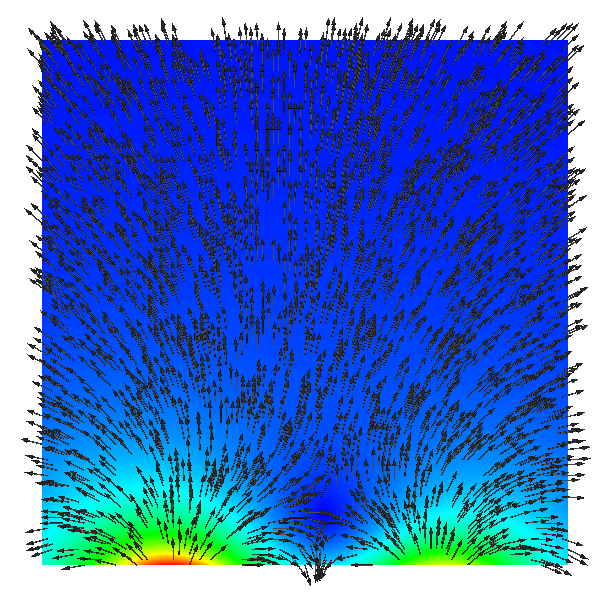}\\
\includegraphics[width=42mm]{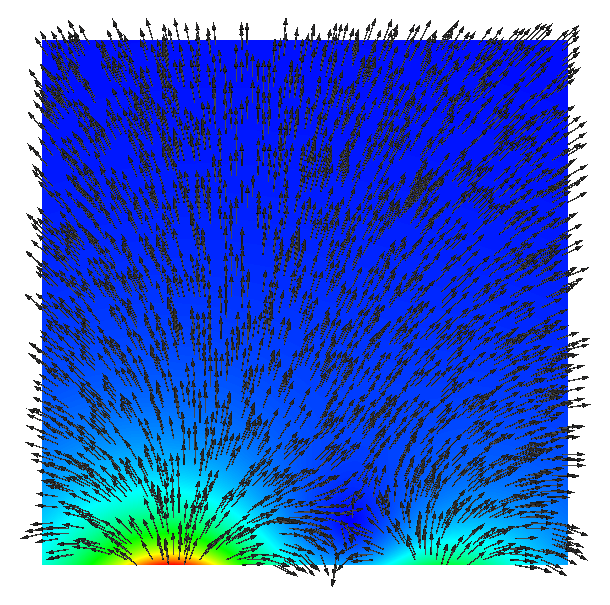}&
\includegraphics[width=42mm]{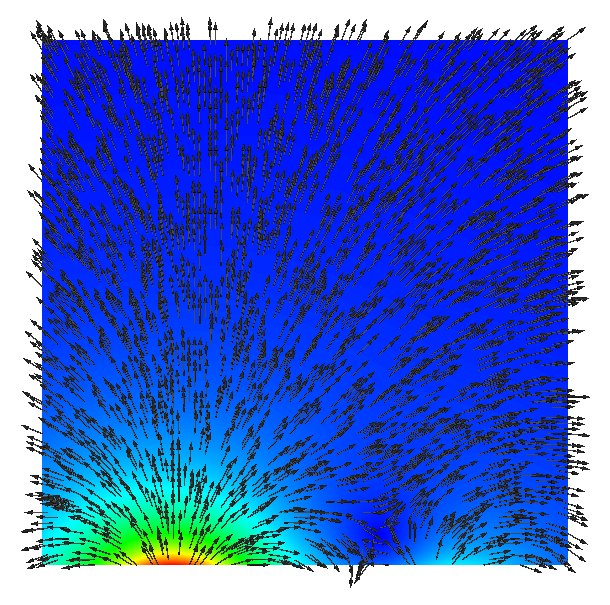}&
\includegraphics[width=42mm]{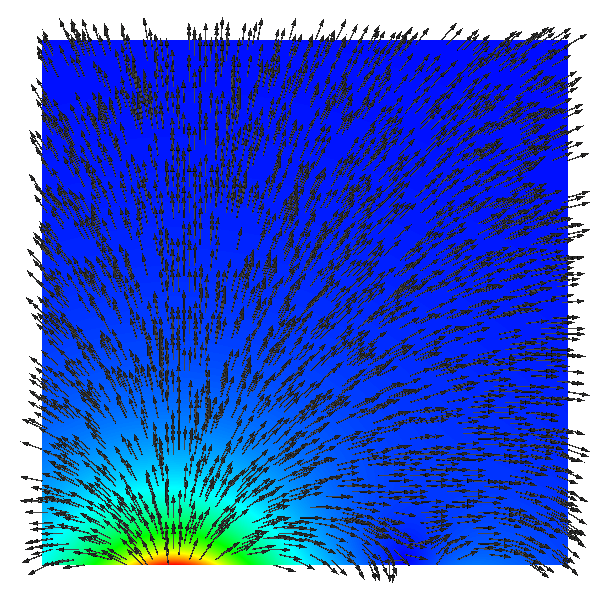}
\end{tabular}
\end{center}
\caption[Ferromagnetic stirring (first approach): evolution of the magnetization field]{\textbf{Ferromagnetic stirring (first approach): evolution of the magnetization field.} This corresponds to the setup described in figure \ref{figura19} during a half period $\tfrac{1}{2f}$. The magnetic field satisfies \eqref{alternating}. The magnets are fixed on the bottom of the container and so is the location where the maximum intensities occur. \label{figura16}}
\end{figure}


\begin{figure}
\begin{center}

  \begin{tabular}{cccc}

    \includegraphics[width=33mm]{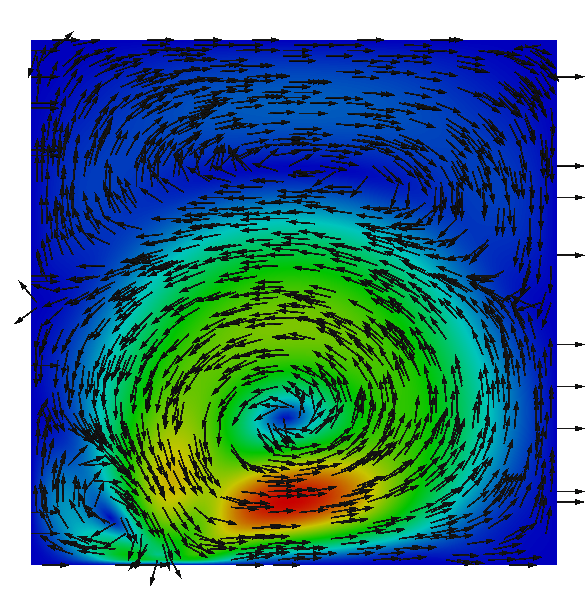}&
    \includegraphics[width=33mm]{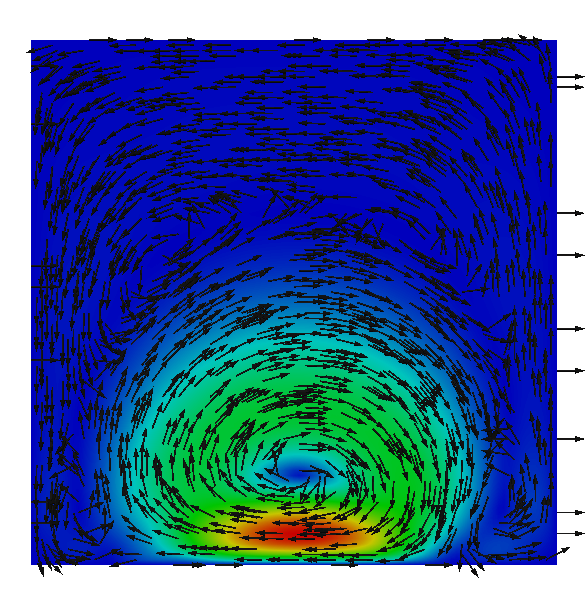}&
    \includegraphics[width=33mm]{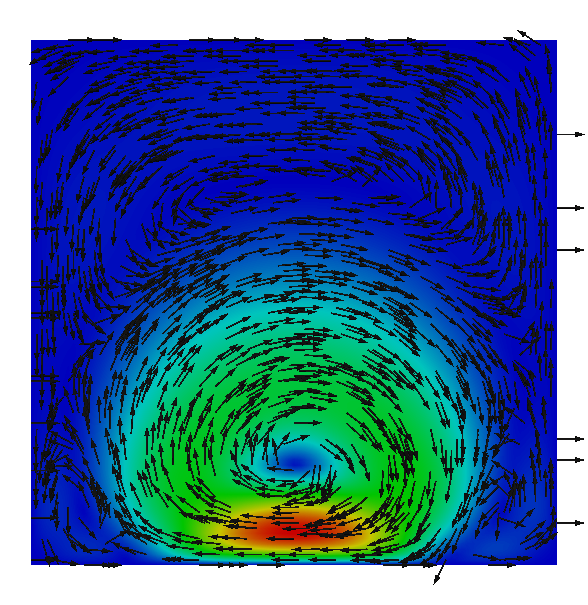}&
    \includegraphics[width=33mm]{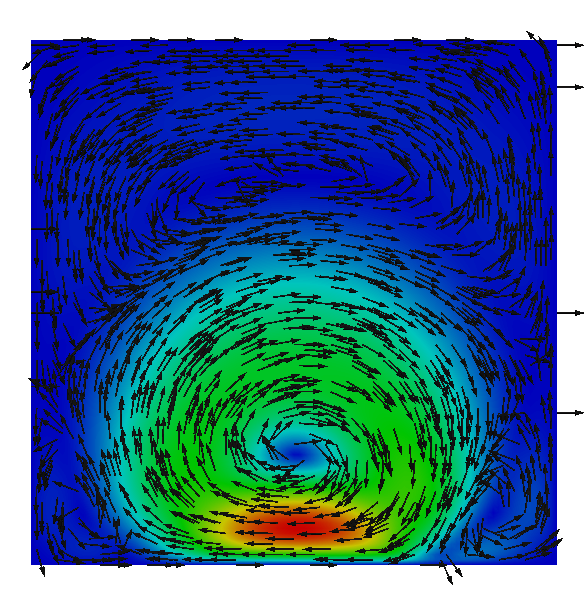}\\
    \includegraphics[width=33mm]{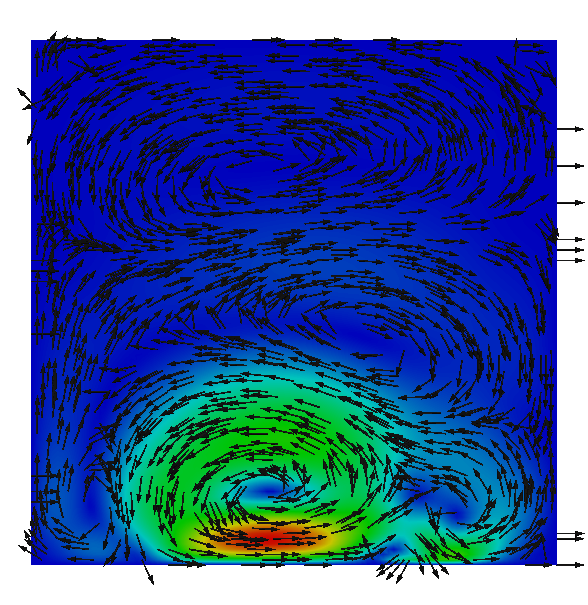}&
    \includegraphics[width=33mm]{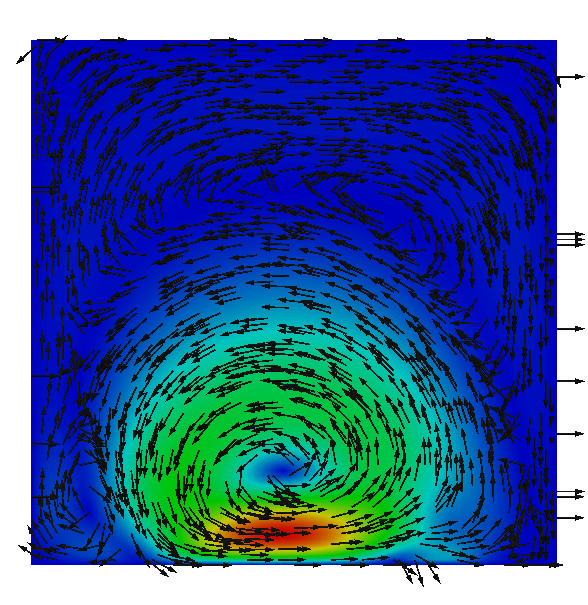}&
    \includegraphics[width=33mm]{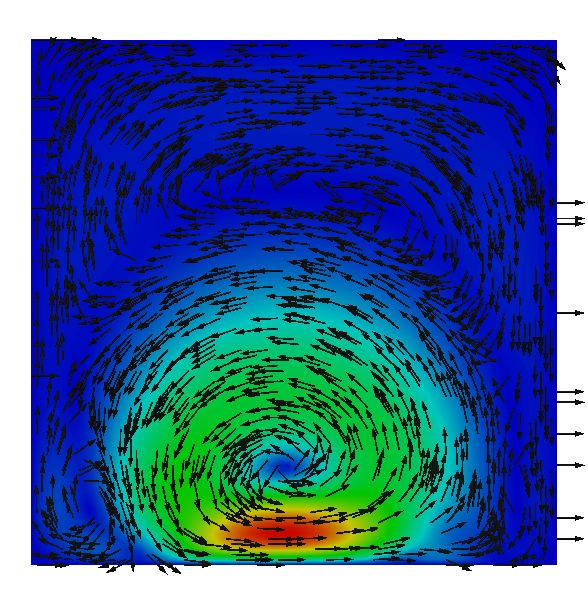}&
    \includegraphics[width=33mm]{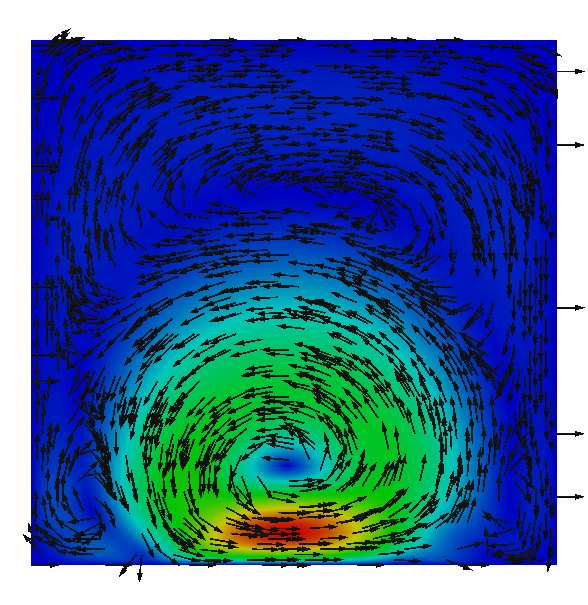}\\
    \includegraphics[width=33mm]{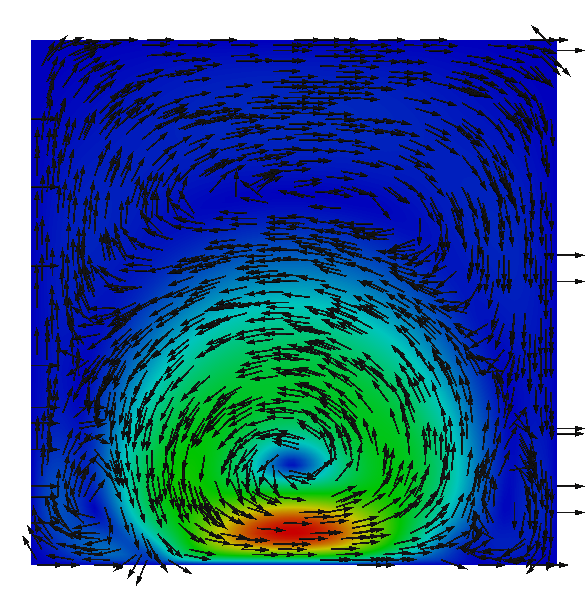}&
    \includegraphics[width=33mm]{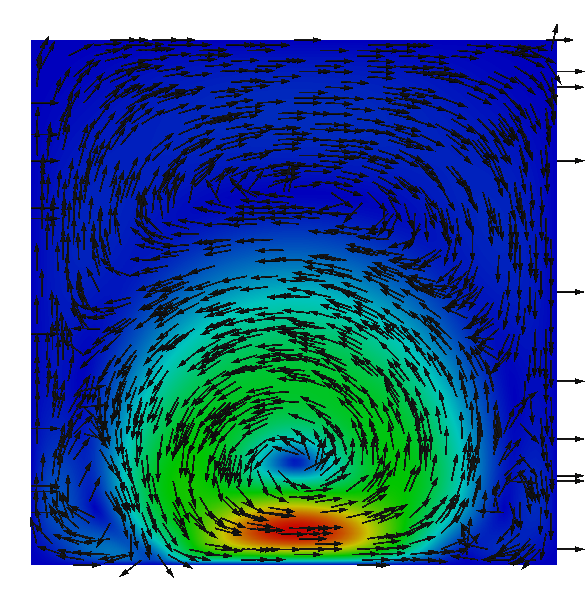}&
    \includegraphics[width=33mm]{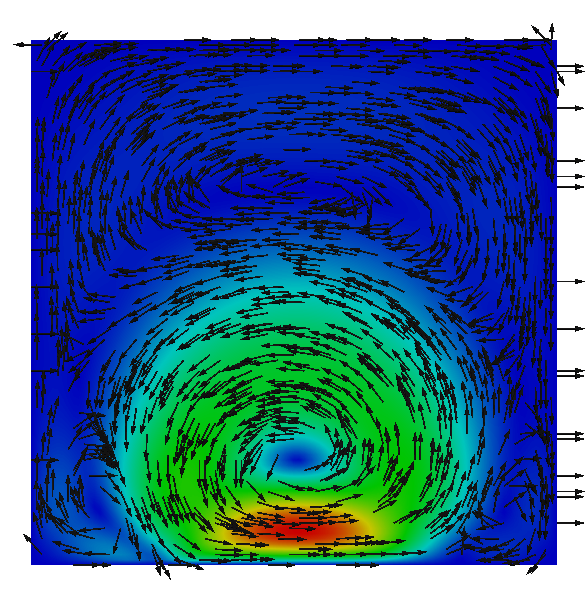}&
    \includegraphics[width=33mm]{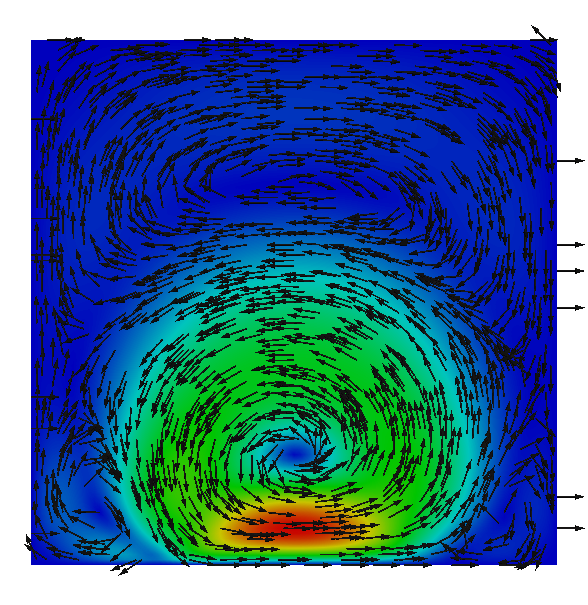}

  \end{tabular}
\end{center}
\caption[Ferromagnetic stirring (first approach): evolution of velocity]{\textbf{Ferromagnetic stirring (first approach): evolution of velocity.} Evolution of the velocity during a half period $\tfrac{1}{2f}$ for the stirring experiment of \S\ref{sub:stirr} (see figure \ref{figura19} for its setup). The magnetic field satisfies \eqref{alternating}. Reading from left to right and top to bottom: there is reversal of the circulation (curl) of the flow between the first two figures, and also between the fourth and fifth figure. Here the scale has been omitted, but the maximum velocity is of the order of $ \simeq 0.008$, enough to induce some mixing of the passive scalar as it can be appreciated in Figure \ref{figura17}, but far from the quality of mixing achieved for instance in Figure \ref{figura15} using the traveling wave approach. \label{figura18}}
\end{figure}

\begin{figure}
\begin{center}
    \setlength\fboxsep{0pt}
    \setlength\fboxrule{1pt}
  \begin{tabular}{cccc}

    \fbox{\includegraphics[width=40mm]{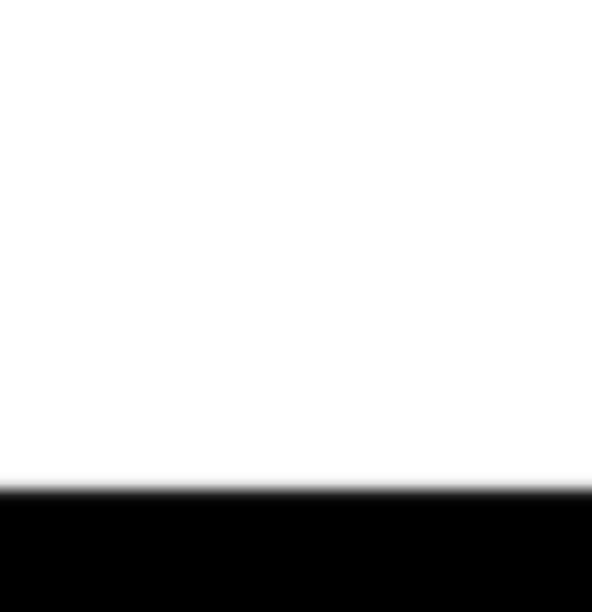}}&
    \fbox{\includegraphics[width=40mm]{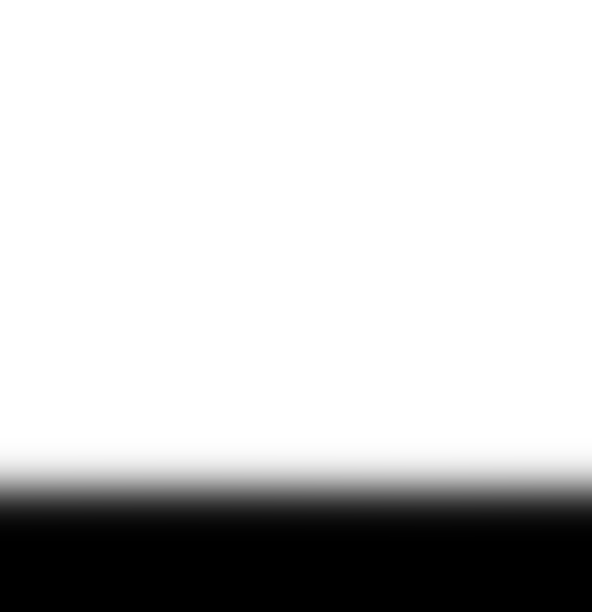}}&
    \fbox{\includegraphics[width=40mm]{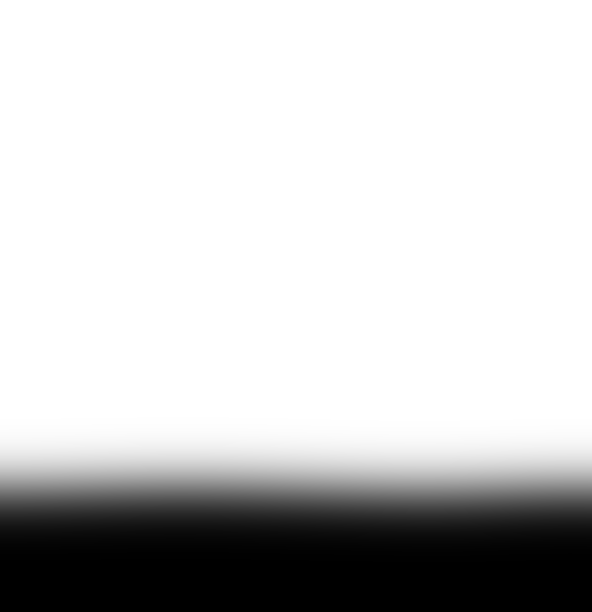}}\\
    \fbox{\includegraphics[width=40mm]{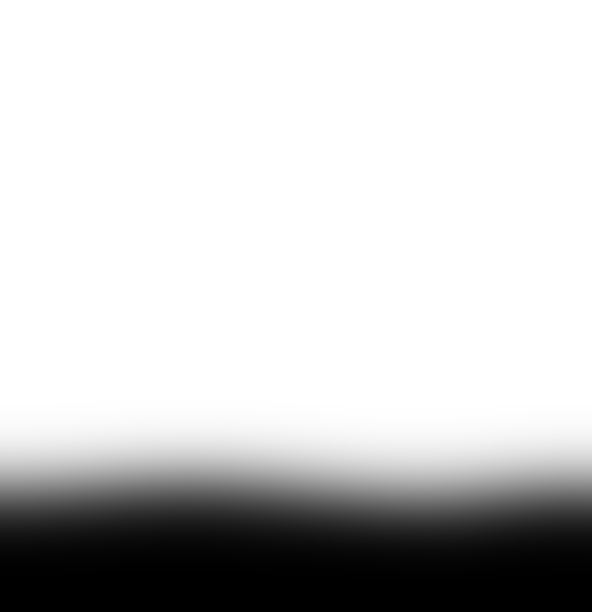}}&
    \fbox{\includegraphics[width=40mm]{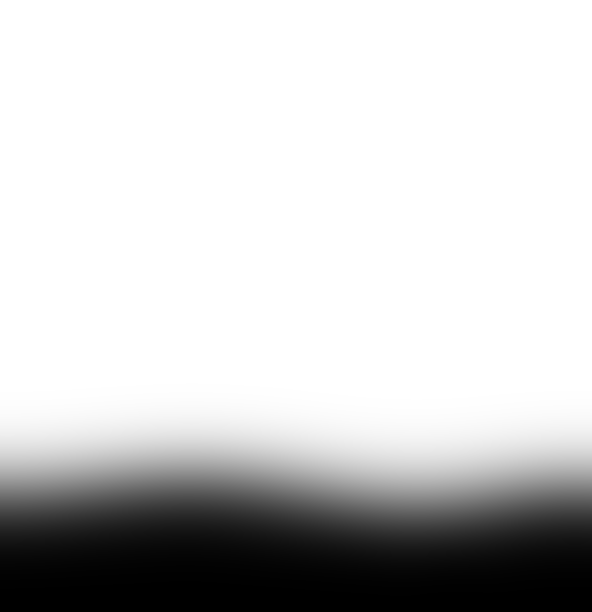}}&
    \fbox{\includegraphics[width=40mm]{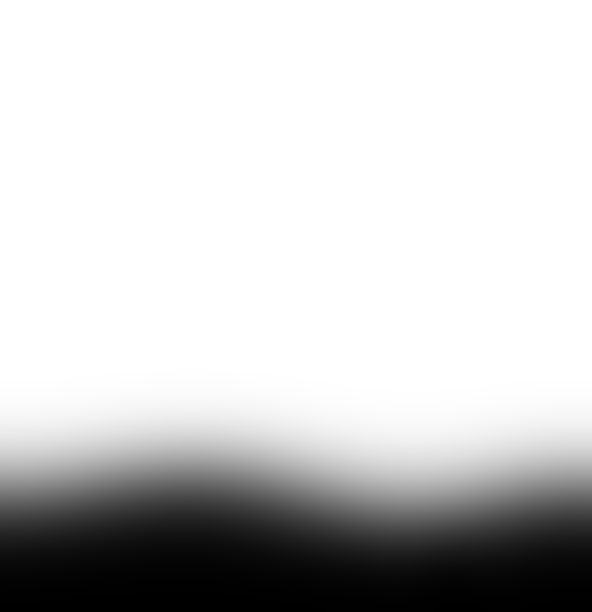}}\\
    \fbox{\includegraphics[width=40mm]{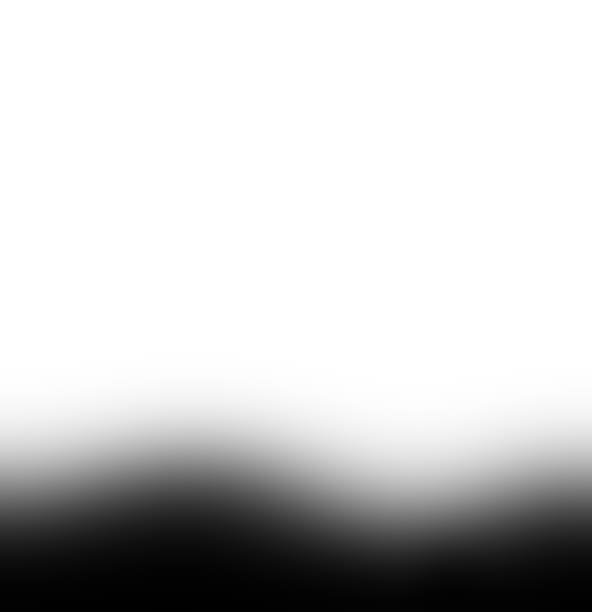}}&
    \fbox{\includegraphics[width=40mm]{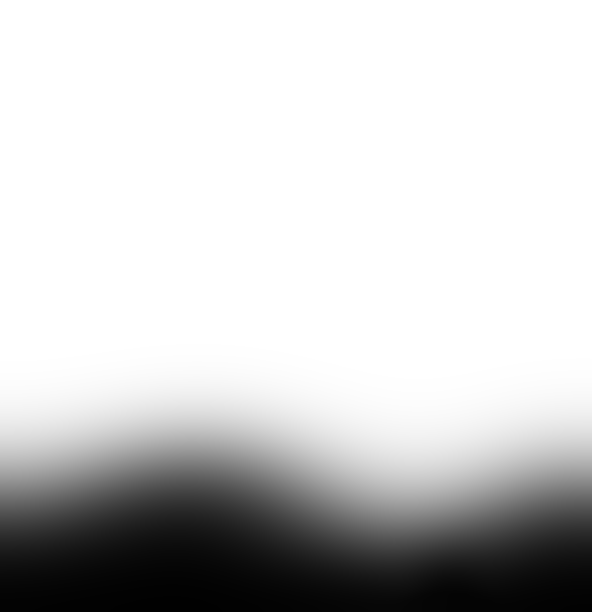}}&
    \fbox{\includegraphics[width=40mm]{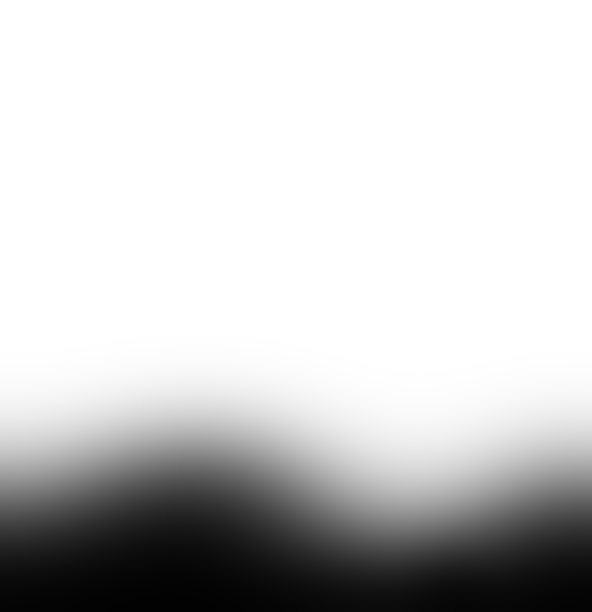}}\\
    \fbox{\includegraphics[width=40mm]{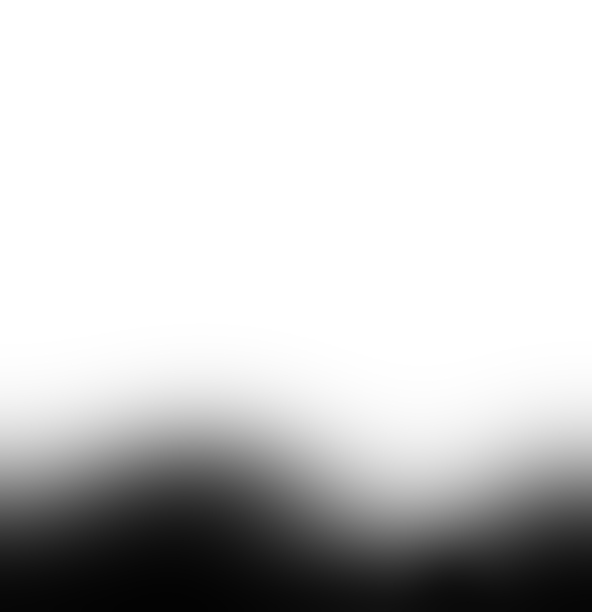}}&
    \fbox{\includegraphics[width=40mm]{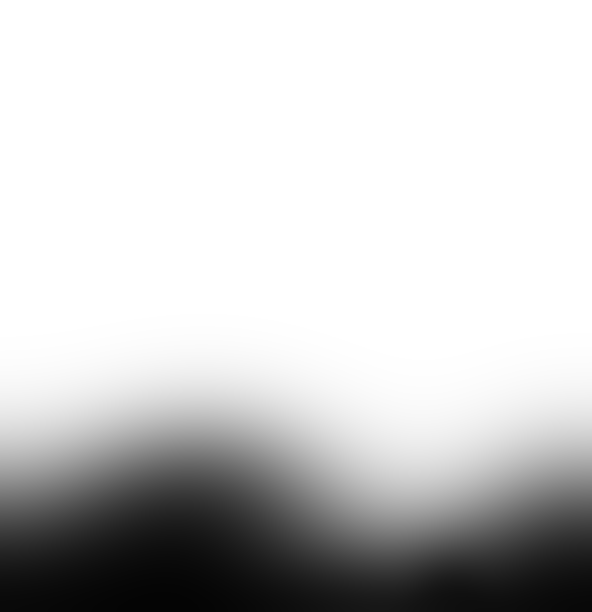}}&
    \fbox{\includegraphics[width=40mm]{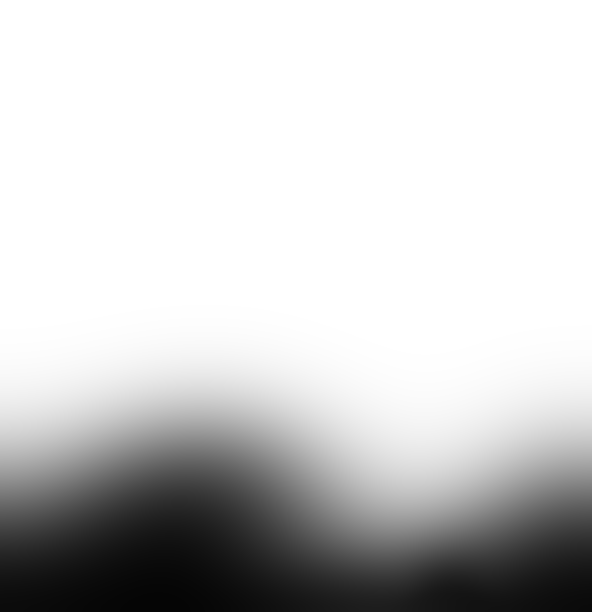}}
  \end{tabular}
\end{center}
\caption[Ferromagnetic stirring (first approach): evolution of the passive scalar]{\textbf{Ferromagnetic stirring (first approach): evolution of the passive scalar.} Evolution of the concentration using the setup described in Figure \ref{figura19}; the magnetizing field satisfies \eqref{alternating}. The sequence of figures shown here was for a total of 4 seconds (100 periods). We start with concentration equal to one on the bottom of the box (the black region on the bottom). These results were obtained with $f = 20\text{Hz}$, and $\nu = \nu_r =0.5$. \label{figura17}}
\end{figure}

%
%
%

A second, and successful, approach is that of a traveling wave. Consider eight dipoles on the lower edge of the box pointing upwards, much like the setup of Figure~\ref{figura19}, so that the magnetic field is given by 
\begin{align}\label{travewavemethod}
\ha = \sum_{s = 1}^{8} \alpha_s \nabla\hapot_s \, , \ \
\alpha_s = \alpha_0 \, |\text{sin}(\omega t - \kappa x_s)| \, , \ \
\kappa = 2 \pi / \lambda \, , \ \
\lambda = 0.8 \, .
\end{align}
Figure~\ref{figura12} displays some promising results. There is mixing, even under quite unfavorable conditions: $f = 20\text{Hz.}$, $\alpha_0 = 5.0$, and $\nu = \nu_r = 0.5$. This setting is sensitive to the inputs, as $\omega$ and $\alpha_0$ increase and the viscosity diminishes, the mixing improves. The magnetization and velocity profile for this setting are shown in Figures \ref{figura13} and \ref{figura14}, respectively. Much more striking results can be found in Figure~\ref{figura15}, where we use a higher value for $\omega$, a lower viscosity, a higher intensity $\alpha_s$, and we also run the simulation for a longer time.

\begin{figure}
\begin{center}
    \setlength\fboxsep{0pt}
    \setlength\fboxrule{1pt}
  \begin{tabular}{cccc}

    \fbox{\includegraphics[width=30mm]{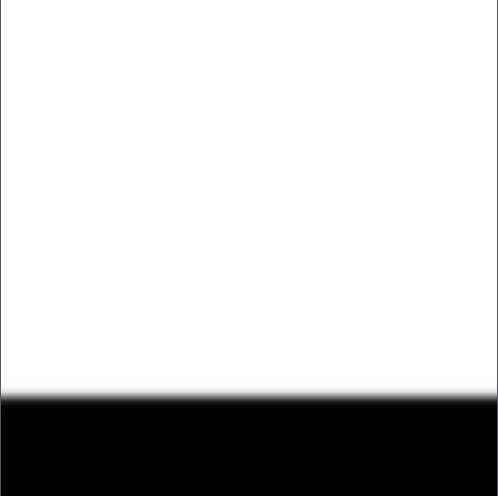}}&
    \fbox{\includegraphics[width=30mm]{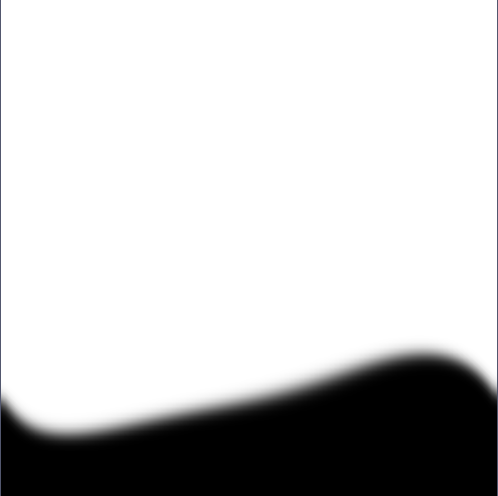}}&
    \fbox{\includegraphics[width=30mm]{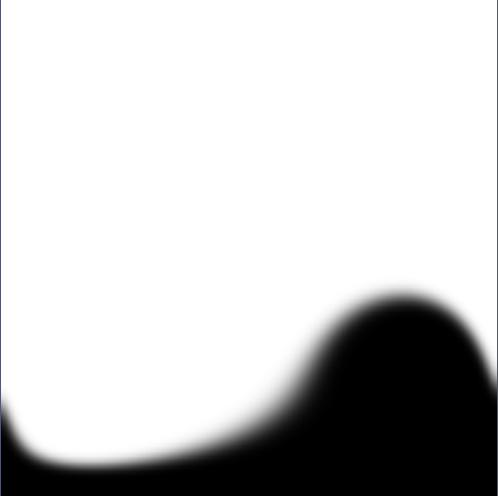}}&
    \fbox{\includegraphics[width=30mm]{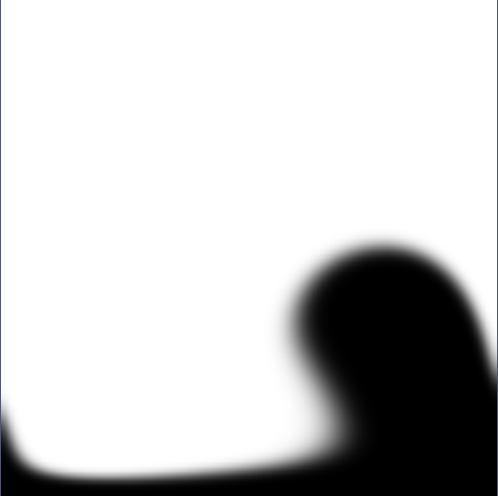}}\\

    \fbox{\includegraphics[width=30mm]{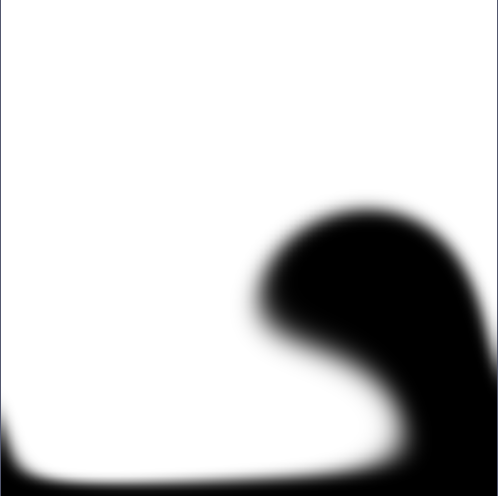}}&
    \fbox{\includegraphics[width=30mm]{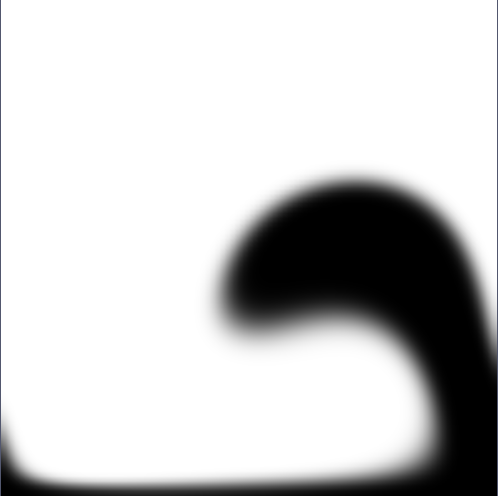}}&
    \fbox{\includegraphics[width=30mm]{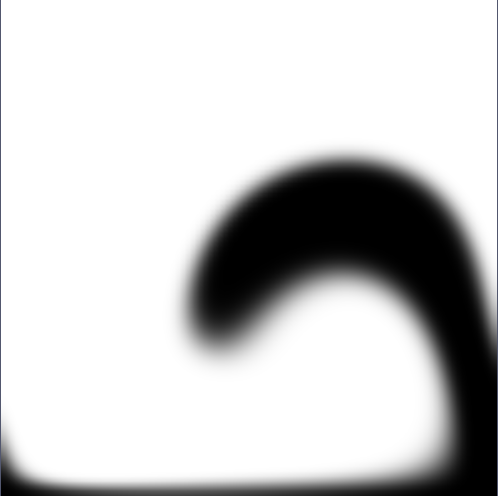}}&
    \fbox{\includegraphics[width=30mm]{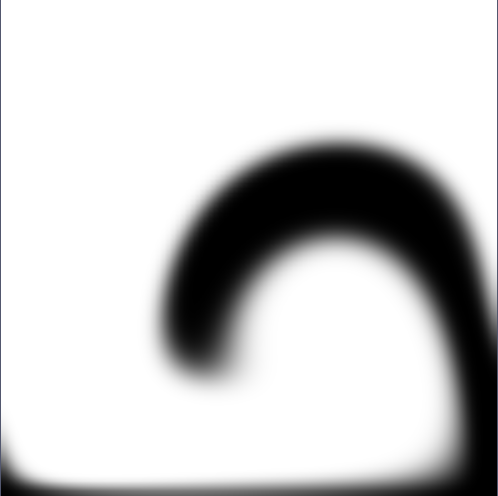}}
  \end{tabular}
\end{center}
\caption[Ferromagnetic stirring (second approach): evolution of the passive scalar]{\textbf{Ferromagnetic stirring (second approach): evolution of the passive scalar.} Contour of the passive scalar satisfying equation \eqref{passiscalareq} from time $t = 0$ to $t = 1.0$ (20 periods). The magnetic field satisfies \eqref{travewavemethod}. We start with concentration equal to one on the bottom of the box, the velocity field induced by the traveling wave of magnetization drags concentration on the bottom and takes it upward. These results were obtained with $f = 20\text{Hz.}$, $\nu = \nu_r =0.5$, and $\alpha_0 = 5.0$. \label{figura12}} 
\end{figure}

%
%

\begin{figure}
\begin{center}
  \begin{tabular}{cccc}
   \includegraphics[width=33mm]{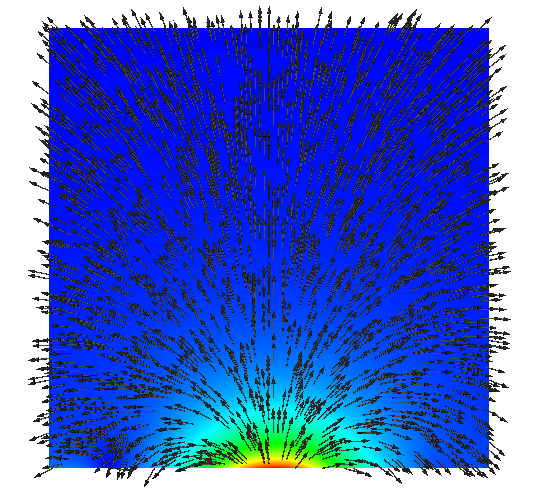}&
   \includegraphics[width=33mm]{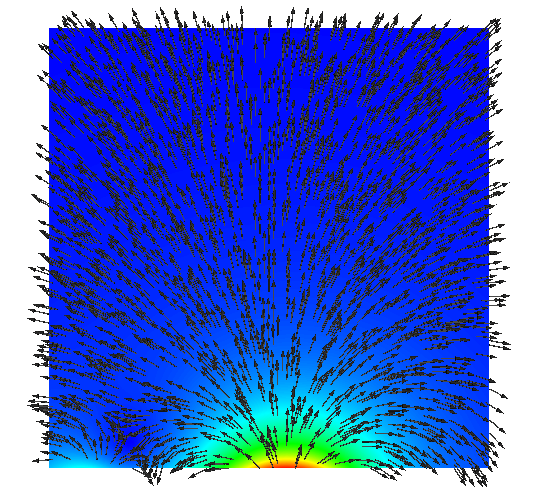}&
   \includegraphics[width=33mm]{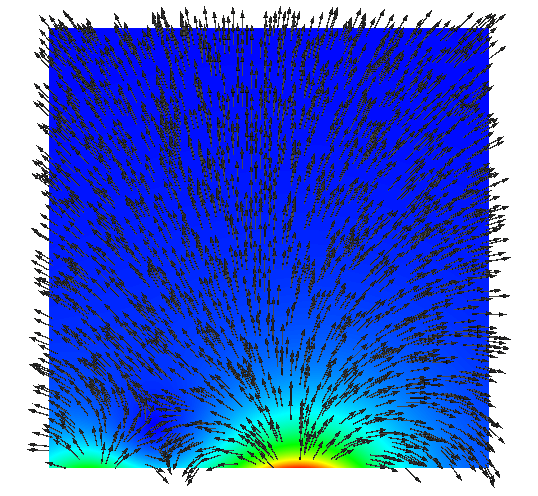}&
   \includegraphics[width=33mm]{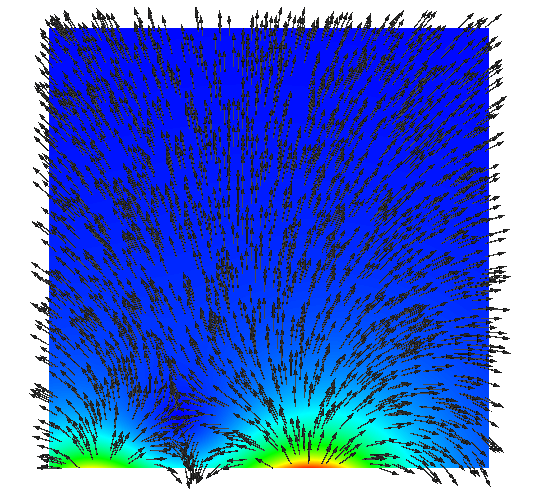}\\
   \includegraphics[width=33mm]{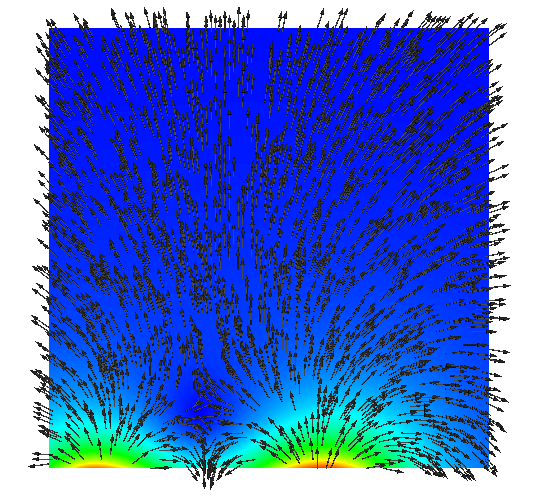}&
   \includegraphics[width=33mm]{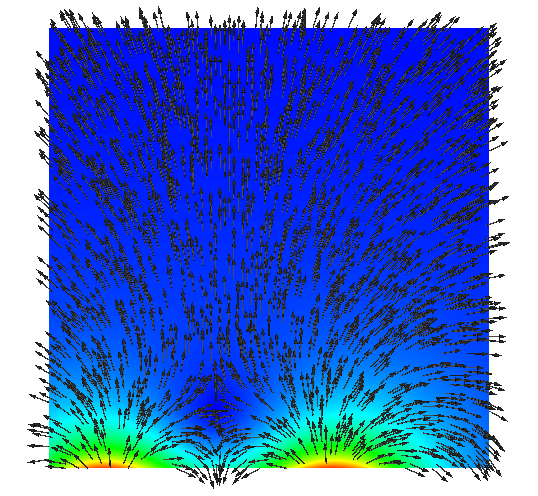}&
   \includegraphics[width=33mm]{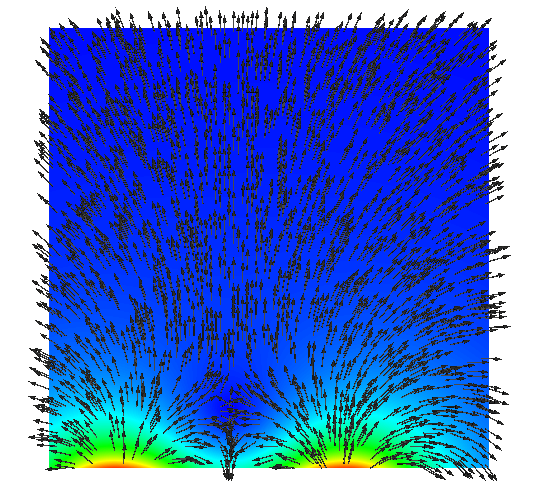}&
   \includegraphics[width=33mm]{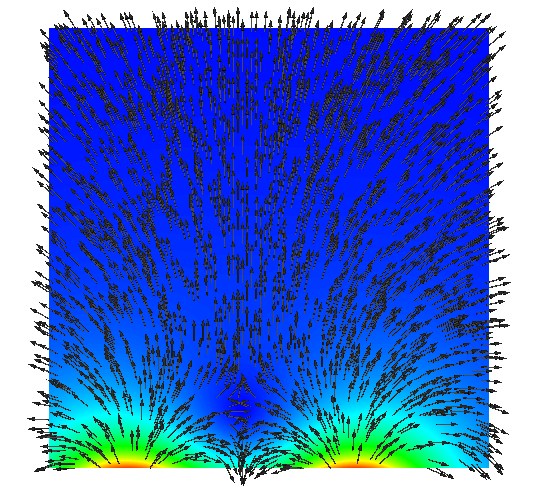}\\
   \includegraphics[width=33mm]{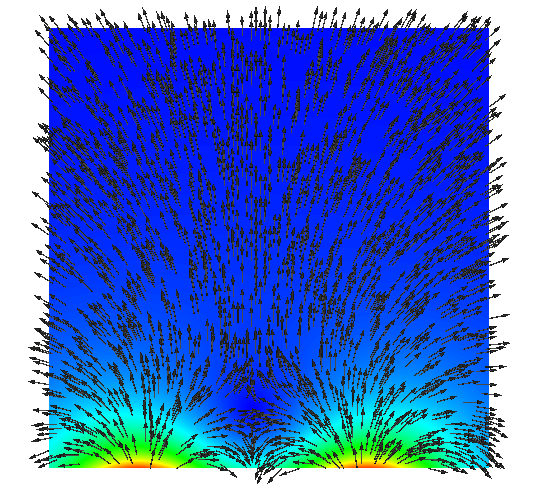}&
   \includegraphics[width=33mm]{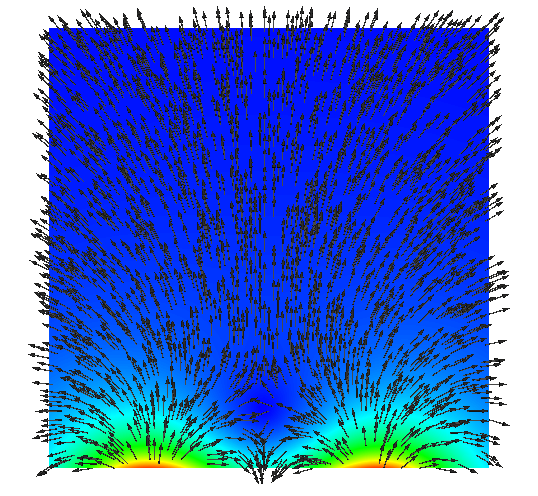}&
   \includegraphics[width=33mm]{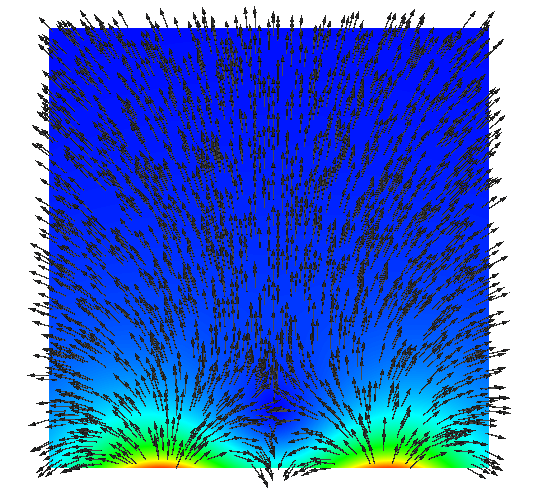}&
   \includegraphics[width=33mm]{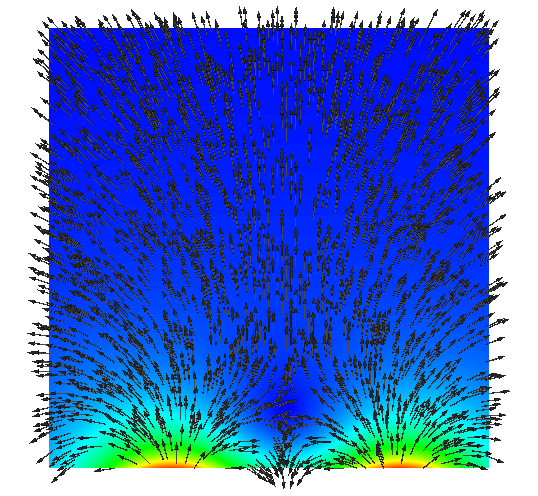}\\
   \includegraphics[width=33mm]{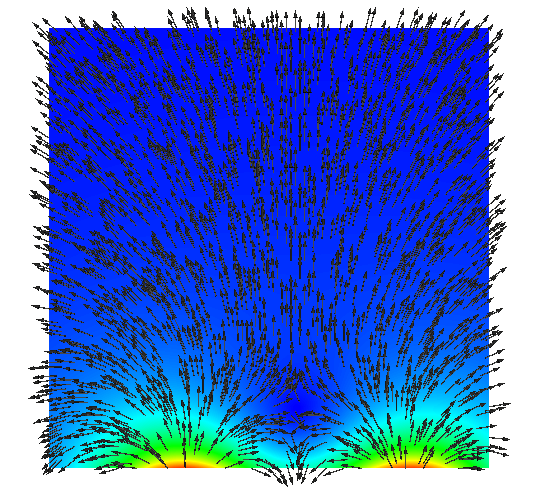}&
   \includegraphics[width=33mm]{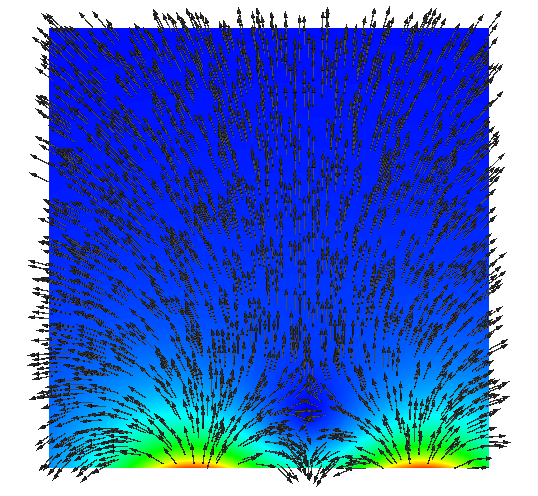}&
   \includegraphics[width=33mm]{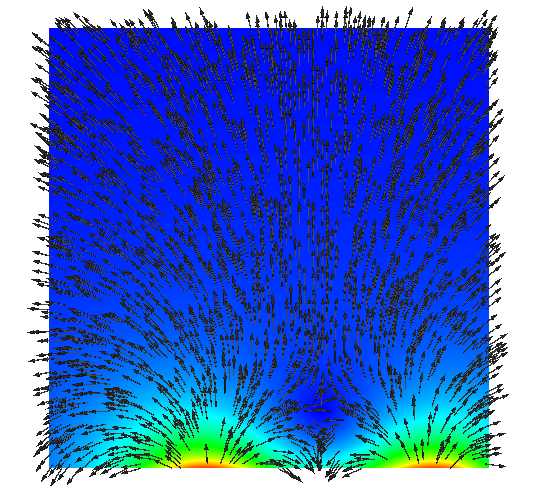}&
   \includegraphics[width=33mm]{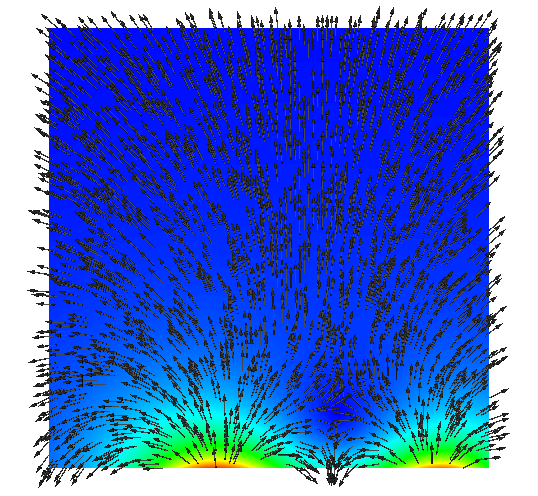}\\
   \includegraphics[width=33mm]{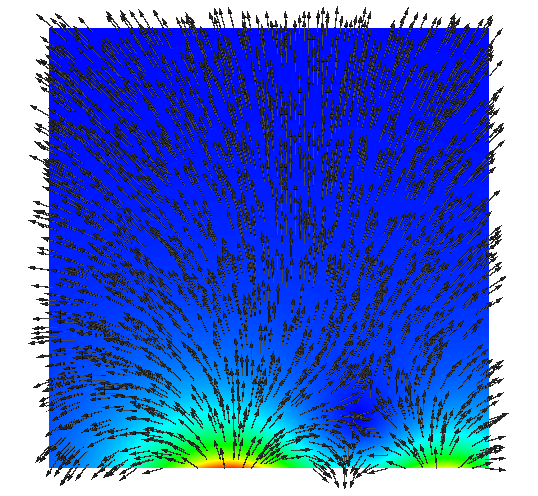}&
   \includegraphics[width=33mm]{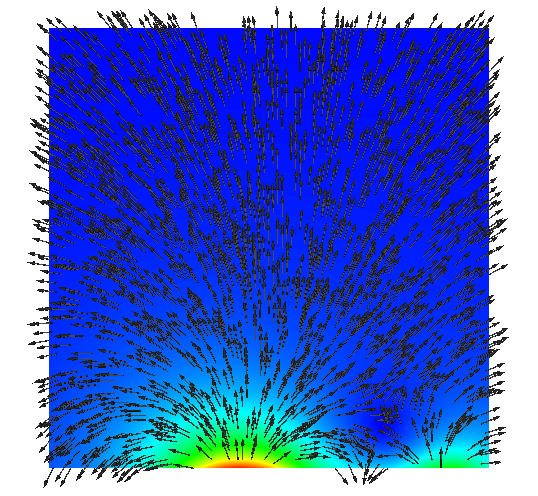}&
   \includegraphics[width=33mm]{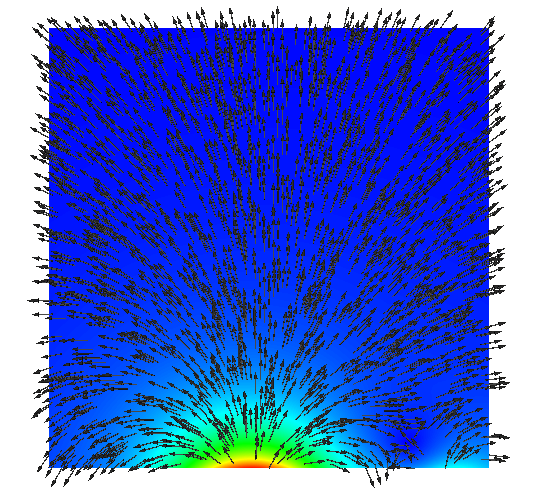}&
   \includegraphics[width=33mm]{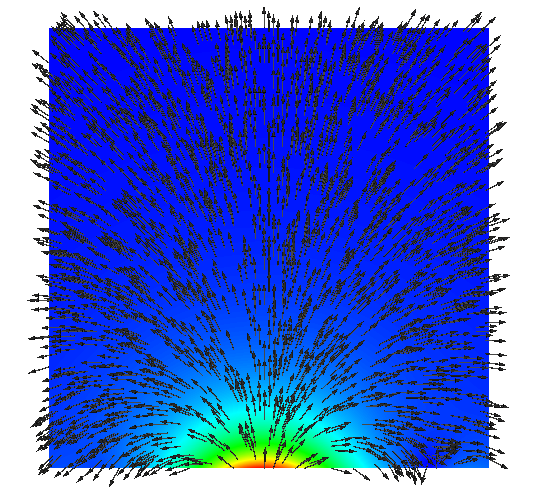}

  \end{tabular}
\end{center}
\caption[Ferromagnetic stirring (second approach): evolution of the magnetization field]{\textbf{Ferromagnetic stirring (second approach): evolution of the magnetization field.} Magnetization vectors and their intensity when using the magnetic field \eqref{travewavemethod}. Here we illustrate a half period of the magnetization profile traveling from left to right. The magnetization is strong in two regions on the lower part of the box. The Kelvin force generated by this magnetization field not only pushes the fluid from left to right, but also creates some effects in the $y$ axis, effectively creating some ripples, which can be appreciated for instance in Figure~\ref{figura15} where streamlines go up and down when they are at the bottom of the box. \label{figura13}} 
\end{figure}

\begin{figure}
\begin{center}
\includegraphics[scale=0.25]{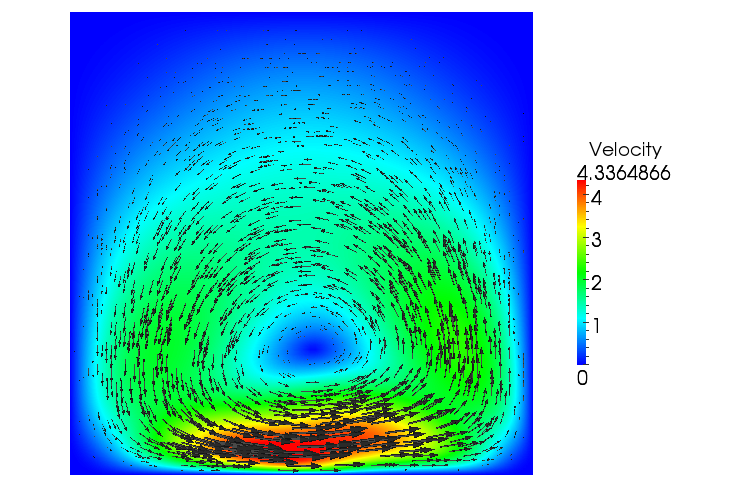}
\end{center}
\caption[Ferromagnetic stirring (second approach): velocity profile]{\textbf{Ferromagnetic stirring (second approach): velocity profile.} Velocity profile induced by the traveling wave of magnetization. This velocity profile gives rise to the evolution of the concentration $c$ in Figure~\ref{figura12}.\label{figura14}}
\end{figure}


\begin{figure}
\begin{center}
  \setlength\fboxsep{0pt}
  \setlength\fboxrule{1pt}
  \begin{tabular}{ccccc}
   \fbox{\includegraphics[width=25mm]{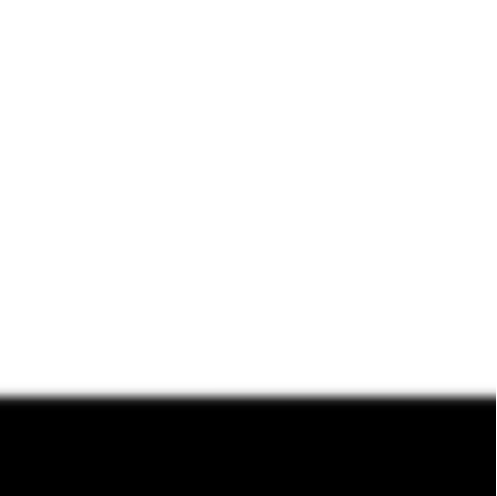}}&
   \fbox{\includegraphics[width=25mm]{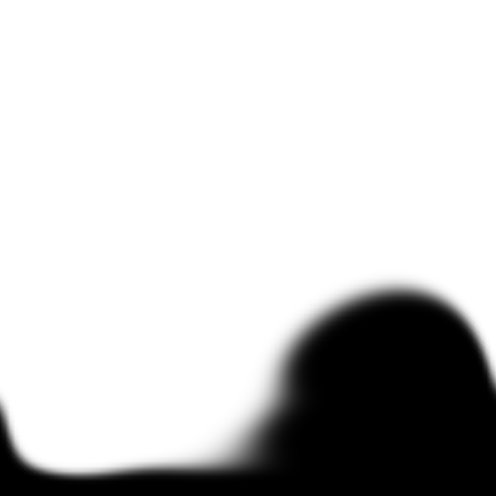}}&
   \fbox{\includegraphics[width=25mm]{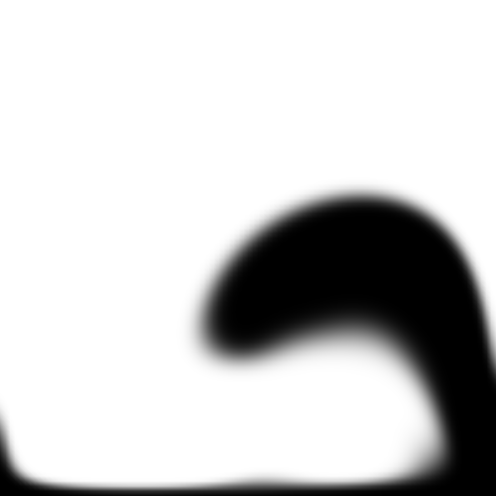}}&
   \fbox{\includegraphics[width=25mm]{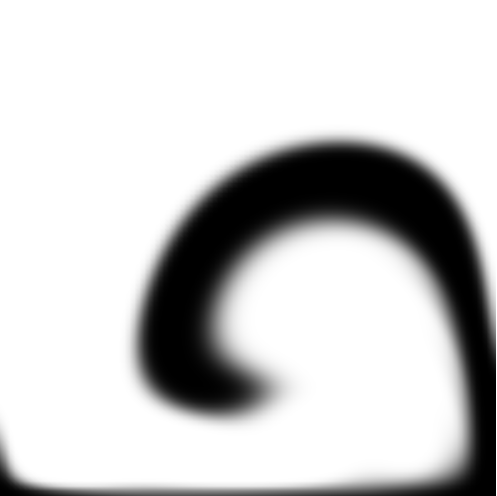}}\\

   \fbox{\includegraphics[width=25mm]{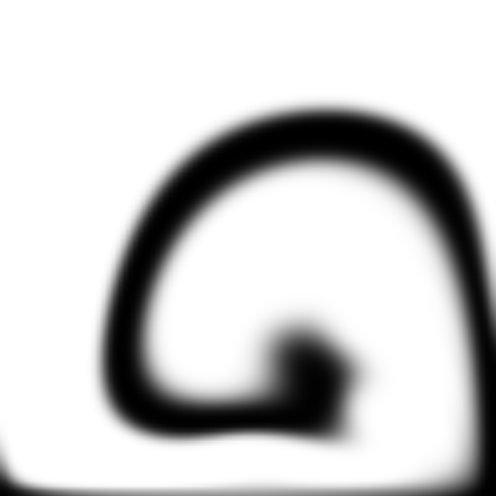}}&
   \fbox{\includegraphics[width=25mm]{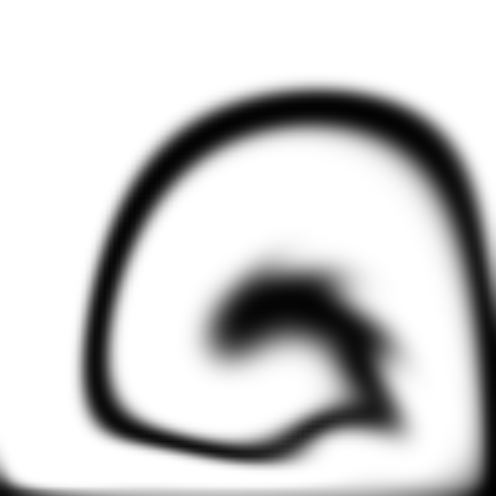}}&
   \fbox{\includegraphics[width=25mm]{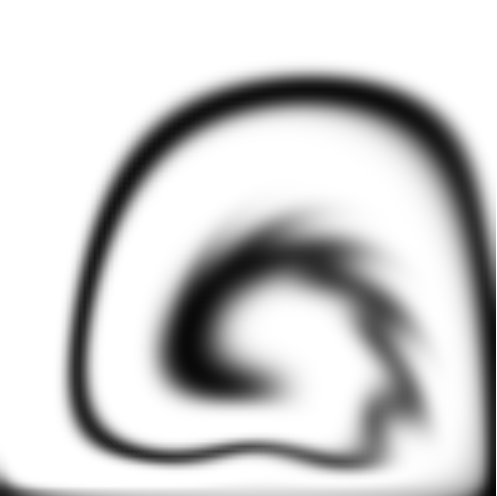}}&
   \fbox{\includegraphics[width=25mm]{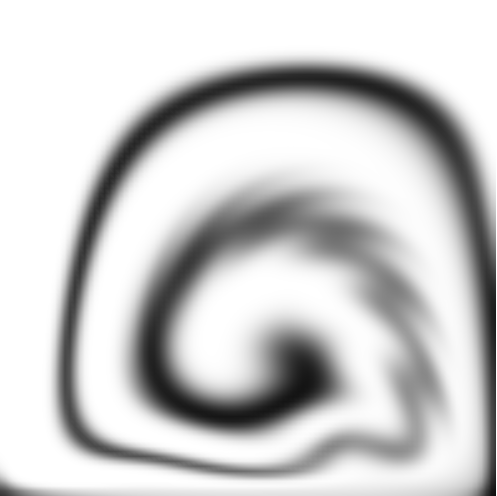}}\\

   \fbox{\includegraphics[width=25mm]{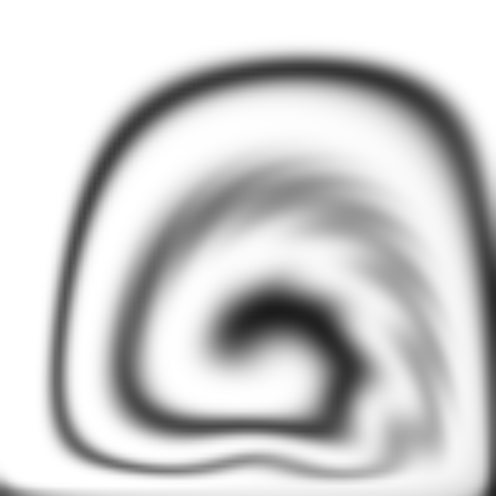}}&
   \fbox{\includegraphics[width=25mm]{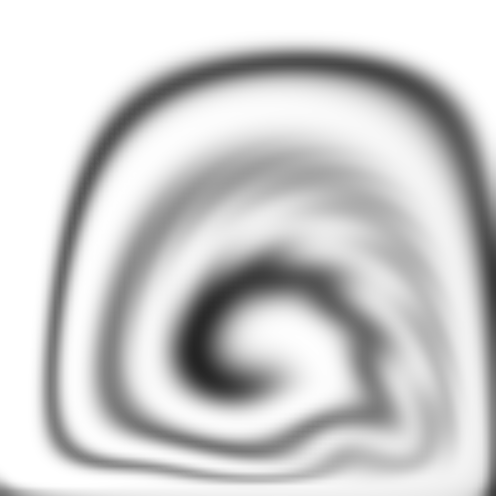}}&
   \fbox{\includegraphics[width=25mm]{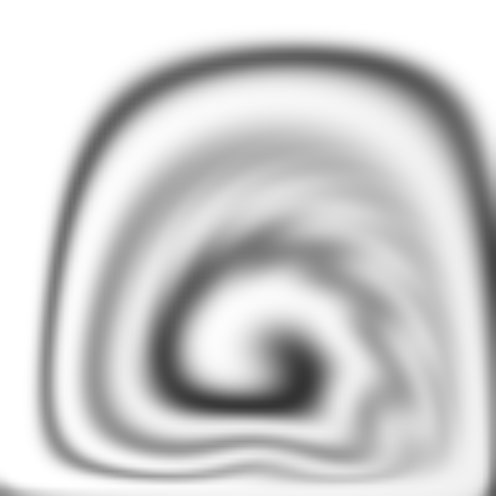}}&
   \fbox{\includegraphics[width=25mm]{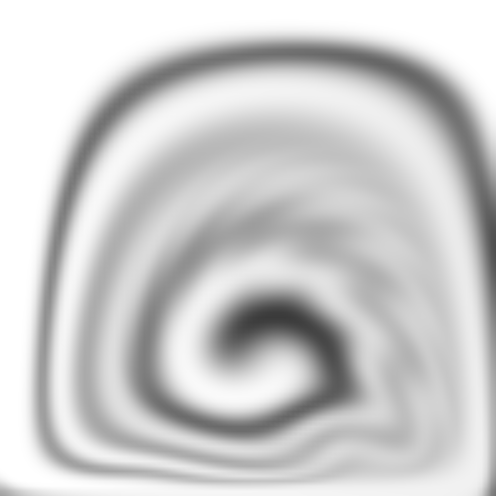}}\\

   \fbox{\includegraphics[width=25mm]{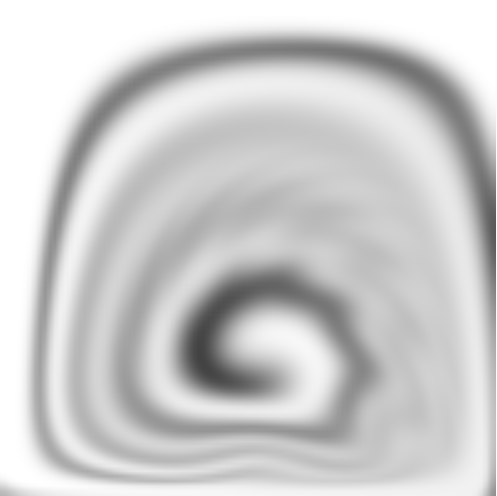}}&
   \fbox{\includegraphics[width=25mm]{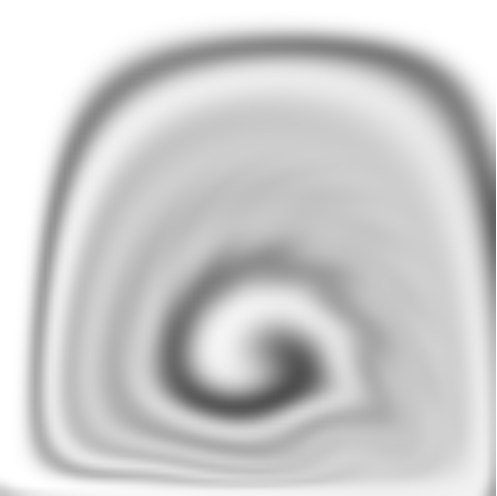}}&
   \fbox{\includegraphics[width=25mm]{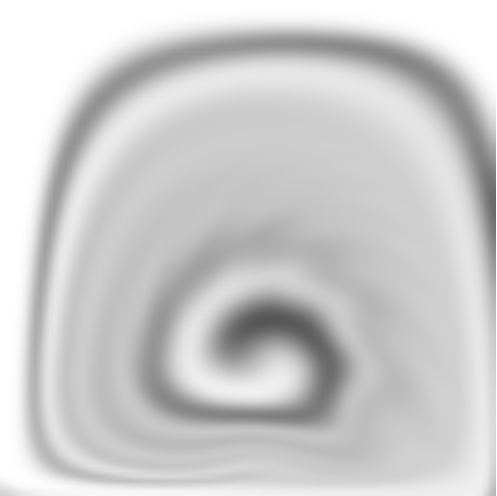}}&
   \fbox{\includegraphics[width=25mm]{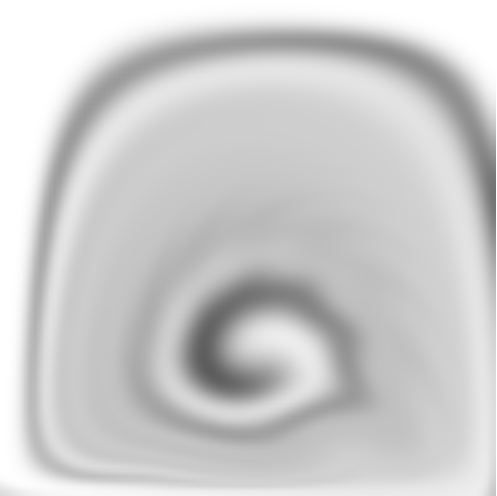}}
  \end{tabular}
\end{center}  
\caption[Ferromagnetic stirring (second approach): evolution of the passive scalar]{\textbf{Ferromagnetic stirring (second approach): evolution of the passive scalar.} Contour of the passive scalar from time $t = 0$ to $t = 4.00s$. The magnetizing field satisfies \eqref{travewavemethod}. These  results where obtained with $f = 40\text{Hz.}$, and $\nu = \nu_r =0.1$. \label{figura15}}
\end{figure}

\section{Conclusions}
\label{sec:conclusions}

In this paper we propose and analyze a numerical scheme for the Rosensweig model of ferrohydrodynamics. The scheme is implicit and we show that, for magnetic diffusion $\magdiff=0$, it is unconditionally stable and that discrete solutions exist. The use of a discontinuous finite element space for the magnetization $\bv{M}$ seems to be mandatory to have a discrete energy law that mimics the continuous one. The scheme delivers the expected convergence rates for smooth solutions. We do show its convergence towards weak solutions under simplifying additional assumptions.

Although not fully understood in the literature, we consider also $\magdiff>0$. The motivation is twofold: adding a regularization to the magnetization equation could be used to obtain global existence of weak solutions. In addition, such regularization enables us to add modeling effects on the boundary. We propose Robin boundary conditions that agree with the tendency of the magnetization to align with the magnetic field. They yield a formal energy estimate which, however, it is hard to reproduce at the discrete level. The main obstruction is the lack of a methodology capable of computing a magnetic field in  $\hcurl \cap \hdiv$ compatible with the energy structure of the system.

We also present some application examples that illustrate the capabilities of the model and the scheme, in particular in the context of ferrofluid pumping and stirring of a passive scalar. These numerical experiments also expose the non-trivial nature of ferrofluids, and how much quantitative tools are needed in order to complement qualitative understanding and experimentation.

Finally, we must comment that many important issues are not discussed. Among them we have to mention how to solve efficiently the system posed by the numerical scheme proposed in this work, modeling of saturation effects (which is a very important physical feature of ferrofluids), and inhomogeneous (e.g. two-phase) ferrofluid flows. We refer to \cite{Tom15} for our first attempt to model and simulate the latter.

\bibliographystyle{siam}
\bibliography{biblio}

\end{document}